\documentclass{amsart}
\title{Triality over Schemes}

\usepackage{amsmath, amssymb, amsthm}
\usepackage{bbm, lmodern, mathrsfs, euscript}
\usepackage{tikz, tikz-cd}
\usepackage{xcolor}
\usepackage{arydshln}
\usepackage{enumitem}
\usepackage[colorlinks=true, pdfstartview=FitV, linkcolor=blue, citecolor=blue, urlcolor=blue, breaklinks=true]{hyperref}
\usepackage{cleveref}

\usepackage{forloop}
\newcounter{fonts}
\let\eeee\edef
\forloop{fonts}{1}{\value{fonts} < 27}{%
\expandafter\eeee\csname \Alph{fonts}\Alph{fonts}\endcsname{\noexpand\mathbb{\Alph{fonts}}}
\expandafter\eeee\csname b\Alph{fonts}\endcsname{\noexpand\mathbf{\Alph{fonts}}}
\expandafter\eeee\csname c\Alph{fonts}\endcsname{\noexpand\mathcal{\Alph{fonts}}}
\expandafter\eeee\csname f\Alph{fonts}\endcsname{\noexpand\mathfrak{\Alph{fonts}}}
\expandafter\eeee\csname fr\alph{fonts}\endcsname{\noexpand\mathfrak{\alph{fonts}}}
\expandafter\eeee\csname s\Alph{fonts}\endcsname{\noexpand\mathscr{\Alph{fonts}}}
\expandafter\eeee\csname e\Alph{fonts}\endcsname{\noexpand\EuScript{\Alph{fonts}}}
}
\mathchardef\mdash="2D

\setlist[enumerate,1]{label={\rm(\roman*)}}

\newtheorem{thm}{Theorem}[section]
\newtheorem{defn}[thm]{Definition}
\newtheorem{prop}[thm]{Proposition}
\newtheorem{cor}[thm]{Corollary}
\newtheorem{lem}[thm]{Lemma}

\newtheoremstyle{remarkstyle}{\topsep}{\topsep}{\rm}{}{\bfseries}{.}{.5em}{}
\theoremstyle{remarkstyle}
\newtheorem{rem}[thm]{Remark}

\newtheorem{ex}[thm]{Example}

\newtheoremstyle{solutionstyle}{\topsep}{\topsep}{\rm}{}{\it}{.}{.5em}{}
\theoremstyle{solutionstyle}

\newtheorem*{sol*}{Solution}

\newcommand{\und}{\underline{\hspace{2ex}}}
\newcommand{\iso}{\overset{\sim}{\longrightarrow}}
\newcommand{\inj}{\hookrightarrow}
\newcommand{\surj}{\twoheadrightarrow}

\newcommand{\runity}{\boldsymbol{\mu}}
\newcommand{\GL}{\mathbf{GL}}

\newcommand{\SO}{\mathbf{SO}}

\newcommand{\PGL}{\mathbf{PGL}}

\newcommand{\PGO}{\mathbf{PGO}}
\newcommand{\Spin}{\mathbf{Spin}}

\newcommand{\slLie}{\mathfrak{sl}}

\DeclareMathOperator{\Grp}{\mathfrak{Grp}}
\DeclareMathOperator{\Ab}{\mathfrak{Ab}}
\DeclareMathOperator{\Sets}{\mathfrak{Sets}}
\DeclareMathOperator{\Rings}{\mathfrak{Rings}}

\DeclareMathOperator{\Sch}{\mathfrak{Sch}}
\DeclareMathOperator{\Aff}{\mathfrak{Aff}}

\newcommand{\cHom}{\mathcal{H}\hspace{-0.4ex}\textit{o\hspace{-0.2ex}m}}
\newcommand{\cEnd}{\mathcal{E}\hspace{-0.4ex}\textit{n\hspace{-0.2ex}d}}
\newcommand{\cAut}{\mathbf{Aut}}
\newcommand{\bAut}{\mathbf{Aut}}
\newcommand{\cSym}{\mathcal{S}\hspace{-0.5ex}\textit{y\hspace{-0.3ex}m}}

\newcommand{\cSkew}{\mathcal{S}\hspace{-0.4ex}\textit{k\hspace{-0.2ex}e\hspace{-0.3ex}w}}
\newcommand{\cAlt}{\mathcal{A}\hspace{-0.1ex}\textit{$\ell$\hspace{-0.3ex}t}}
\newcommand{\cIsom}{\mathcal{I}\hspace{-0.35ex}\textit{s\hspace{-0.15ex}o\hspace{-0.15ex}m}}

\DeclareMathOperator{\Span}{Span}

\DeclareMathOperator{\Img}{Img}
\DeclareMathOperator{\Ker}{Ker}

\DeclareMathOperator{\Hom}{Hom}

\DeclareMathOperator{\Isom}{Isom}
\DeclareMathOperator{\Aut}{Aut}

\DeclareMathOperator{\Id}{Id}

\DeclareMathOperator{\Trd}{Trd}
\DeclareMathOperator{\Nrd}{Nrd}

\DeclareMathOperator{\diag}{diag}

\DeclareMathOperator{\Spec}{Spec}
\DeclareMathOperator{\Inn}{Inn}

\DeclareMathOperator{\Mat}{M}

\newcommand{\Csigma}{\underline{\sigma}}
\newcommand{\Cf}{\underline{f}}

\newcommand{\Cl}{\mathcal{C}\hspace{-0.3ex}\ell}

\newcommand{\fTrip}{\mathfrak{Trip}}
\newcommand{\fTrial}{\mathfrak{TriAlg}}
\newcommand{\fTriAlg}{\mathfrak{TriAlg}}

\newcommand{\qt}{\mathrm{qt}}
\newcommand{\spl}{_\mathrm{split}}

\newcommand{\sw}{\mathrm{sw}}
\newcommand{\oct}{_\mathrm{oct}}

\DeclareMathOperator{\Out}{Out}

\newcommand{\adj}{\mathrm{ad}}
\newcommand{\simpc}{\mathrm{sc}}

\newcommand{\preSpin}{\mathbf{preSpin}}

\begin{document}
\author[C. Ruether]{Cameron Ruether}
\address{The ``Simion Stoilow" Institute of Mathematics of the Romanian Academy, 21 Calea Grivitei Street, 010702 Bucharest, Romania.}
\email{cameronruether@gmail.com}

\thanks{This work was supported by the NSERC grants of Mikhail Kotchetov (Discovery Grant 2018-04883) and Yorck Sommerh\"auser (RGPIN-2017-06543), as well as by the AARMS collaborative research group ``Groups, Rings, Lie and Hopf Algebras" and the Atlantic Algebra Centre. It was also supported by the project ``Group schemes, root systems, and related representations" founded by the European Union - NextGenerationEU through Romania's National Recovery and Resilience Plan (PNRR) call no. PNRR-III-C9-2023-I8, Project CF159/31.07.2023, and coordinated by the Ministry of Research, Innovation and Digitalization (MCID) of Romania}

\date{April 29, 2024}

\maketitle

\noindent{\bf Abstract:} {Working over an arbitrary base scheme, we provide an alternative development of triality which does not use Octonion algebras or symmetric composition algebras. Instead, we use the Clifford algebra of the split hyperbolic quadratic form of rank 8 and computations with Chevalley generators of groups of type $D_4$. Following the strategy of \cite{KMRT}, we then define the stack of trialitarian triples and show it is equivalent to the gerbe of $\PGO_8^+$--torsors. We show it has endomorphisms generating a group isomorphic to $\SS_3$ and that several familiar cohomological properties of $\PGO_8^+$ follow in this setting as a result. Next, we define the stack of trialitarian algebras and show it is equivalent to the gerbe of $\PGO_8^+\rtimes \SS_3$--torsors. Because of this, it is also equivalent to the gerbes of simply connected, respectively adjoint, groups of type $D_4$. We define $\Spin_\cT$ and $\PGO^+_\cT$ for a trialitarian algebra and define concrete functors $\cT \mapsto \Spin_\cT$ and $\cT \mapsto \PGO^+_\cT$ which realize these equivalences.}
\medskip

\noindent{\bf Keywords: {Triality, $D_4$, Quadratic Pairs, Chevalley Groups, Octonion Algebras, Symmetric Composition Algebras, Clifford Algebras}}\\
\medskip
\noindent {\em MSC 2020: Primary 20G35, Secondary 11E57 11E88 16H05 17A75 20G10.}
\bigskip

\section*{Introduction}
{%
\renewcommand{\thethm}{\Alph{thm}}
The standard development and exposition of the phenomenon of triality, that simply connected and adjoint semisimple groups of type $D_4$ have extra outer automorphisms due to the extra symmetries of the Dynkin diagram, is built on the theory of Octonion algebras. This approach dates back to \'Elie Cartan starting in 1925. Later approaches, such as in \cite{KMRT} which works over a field of characteristic not $2$, use symmetric composition algebras, which are twisted forms of para-Cayley algebras. Following \cite{KMRT} but including characteristic $2$, Section 4 of \cite{DQ21} utilizes para-Cayley algebras as well. The triality discussed in \cite{AG19} over rings is also in the context of Octonions. Their definition of $\Spin_{8}$ as triples uses Octonion multiplication explicitly and they also show how their results can be obtained by generalizing the techniques of \cite{KMRT}. We instead present an approach which avoids the use of Octonions or symmetric composition algebras. Instead, we are able to identify and describe the phenomena of triality using the Chevalley generators of split groups of type $D_4$.

We work in the same setting as in \cite{GNR} and \cite{Rue23}; primarily with sheaves on the fppf ringed site $(\Sch_S,\cO)$ over a base scheme $S$ where $\cO$ is the global sections functor. In particular, we make no assumption on the invertibility of $2$. We write sheaves in calligraphic letters and other rings, algebras, etc., in roman letters. In order to work with groups of type $D$ in this setting, we utilize quadratic pairs. See \Cref{quad_forms_and_triples} for an overview of their definition and properties. Given a central simple algebra with quadratic pair $(A,\sigma,f)$ over a field (with some assumptions on the degree), in \cite{DQ21} it was shown how to equip the associated Clifford algebra $(\mathrm{Cl}(A,\sigma,f),\underline{\sigma})$ with its own canonical quadratic pair $(\underline{\sigma},\underline{f})$, where $\underline{\sigma}$ is the usual canonical involution. This was extended to our setting over a scheme in \cite{Rue23} where, for an Azumaya algebra with quadratic pair $(\cA,\sigma,f)$ over $S$, the canonical quadratic pair on the Clifford algebra $\Cl(\cA,\sigma,f)$ was defined. We similarly denote it $(\underline{\sigma},\underline{f})$. With this definition as a starting point, we are now able to describe triality over schemes as well.

Our work here starts with the isomorphism of algebras with quadratic pair \cite[(4)]{Rue23} (which is built on the isomorphisms of algebras from \cite{CF} and \cite{Knus}). Let $(\cO^8,q)$ be the split hyperbolic quadratic form of rank $8$, in detail, $q(x_1,\ldots,x_8) = x_1x_8 + x_2x_7 + x_3 x_6 + x_4 x_5$. Let $(\Mat_8(\cO),\sigma,f)$ be its adjoint quadratic triple. After choosing appropriate bases, the isomorphism in \cite{Rue23} can be written as
\[
\Phi\colon \big(\Cl\big(\Mat_8(\cO),\sigma,f\big),\underline{\sigma},\underline{f}\big) \iso (\Mat_8(\cO),\sigma,f)\times (\Mat_8(\cO),\sigma,f).
\]
By analyzing the images of the Chevalley generators of $\Spin_8 \subset \Cl(\Mat_8(\cO),\sigma,f)$ under $\Phi$, we show how one can identify the phenomenon of triality within the isomorphism $\Phi$. Precisely, we show the following.
\begin{thm}\label{intro_thm_Phi_restricts}
The isomorphism $\Phi$ restricts to a group homomorphism $\Spin_8 \to \bO_8^+ \times \bO_8^+$ which induces the homomorphism
\begin{align*}
\PGO_8^+ &\to \PGO_8^+ \times \PGO_8^+ \\
x &\mapsto \big(\theta^+(x),(\theta^+)^2(x)\big)
\end{align*}
where $\theta^+$ is an order three outer automorphism of $\PGO_8^+$.
\end{thm}
Readers familiar with triality will know that the usual descriptions of the phenomena come from a similar looking isomorphism of algebras
\[
\big(\Cl(\cEnd_\cO(\cS),\sigma_n),\underline{\sigma_n}\big) \iso (\cEnd_\cO(\cS),\sigma_n)\times (\cEnd_\cO(\cS),\sigma_n)
\]
where $(\cS,n)$ is a symmetric composition algebra with its quadratic norm. Here, $\sigma_n$ is the adjoint involution of $n$. The isomorphism above is defined using properties of the non-associative multiplication in $\cS$. We compare our Chevalley generator based approach to this usual approach in \Cref{without_octonions}, showing, as one would hope, that they are equivalent.

Once we have the result of \Cref{intro_thm_Phi_restricts}, we no longer need Chevalley generators or symmetric composition algebras and we proceed down a path analogous to \cite{KMRT}. To shorten notation, we write $\cA_\qt = (\cA,\sigma_\cA,f_\cA)$ for an Azumaya algebra with quadratic pair (qt standing for \emph{quadratic triple}, the term used in \cite{GNR} and \cite{Rue23}). We begin by defining trialitarian triples in the same way as in \cite{KMRT}, \cite{DQ21}, or \cite{BT23}, i.e., they are the data of $(\cA_\qt,\cB_\qt,\cC_\qt,\alpha_\cA)$, where $\cA_\qt$, $\cB_\qt$, $\cC_\qt$ are quadratic triples of degree $8$ and
\[
\alpha_\cA \colon \big(\Cl(\cA_\qt),\underline{\sigma_\cA},\underline{f_\cA}\big)\iso \cB_\qt \times \cC_\qt
\]
is an isomorphism of quadratic triples. By considering quadratic triples over differing base schemes, these form a stack over $\Sch_S$ which we denote by $\fTrip$.
\begin{thm}\label{intro_thm_triples}
The stack $\fTrip$ is equivalent to the gerbe of $\PGO_8^+$--torsors. Furthermore, there are stack endomorphisms
\begin{align*}
\Theta^+ \colon \fTrip &\to \fTrip \\
(\cA_\qt,\cB_\qt,\cC_\qt,\alpha_\cA) &\mapsto (\cB_\qt,\cC_\qt,\cA_\qt,\alpha') \\[2pt]
\Theta \colon \fTrip &\to \fTrip \\
(\cA_\qt,\cB_\qt,\cC_\qt,\alpha_\cA) &\mapsto (\cA_\qt,\cC_\qt,\cB_\qt,\sw \circ \alpha_\cA)
\end{align*}
where $\sw \colon \cB_\qt \times \cC_\qt \iso \cC_\qt \times \cB_\qt$ is the switch isomorphism. These satisfy the relations
\[
(\Theta^+)^3 = \Id, \quad \Theta^2 = \Id, \text{ and } \Theta^+ \circ \Theta = \Theta \circ (\Theta^+)^2.
\]
\end{thm}
The fact that $\fTrip$ is equivalent to the gerbe of $\PGO_8^+$--torsors follows from showing that the automorphism sheaf of a designated split object in $\fTrip$ is $\PGO_8^+$. The morphisms $\Theta^+$ and $\Theta$ are stack level analogues of outer automorphisms $\theta^+$ and $\theta$ of $\PGO_8^+$ which generate $\Out(\PGO_8^+)(T) \cong \SS_3(T)$ over a connected scheme $T$. The existence of these endomorphisms recovers the results of \cite[42.3]{KMRT} or \cite[4.11]{DQ21} in our setting, that if we have Azumaya algebras of degree $8$ with quadratic pairs such that $\big(\Cl(\cA_\qt),\underline{\sigma_\cA},\underline{f_\cA}\big) \cong \cB_\qt\times \cC_\qt$, then also
\begin{align*}
\big(\Cl(\cB_\qt),\underline{\sigma_\cB},\underline{f_\cB}\big) &\cong \cC_\qt\times \cA_\qt, \text{ and} \\
\big(\Cl(\cC_\qt),\underline{\sigma_\cC},\underline{f_\cC}\big) &\cong \cA_\qt\times \cB_\qt.
\end{align*}
Because we have worked on the level of stacks, we obtain as an easy corollary the description of how the outer automorphisms $\theta^+$ and $\theta$ act on the first cohomology set $H^1(S,\PGO_8^+)$, which classifies trialitarian triples up to isomorphism.

The final goal is to mirror the classical correspondence between algebraic groups and algebras with involution for the case of type $D_4$. This is done over fields in \cite[\S 43]{KMRT} by defining trialitarian algebras. In \Cref{Trialitarian_Algebras}, we define trialitarian algebras over a scheme and assemble them into a gerbe, $\fTriAlg$, whose split object has automorphism sheaf $\PGO_8^+\rtimes \SS_3$. This means $\fTriAlg$ is equivalent to the gerbe of $\PGO_8^+\rtimes \SS_3$--torsors. Therefore, it will abstractly be equivalent to the gerbe of simply connected groups of type $D_4$ as well as the gerbe of adjoint groups of type $D_4$. In \Cref{Groups}, for a trialitarian algebra $\cT$, we define the groups $\PGO^+_\cT$ and $\Spin_\cT$ and show that these definitions provide concrete equivalences of gerbes.
\begin{thm}\label{intro_thm_equivalences}
Let $\fD_4^\simpc$ and $\fD_4^\adj$ be the gerbes of simply connected, respectively adjoint, groups of type $D_4$ over $S$. There are equivalences of gerbes
\begin{align*}
\fTriAlg &\to \fD_4^\simpc & \fTriAlg &\to \fD_4^\adj \\
\cT &\mapsto \Spin_\cT & \cT &\mapsto \PGO^+_\cT.
\end{align*}
Furthermore, there are canonical surjections $\Spin_\cT \surj \PGO^+_\cT$, whose kernels are the respective centers, which induce an equivalence of categories $\fD_4^\simpc \to \fD_4^\adj$ which makes the diagram
\[
\begin{tikzcd}[row sep=-1ex]
 & \fD_4^\simpc \ar[dd] \\
\fTriAlg \ar[ur] \ar[dr] & \\
& \fD_4^\adj
\end{tikzcd}
\]
commute.
\end{thm}

This paper is organized as follows. In the preliminaries, we establish our setting over a scheme and review the necessary definitions and background results for key objects such as quadratic triples, the groups $\Spin(\cA,\sigma,f)$, $\bO^+(\cA,\sigma,f)$, and $\PGO^+(\cA,\sigma,f)$, as well as stacks, cohomology, and \'etale covers. Our choices of Chevalley generators for the split groups $\Spin_{2n}$, $\bO^+_{2n}$, and $\PGO^+_{2n}$ are described in \Cref{Chevalley_generators}. Moving into \Cref{without_octonions}, we begin by reviewing the automorphism groups of our split groups in terms of Chevalley generators. \Cref{intro_thm_Phi_restricts} is justified in \Cref{Clifford_Triality} where the necessary computations are the subject of \Cref{Spin_Chevalley_computations}. As mentioned above, \Cref{without_octonions} compares our approach to the usual methods of describing triality.

\Cref{Triples} is where we deal with all things trialitarian triple and justify \Cref{intro_thm_triples}. The fact that $\fTrip$ is equivalent to the gerbe of $\PGO_8^+$--torsors follows from \Cref{Trip_properties}\ref{Trip_properties_ii} where we show that the automorphism sheaf of the designated split object in $\fTrip$ is $\PGO_8^+$. The morphisms $\Theta^+$ and $\Theta$ are constructed explicitly over each fiber throughout \Cref{endo_Trip(S)} and the relations between them are shown in \Cref{endo_Trip_group_homs}.

\Cref{Trialitarian_Algebras} is about trialitarian algebras. In order to define them, we use what we call the trialitarian cover $\Sigma_2(L)$ of a degree $3$ \'etale cover $L\to S$ as well as various morphisms involving $\Sigma_2(L)$, all reviewed in the preliminaries. Using these morphisms, we define trialitarian algebras and the stack $\fTriAlg$. We show in \Cref{TriAlg_automorphism_sheaf} that $\fTriAlg$ is equivalent to the gerbe of $\PGO_8^+ \rtimes \SS_3$--torsors.

Finally, we discuss the simply connected and adjoint groups of type $D_4$ in \Cref{Groups}. Their definitions are given in \Cref{defn_Spin_T} and \Cref{defn_PGO+_T} respectively. \Cref{intro_thm_equivalences} is justified across \Cref{F^ad_equiv}, \Cref{F^sc_equiv}, and \Cref{PGO_T_Spin_T_isogeny}.
}

\section{Preliminaries}
\subsection{Setting over a Scheme}
We work in the same setting as in \cite{Rue23}, which we review here. We fix a base scheme $S$ and primarily work with sheaves over the big fppf ringed site $(\Sch_S,\cO)$ where $\Sch_S$ is as in \cite[Tag 021R]{Stacks} and $\cO$ is the global sections sheaf of \cite[Tag 03DU]{Stacks}. Covers in the fppf topology will be denoted by $\{T_i \to T\}_{i\in I}$ with $T_{ij} = T_i \times_T T_j$. We denote by $\Aff_S$ the big affine site over $S$ as in \cite[Tag 021S (2)]{Stacks}. Affine schemes will usually be denoted by $U$ or $V$. We let
\begin{enumerate}[label={\rm(\roman*)}]
\item $\Sets$ be the category of sets,
\item $\Grp$ be the category of groups,
\item $\Ab$ be the category of abelian groups, and
\item $\Rings$ be the category of unital commutative associative rings.
\end{enumerate}
We use restriction notation as in \cite{Stacks}. In particular, if $\cF$ is a sheaf on $\Sch_S$ and $T\in \Sch_S$, we denote by $\cF|_T$ the sheaf restricted to the subcategory $\Sch_T$. If we have a morphism $T' \to T$ and a section $x\in \cF(T)$, then $x|_{T'}$ denotes the image of $x$ under the restriction map $\cF(T) \to \cF(T')$. A phrase like ``for appropriate sections $x\in \cF$ and $y\in \cG$" means that we consider all $T\in \Sch_S$ and pairs $x\in \cF(T)$ and $y\in \cG(T)$.

Modules and algebras will be over the structure sheaf $\cO$. We use $\cHom_{\cO}(\cM_1,\cM_2)$ to denote the sheaf of internal homomorphisms between two $\cO$--modules. We set $\cEnd_{\cO}(\cM) = \cHom_{\cO}(\cM,\cM)$ and set $\bAut_{\cO}(\cM)$ to be the subsheaf of automorphisms. For an $\cO$--module $\cM$, its \emph{dual} is $\cM^* = \cHom_{\cO}(\cM,\cO)$. Modules can be finite locally free, everywhere positive rank, and/or of constant rank (among other properties) as in \cite[Tag 03DL (2)]{Stacks}. By \cite[Tag 03DN]{Stacks}, it is sufficient to check these conditions over a cover of $S$.

We make extensive use of Azumaya $\cO$--algebras, referring to \cite{Ford} for the background on Azumaya algebras over rings. We use the definition \cite[2.5.3.4]{CF}, stating that $\cA$ is an Azumaya $\cO$--algebra if 
\begin{enumerate}[label={\rm(\roman*)}]
\item For each affine scheme $U\in \Aff_S$, $\cA(U)$ is an Azumaya algebra over the ring $\cO(U)$.
\end{enumerate}
This is equivalent to each of the following two conditions.
\begin{enumerate}[label={\rm(\roman*)}]
\setcounter{enumi}{1}
\item There exists a cover $\{T_i \to S\}_{i\in I}$ such that for each $i\in I$, $\cA|_{T_i} \cong \cEnd_{\cO|_{T_i}}(\cM_i)$ for some $\cO|_{T_i}$--module $\cM_i$ which is finite locally free of everywhere positive rank.
\item There exists a cover $\{T_i \to S\}_{i\in I}$ such that for each $i\in I$, $\cA|_{T_i} \cong \Mat_{n_i}(\cO|_{T_i})$ for some $n_i>0$.
\end{enumerate} 

We use the \emph{reduced norm} $\Nrd_{\cB} \colon \cB \to \cO$ and \emph{reduced trace} $\Trd_{\cB} \colon \cB \to \cO$ of a separable algebra $\cB$. When case the algebra is Azumaya, these are descended versions of the familiar determinant and trace of matrix algebras, respectively. For an Azumaya algebra $\cA$, we set $\slLie_{\cA} = \Ker(\Trd_{\cA})$ to be the submodule of trace zero sections. In case $\cB$ is not Azumaya, but is instead locally a cartesian product of finitely many Azumaya algebras, then $\Trd_{\cB}$ is locally the sum of the traces. Precisely, if $\{T_i \to S\}_{i\in I}$ is a cover such that
\[
\cB|_{T_i} \cong \cA_{i,1}\times \ldots \times \cA_{i,n}
\]
for Azumaya $\cO|_{T_i}$--algebras $\cA_{i,1},\ldots,\cA_{i,n}$, then
\[
\Trd_{\cB}|_{T_i} = \Trd_{\cA_{i,1}} + \ldots + \Trd_{\cA_{i,n}} \colon \cA_{i,1}\times \ldots \times \cA_{i,n} \to \cO|_{T_i}.
\]

\subsection{Quadratic Forms, Quadratic Triples, and Clifford Algebras}\label{quad_forms_and_triples}
We refer to \cite{Knus} for background on quadratic forms over rings. For further background on quadratic triples over schemes we refer to \cite{CF} and \cite{GNR}. 

Given a quadratic form $q\colon \cM \to \cO$ on an $\cO$--module $\cM$, we call $(\cM,q)$ a \emph{quadratic module}. Its polar bilinear form is denoted by $b_q \colon \cM\times \cM \to \cO$. A general bilinear form $b\colon \cM \times \cM \to \cO$ is called \emph{regular} if the map
\begin{align*}
b^* \colon \cM &\to \cM^* \\
m &\mapsto b(m,\und)
\end{align*}
is an isomorphism. A quadratic form $q$ or quadratic module $(\cM,q)$ is called \emph{regular} if $b_q$ is regular. Every regular bilinear form has an associated adjoint involution $\sigma_b$ on the algebra $\cEnd_{\cO}(\cM)$. A regular bilinear form also induces an isomorphism
\begin{align}
\varphi_b \colon \cM\otimes_{\cO}\cM &\iso \cEnd_{\cO}(\cM) \label{eq_bilinear_iso} \\
m_1\otimes m_2 &\mapsto b(m_1,\und)\cdot m_2. \nonumber
\end{align}

By an involution $\sigma$ on an Azumaya $\cO$--algebra $\cA$, we mean an $\cO$--linear anti-automorphism. Over a sufficiently fine cover, any such involution $\sigma$ will be the adjoint involution of some regular bilinear forms. We call $\sigma$
\begin{enumerate}[label={\rm(\roman*)}]
\item \emph{orthogonal} if the bilinear forms are symmetric,
\item \emph{weakly-symplectic} if the bilinear forms are skew-symmetric, and
\item \emph{symplectic} if the bilinear forms are alternating.
\end{enumerate}
Every symplectic involution is weakly-symplectic. If $\frac{1}{2}\in \cO$, then weakly-symplectic and symplectic coincide and they are disjoint from orthogonal. However, in the case when $2=0\in \cO$, the notions of orthogonal and weakly-symplectic coincide. If $\cB$ is a separable $\cO$--algebra, then an involution $\sigma$ on $\cB$ will be called orthogonal if it acts trivially on the center and locally whenever $\cB$ factors as a product of Azumaya algebras, the restriction of $\sigma$ to each Azumaya factor is orthogonal. We call $\sigma$ weakly-symplectic or symplectic similarly.

Given a quadratic form $q\colon \cM \to \cO$, using the definition over a scheme from \cite[4.2]{CF}, we have the \emph{Clifford algebra} \[
\Cl(\cM,q) = \cT(\cM)/\cI
\]
where $\cI$ is the subsheaf of two-sided ideals generated by elements of the form $m\otimes m - q(m)$. It has a canonical involution $\underline{\sigma}$ which reverses the tensor factors and hence is the identity on the submodule isomorphic to $\cM$. By \cite[4.2.0.7]{CF}, if $U\in \Aff_S$, then $\Cl(\cM,q)(U) \cong \mathrm{Cl}(\cM(U),q(U))$ where the right hand side is the Clifford algebra over a ring as in \cite[IV.1.1.2]{Knus}. As usual, we denote the even Clifford algebra generated by products of even length by $\Cl_0(\cM,q)$. The canonical involution restricts to $\underline{\sigma}_0$ on $\Cl_0(\cM,q)$. 

Let $\cM$ be an $\cO$--module and $\sigma \colon \cM \to \cM$ an order two endomorphism. We have two maps $\Id \pm \sigma \colon \cA \to \cA$ and we define the submodules
\[
\cSym_{\cM,\sigma} = \Ker(\Id - \sigma)\text{, } \cSkew_{\cM,\sigma} = \Ker(\Id + \sigma)\text{, and }  \cAlt_{\cM,\sigma} = \Img(\Id - \sigma) 
\]
where $\cAlt_{\cM,\sigma}$ is an image sheaf and therefore involves sheafification.

A \emph{quadratic triple} is $(\cA,\sigma,f)$ where $(\cA,\sigma)$ is an Azumaya $\cO$--algebra with orthogonal involution and $f\colon \cSym_{\cA,\sigma} \to \cO$ is an $\cO$--linear map satisfying
\[
f(a+\sigma(a)) = \Trd_{\cA}(a)
\]
for all $a\in \cA$. We call the function $f$ a \emph{semi-trace} and note that $(\sigma,f)$ is commonly referred to as a \emph{quadratic pair} on $\cA$. A morphism of quadratic triples $\varphi \colon (\cA,\sigma,f) \to (\cA',\sigma',f')$ is a morphism of the Azumaya algebras satisfying $\varphi \circ \sigma = \sigma'\circ \varphi$ and $f'\circ \varphi = f$.

If $(\cM,q)$ is a regular quadratic module, then by \cite[4.4]{GNR} it corresponds to a unique quadratic triple $(\cEnd_{\cO}(\cM),\sigma_q,f_q)$ where $\sigma_q$ is the adjoint involution of $b_q$ and the semi-trace satisfies
\[
f_q(\varphi_b(m\otimes m))=q(m)
\]
for all $m\in \cM$. The triple $(\cEnd_{\cO}(\cM),\sigma_q,f_q)$ is called the \emph{adjoint} of $(\cM,q)$.

By \cite[3.6(iii)]{GNR}, if $\sigma$ is an orthogonal involution, then the submodules $\cSym_{\cA,\sigma}$ and $\cAlt_{\cA,\sigma}$ are mutually orthogonal with respect to the trace form on $\cA$. Therefore, any element $\lambda \in \cA/\cAlt_{\cA,\sigma}(S)$ gives a well defined map $g_\lambda \colon \cSym_{\cA,\sigma} \to \cO$ which is defined locally over those $T\in \Sch_S$ for which there exists $\ell_\lambda \in \cA(T)$ with $\ell_\lambda\mapsto \lambda|
_T$ by $\Trd_{\cA}(\ell_\lambda \und)$. Not all such $g_\lambda$ are semi-traces, but by \cite[4.18]{GNR}, given a quadratic triple $(\cA,\sigma,f)$ there exists a unique element $\lambda_f\in \cA/\cAlt_{\cA,\sigma}(S)$ such that $f = g_{\lambda_f}$.

Let $(\cB,\sigma)$ be a separable $\cO$--algebra with orthogonal involution. A linear map $f\colon \cSym_{\cB,\sigma} \to \cO$ satisfying $f(b+\sigma(b)) = \Trd_{\cB}(b)$ for all $b\in \cB$ will also be called a \emph{semi-trace}. The triple $(\cB,\sigma,f)$ will be called an \emph{algebra with semi-trace} to distinguish it from a quadratic triple, which assumes that $\cB$ is Azumaya. An algebra with semi-trace will locally be a cartesian product of quadratic triples. The process of defining a semi-trace using an element $\lambda \in \cB/\cAlt_{\cB,\sigma}$ works analogously by working sufficiently locally. We note that our terminology differs slightly from \cite{DQ21} where $(\sigma,f)$ would also be called a quadratic pair on $\cB$.

Associated to a quadratic triple $(\cA,\sigma,f)$ is a Clifford algebra $(\Cl(\cA,\sigma,f),\underline{\sigma})$ with canonical involution defined in \cite[4.2.0.13]{CF} or equivalently in \cite[1.3]{Rue23}. It comes with a canonical linear map $c\colon \cA \to \Cl(\cA,\sigma,f)$ whose image generates $\Cl(\cA,\sigma,f)$ as an algebra. The canonical involution is defined by $\underline{\sigma}(c(a))=c(\sigma(a))$. By \cite[4.2.0.14]{CF}, if $(\cM,q)$ is a regular quadratic module of constant even rank, then there is an isomorphism of algebras with involution
\begin{equation}\label{eq_Clifford_iso}
\Phi \colon (\Cl_0(\cM,q),\underline{\sigma}_0) \iso (\Cl(\cEnd_{\cO}(\cM),\sigma_q,f_q),\underline{\sigma_q})
\end{equation}
induced by the isomorphism of modules $\varphi_b \colon \cM\otimes_{\cO} \cM \iso \cEnd_{\cO}(\cM)$ of \Cref{eq_bilinear_iso}. If $\cA$ is of degree $2n$, then by \cite[1.9]{Rue23}, $\Cl(\cA,\sigma,f)$ is locally a Cartesian product of two Azumaya algebras and the involution $\underline{\sigma}$
\begin{enumerate}[label={\rm(\roman*)}]
\item does not act trivially on the center if $n$ is odd,
\item is orthogonal if $n\equiv 0 \pmod{4}$, and
\item is symplectic if $n\equiv 2 \pmod{4}$.
\end{enumerate}
In particular, due to our conventions, $\underline{\sigma}$ is an orthogonal involution if and only if $n\equiv 0 \pmod{4}$ or both $n \equiv 2 \pmod{4}$ and $2=0\in \cO$. In the cases where $\deg(\cA)\geq 8$ and $\underline{\sigma}$ is orthogonal, then the Clifford algebra also has a canonical semi-trace $\underline{f}$ defined in \cite[2.7]{Rue23}. There is a commutative diagram with exact rows
\[
\begin{tikzcd}[column sep=3ex]
0 \arrow{r} & \slLie_{\cA} \arrow[hookrightarrow]{r} \arrow{d}{c} & \cA \arrow{r}{\Trd_{\cA}} \arrow{d}{c} & \cO \arrow{d}{\rho} \arrow{r} & 0 \\
0 \arrow{r} & \cAlt_{\Cl(\cA,\sigma,f),\underline{\sigma}} \arrow[hookrightarrow]{r} & \Cl(\cA,\sigma,f) \arrow{r} & \Cl(\cA,\sigma,f)/\cAlt_{\Cl(\cA,\sigma,f),\underline{\sigma}} \arrow{r} & 0
\end{tikzcd}
\]
and the canonical semi-trace $\underline{f}$ is defined to be the linear map $g_{\rho(1)} \colon \cSym_{\Cl(\cA,\sigma,f),\underline{\sigma}} \to \cO$ associated to $\rho(1) \in (\Cl(\cA,\sigma,f)/\cAlt_{\Cl(\cA,\sigma,f),\underline{\sigma}})(S)$. By \cite[2.12]{Rue23}, this construction extends the usual functoriality of the Clifford algebra with respect to isomorphisms. Given an isomorphism of quadratic triples $\varphi \colon (\cA,\sigma,f) \to (\cA',\sigma',f')$, there is an associated isomorphism of Clifford algebras with semi-trace
\begin{align}
\Cl(\varphi)\colon (\Cl(\cA,\sigma,f),\underline{\sigma},\underline{f}) &\to (\Cl(\cA',\sigma',f'),\underline{\sigma'},\underline{f'}) \label{eq_Clifford_functorial} \\
c(a) &\mapsto c(\varphi(a)). \nonumber
\end{align}

We will frequently use the Clifford algebra of a hyperbolic quadratic form for calculations. We recall from \cite[1.5]{Rue23} the relevant details of such a Clifford algebra and fix our conventions.
\begin{ex}\label{split_Clifford}
We let $\cV = \cO^n$ be the free $\cO$--module of rank $n$, and set $\HH(\cV) = \cV \oplus \cV^*$. Let $\{v_1,\ldots,v_n\}$ be the standard basis of $\cV$ and $\{v_n^*,\ldots,v_1^*\}$ the dual basis of $\cV^*$. We define a hyperbolic quadratic form on $\HH(\cV)$ by setting $q_{2n}(x+g)=g(x)$ for $x\in \cV$ and $g\in \cV^*$. This is a regular quadratic form. The Clifford algebra $\Cl(\HH(\cV),q_{2n})$ is generated by the elements $v_i$ and $v_j^*$ subject to the relations
\begin{align*}
v_i^2 &= 0 & (v_i^*)^2 &= 0 \\
v_iv_j &= -v_jv_i & v_i^* v_j^* &= -v_j^* v_i^* \\
v_iv_i^* &= 1-v_i^*v_i & v_iv_j^* &= -v_j^*v_i
\end{align*}
for all $1\leq i,j \leq n$ with $i\neq j$.

Let $\wedge \cV$ be the exterior algebra of $\cV$ and let $\wedge_0 \cV$ and $\wedge_1 \cV$ be the even and odd parts respectively. We define two families of endomorphisms of $\wedge\cV$. For $x\in \cV$, let $\ell_x \colon \wedge \cV \to \wedge\cV$ be the left multiplication by $x$. For $g\in \cV^*$, define $d_g \colon \wedge \cV \to \wedge \cV$ by
\[
d_g(x_1\wedge \ldots \wedge x_k) = \sum_{i=1}^k (-1)^{i+1}g(x_i)\cdot x_1\wedge\ldots\wedge \hat{x_i} \wedge \ldots \wedge x_k
\]  
where hat denotes omission. Then by \cite[4.2.0.10]{CF} (which follows from \cite[IV.2.1.1]{Knus}), there is an isomorphism of algebras
\begin{align*}
\Phi \colon \Cl(\HH(\cV),q_{2n}) &\iso \cEnd_{\cO}(\wedge \cV) \\
v_i &\mapsto \ell_{v_i} \\
v_i^* &\mapsto d_{v_i^*}
\end{align*}
which restricts to an isomorphism
\[
\Phi_0 \colon \Cl_0(\HH(\cV),q_{2n}) \iso \cEnd_{\cO}(\wedge_0 \cV)\times \cEnd_{\cO}(\wedge_1 \cV).
\]
The $\cO$--module $\wedge \cV$ is equipped with a bilinear form following \cite[pg.90]{KMRT}. Let $[n]=\{1,\ldots,n\}$ be the set of integers from $1$ to $n$, and then for $I=\{i_1,\ldots,i_k\}\subseteq [n]$ written in increasing order, we set $v_I = v_{i_1}\wedge v_{i_2} \wedge \ldots \wedge v_{i_k}$. These $v_I$ form a basis of $\wedge\cV$. Now, for $x\in \wedge\cV$ let $\overline{x}$ be the element with the wedge factors reversed, i.e., $\overline{x_1\wedge x_2 \wedge \ldots \wedge x_k}=x_k\wedge \ldots \wedge x_2 \wedge x_1$ and extended linearly. Let $\pi \colon \wedge\cV \to \cO$ be the map coming from the projection onto the submodule spanned by $e_{[n]}$. Define a bilinear form $b_\wedge \colon \wedge \cV \times \wedge\cV \to \cO$ by
\[
b_{\wedge}(x,y) = \pi(\overline{x}\wedge y).
\]
By \cite[1.7]{Rue23}, $b_\wedge$ is a regular bilinear form whose adjoint involution $\sigma_\wedge$ corresponds to $\underline{\sigma}$ under the isomorphism $\Phi$. When $n\equiv 0,1\pmod{4}$ or $2=0\in\cO$ then $b_\wedge$ is the polar of a quadratic form $q_\wedge$ defined in \cite[Eq (6)]{Rue23}. 

Let $(\cEnd_{\cO}(\HH(\cV),\sigma_{2n},f_{2n})$ be the adjoint quadratic triple of $(\HH(\cV),q_{2n})$. We identify $\Cl_0(\HH(\cV),q_{2n}) = \Cl(\cEnd_{\cO}(\HH(\cV),\sigma_{2n},f_{2n})$. When $n\geq 4$ and $\underline{\sigma_{2n}}_0$ is orthogonal, i.e. when $n\equiv 0 \pmod{4}$ or both $n\equiv 2\pmod{4}$ and $2=0\in \cO$, we have restricted bilinear forms $b_{\wedge i}$ and quadratic forms $q_{\wedge i}$ on each $\wedge_i \cV$, as well as restricted orthogonal involutions $\sigma_{\wedge i}$ on each $\cEnd_{\cO}(\wedge_i \cV)$. By \cite[2.11]{Rue23}, the isomorphism $\Phi_0$ then becomes an isomorphism of Clifford algebras with canonical semi-trace \refstepcounter{equation}\label{eq_Phi_not}
\[
\textnormal{(\theequation) }\Phi_0 \colon \! (\Cl_0(\HH(\cV),q_{2n}),\underline{\sigma_{2n}}_0,\underline{f_{2n}})\! \iso\! (\cEnd_{\cO}(\wedge_0 \cV),\sigma_{\wedge 0},f_{\wedge 0})\!\times\! (\cEnd_{\cO}(\wedge_1 \cV),\sigma_{\wedge 1},f_{\wedge 1}).
\]
\end{ex}

\subsection{Algebraic Groups}\label{algebraic_groups}
We are interested in linear algebraic groups of type $D$, which we consider as sheaves of groups on $\Sch_S$. For the split groups arising from a hyperbolic quadratic form, we will also work with their subpresheaf of elementary groups generated by their Chevalley generators. 

\subsubsection{Relevant Definitions}\label{revelant_definitions}
We recall the definitions of some key groups from \cite{CF} and fix our notation.

If $\cB$ is an $\cO$--algebra, the \emph{general linear group} of $\cB$
\begin{align*}
\GL_{\cB} \colon \Sch_S &\to \Grp \\
T &\mapsto \cB(T)^\times .
\end{align*}
is the subsheaf of invertible elements in $\cB$. The \emph{projective general linear group} is the group of algebra automorphisms $\PGL_{\cB} = \cAut_{\cO\mathrm{\mdash alg}}(\cB)$. If $\cM$ is an $\cO$--module, we set $\GL_{\cM} = \GL_{\cEnd_{\cO}(\cM)} = \cAut_{\cO\mathrm{\mdash mod}}(\cM)$ and $\PGL_{\cM} = \PGL_{\cEnd_{\cO}(\cM)}$. We set $\GL_n = \GL_{\cO^n} \cong \GL_{\Mat_n(\cO)}$ and $\PGL_n$ similarly.

The group of $n^{\textnormal{th}}$--roots of unity is
\begin{align*}
\runity_n \colon \Sch_S &\to \Grp \\
T &\mapsto \{x \in \cO(T) \mid x^n = 1\}.
\end{align*}

Let $(\cM,q)$ be a regular quadratic module. The \emph{orthogonal group} is the group of automorphisms which fix $q$,
\begin{align*}
\bO_q \colon \Sch_S &\to \Grp \\
T &\mapsto \{\varphi \in \GL_{\cM}(T) \mid q|_T \circ \varphi = q|_T\}.
\end{align*}
When $(\cM,q)=(\HH(\cV),q_{2n})$ is the even rank hyperbolic quadratic form of \Cref{split_Clifford}, we write $\bO_{q_{2n}} = \bO_{2n}$. The \emph{spin group} of $(\cM,q)$ is
\begin{align*}
\Spin_q \colon \Sch_S &\to \Grp \\
T &\mapsto \{a \in \Cl_0(\cM,q)(T) \mid a\underline{\sigma}(a)=1,\; \Inn(a)(\cM|_T) = \cM|_T\}
\end{align*}
where $\Inn(a)$ is the inner automorphism of $\Cl(\cM,q)|_T$ defined by the section $a\in \Cl_0(\cM,q)(T)$, and as usual we view $\cM|_T$ as a submodule of $\Cl(\cM,q)|_T$ via the canonical inclusion. For $(\HH(\cV),q_{2n})$ of \Cref{split_Clifford}, we write $\Spin_{2n}$. The spin group comes with a surjective map of sheaves
\begin{align*}
\chi \colon \Spin_q &\to \bO_q \\
a &\mapsto \Inn(a)|_{\cM},
\end{align*}
called the \emph{vector representation}, whose kernel is isomorphic to $\runity_2$.

Let $(\cA,\sigma,f)$ be a quadratic triple. The \emph{orthogonal group} of $(\cA,\sigma,f)$ is
\begin{align*}
\bO_{(\cA,\sigma,f)} \colon \Sch_S &\to \Grp \\
T &\mapsto \{ a \in \cA(T) \mid a\sigma(a)=1,\; f|_T \circ \Inn(a) = f|_T \}
\end{align*}
where $\Inn(a) \colon \cA|_T \to \cA|_T$ is the inner automorphism defined by the section $a \in \cA(T)$. By \cite[4.4.0.44]{CF}, for a regular quadratic module $(\cM,q)$ there is a canonical isomorphism $\bO_q \cong \bO_{(\cEnd_{\cO}(\cM),\sigma_q,f_q)}$. The \emph{projective orthogonal group} is the group of quadratic triple automorphisms
\[
\PGO_{(\cA,\sigma,f)} = \cAut(\cA,\sigma,f).
\]
We set $\PGO_{q} = \PGO_{(\cEnd_{\cO}(\cM),\sigma_q,f_q)}$. For $(\HH(\cV),q_{2n})$ of \Cref{split_Clifford}, we write $\PGO_{2n}$. The canonical projection
\begin{align*}
\pi_{\bO}\colon \bO_{(\cA,\sigma,f)} &\surj \PGO_{(\cA,\sigma,f)} \\
a &\mapsto \Inn(a)
\end{align*}
is a surjective map of sheaves with kernel isomorphic to $\runity_2$.

The spin group of a quadratic triple is defined over a field in \cite{KMRT} using the Clifford bimodule and the generalization of the same approach is used in \cite{CF}, though the authors omit the technicalities. By \cite[4.5.2.3]{CF}, combined with their definition \cite[4.5.2.5]{CF}, they show that for the adjoint quadratic triple of a regular quadratic form $(\cM,q)$, there is an isomorphism
\[
\Spin_{q} \iso \Spin_{(\cEnd_{\cO}(\cM),\sigma_q,f_q)}
\]
which is the restriction of the isomorphism $\Phi$ of \Cref{eq_Clifford_iso}. This allows us to also avoid the technicalities of the Clifford bimodule by using the following descent flavoured definition.
\begin{lem}\label{nonsplit_Spin}
Let $(\cA,\sigma,f)$ be a quadratic triple. For each $T \in \Sch_S$, let $\cU_T$ be the class of $T$--schemes over which $(\cA,\sigma,f)$ becomes adjoint to a regular quadratic form, i.e., for all $T'\in \cU_T$ we have
\begin{align*}
(\cA,\sigma,f)|_{T'} &\cong (\cEnd_{\cO|_{T'}}(\cM'),\sigma_{q'},f_{q'}), \text{ and}\\
\Cl(\cA,\sigma,f)|_{T'} &\cong \Cl_0(\cM',q')
\end{align*}
for a regular quadratic $\cO|_{T'}$--module $(\cM',q')$. Then, the spin group of $(\cA,\sigma,f)$ can be described as
\begin{align*}
\Spin_{(\cA,\sigma,f)} \colon \Sch_S &\to \Grp \\
T &\mapsto \{a\in \Cl(\cA,\sigma,f) \mid \Phi_{T'}^{-1}(a|_{T'}) \in \Spin_{q'}(T') \text{ for all } T'\in \cU_T \}.
\end{align*}
where $\Phi_{T'}$ is the isomorphism of \Cref{eq_Clifford_iso}.
\end{lem}
\begin{proof}
Let $\cF$ be the sheaf
\[
T \mapsto \{a\in \Cl(\cA,\sigma,f) \mid \Phi_{T'}^{-1}(a|_{T'}) \in \Spin_{q'}(T') \text{ for all } T'\in \cU_T \}.
\]
For a scheme $T\in \Sch_S$, if $T\in \cU_T$, i.e., $(\cA,\sigma,f)|_T$ is adjoint to a regular quadratic $\cO|_T$--module $(\cM',q')$, then it is clear that $\Phi_T^{-1}(\cF(T))=\Spin_{q'}(T)$, or equivalently, $\cF(T)=\Spin_{(\cA,\sigma,f)}(T)$. Since there exists an fppf cover $\{T_i \to S\}_{i\in I}$ such that $(\cA,\sigma,f)|_{T_i}$ is adjoint to a regular quadratic module for each $i\in I$, and since $\cF$ and $\Spin_{(\cA,\sigma,f)}$ share the same restriction maps because they are both subsheaves of $\Cl(\cA,\sigma,f)$, this means that $\cF = \Spin_{(\cA,\sigma,f)}$ as claimed.
\end{proof}
These spin groups also come with a homomorphism $\chi \colon \Spin_{(\cA,\sigma,f)} \to \bO_{(\cA,\sigma,f)}$ which is locally the vector representation defined above.

Let $(\cM,q)$ be a regular quadratic module. We have group homomorphisms
\begin{align}
\bC_{\Inn} \colon \Spin_q &\to \cAut(\Cl(\cM,q),\underline{\sigma}), \nonumber \\
a &\mapsto \Inn(a) \nonumber \\
\bC \colon \bO_q &\to \cAut(\Cl(\cM,q),\underline{\sigma}), \label{eq_Clifford_actions} \\
B &\mapsto \big(m\mapsto B(m)\big) \nonumber \\
\Cl \colon \PGO_q &\to \cAut(\Cl_0(\cM,q),\underline{\sigma}_0) \nonumber \\
\psi &\mapsto \big(m_1\otimes m_2 \mapsto (\varphi_{b_q}^{-1}\circ \psi \circ \varphi_{b_q})(m_1\otimes m_2)\big),\nonumber
\end{align}
where $\bC(B)$ and $\Cl(\psi)$ are defined uniquely by their action on generators given above, where $\varphi_{b_q}$ is the isomorphism of \Cref{eq_bilinear_iso}. Of course, for $\psi \in \PGO_q$, its image $\Cl(\psi)$ is defined by the functoriality of the Clifford algebra after accounting for the isomorphism $\Phi$ of \Cref{eq_Clifford_iso}, which is itself induced by $\varphi_{b_q}$. Note that the images of both $\bC_{\Inn}$ and $\bC$ also stabilize the even Clifford algebra, and so by abuse of notation we will also discuss homomorphisms
\begin{align*}
\bC_{\Inn} \colon \Spin_q &\to \cAut(\Cl_0(\cM,q),\underline{\sigma}_0), \text{ and} \\
\bC \colon \bO_q &\to \cAut(\Cl_0(\cM,q),\underline{\sigma}_0).
\end{align*}
We then get commutative diagrams
\begin{equation}\label{action_on_Clifford}
\begin{tikzcd}
\Spin_q \arrow{dr}{\bC_{\Inn}} \arrow{d}{\chi} & \\
\bO_q \arrow{d}{\pi_{\bO}} \arrow{r}{\bC} & \cAut(\Cl_0(\cM,q),\underline{\sigma}_0) \\
\PGO_q \arrow[swap]{ur}{\Cl} & 
\end{tikzcd}
\begin{tikzcd}
\Spin_{(\cA,\sigma,f)} \arrow{dr}{\bC_{\Inn}} \arrow{d}{\chi} & \\
\bO_{(\cA,\sigma,f)} \arrow{d}{\pi_{\bO}} \arrow{r}{\bC} & \cAut(\Cl(\cA,\sigma,f),\underline{\sigma}) \\
\PGO_{(\cA,\sigma,f)} \arrow[swap]{ur}{\Cl} & 
\end{tikzcd}
\end{equation}
where in the left diagram, by definition of $\chi \colon \Spin_q \to \bO_q$, we have that $\bC_{\Inn} = \bC \circ \chi$, so the top triangle commutes and the bottom triangle commutes by \cite[1.14]{Rue23}. We reuse notation for the diagram on the right since it is the twisted version of the diagram on the left, i.e., both diagrams agree locally. On the right, $\bC_{\Inn}$ still sends an element to its inner automorphism and now $\Cl$ is directly given by the functoriality of the Clifford algebra. In the cases when the Clifford algebra has its canonical quadratic triple, we know from \Cref{eq_Clifford_functorial} that $\Cl$ lands in the automorphism group of this canonical triple, and so we have a commutative diagram
\begin{equation}\label{eq_action_on_canonical}
\begin{tikzcd}
\Spin_{(\cA,\sigma,f)} \arrow{dr}{\bC_{\Inn}} \arrow{d}{\chi} & \\
\bO_{(\cA,\sigma,f)} \arrow{d}{\pi_{\bO}} \arrow{r}{\bC} & \cAut(\Cl(\cA,\sigma,f),\underline{\sigma},\underline{f}). \\
\PGO_{(\cA,\sigma,f)} \arrow[swap]{ur}{\Cl} & 
\end{tikzcd}
\end{equation}

There are two final groups we will need. For this, we assume $(\cA,\sigma,f)$ is a quadratic triple of even rank. In this case, \cite[4.2.0.15]{CF} says that the center of $\Cl(\cA,\sigma,f)$, denoted $\cZ$, is a finite \'etale $\cO$--algebra of degree $2$. Since any automorphism of an algebra restricts to an automorphism of the center, this gives us a map
\[
\mathrm{Arf} \colon \bO_{(\cA,\sigma,f)} \xrightarrow{\bC} \cAut(\Cl(\cA,\sigma,f),\underline{\sigma}) \to \cAut(\cZ)\cong \ZZ/2\ZZ
\]
called the \emph{Arf map} in \cite[4.3.0.27]{CF} (and over rings is called the \emph{Dickson map} in \cite[IV.5.1]{Knus}). Of course, there is also another Arf map
\[
\mathrm{Arf}' \colon \PGO_{(\cA,\sigma,f)} \xrightarrow{\Cl} \cAut(\Cl(\cA,\sigma,f),\underline{\sigma}) \to \cAut(\cZ)\cong \ZZ/2\ZZ
\]
and $\mathrm{Arf} = \mathrm{Arf'}\circ \pi_{\bO}$. We set
\begin{align*}
\bO_{(\cA,\sigma,f)}^+ &= \Ker(\mathrm{Arf})\text{, and} \\
\PGO_{(\cA,\sigma,f)}^+ &= \Ker(\mathrm{Arf}').
\end{align*}
When $(\cA,\sigma,f)$ is adjoint to $(\cM,q)$ we denote these $\bO_q^+$ and $\PGO_q^+$. When $(\cM,q)=(\HH(\cV),q_{2n})$ is the even rank hyperbolic quadratic form of \Cref{split_Clifford}, we write $\bO_{2n}^+$ and $\PGO_{2n}^+$. The vector representation factors through $\bO_{(\cA,\sigma,f)}^+$, in particular we have surjective maps of sheaves
\[
\Spin_{(\cA,\sigma,f)} \overset{\chi}{\surj} \bO_{(\cA,\sigma,f)}^+ \overset{\pi_{\bO}}{\surj} \PGO_{(\cA,\sigma,f)}^+.
\] 
These groups are semisimple groups of type $D$. The spin groups are simply connected and $\PGO^+$ is adjoint.

A small fact we mention here for later use is that the vector representation respects isomorphisms of quadratic triples.
\begin{lem}\label{autos_respect_vector_rep}
Let $\psi \colon (\cA_1,\sigma_1,f_1) \iso (\cA_2,\sigma_2,f_2)$ be an isomorphism of quadratic triples. Then, the diagram
\[
\begin{tikzcd}
\Spin_{(\cA_1,\sigma_1,f_1)} \ar[r,"\Cl(\psi)"] \ar[d,"\chi_1"] & \Spin_{(\cA_2,\sigma_2,f_2)} \ar[d,"\chi_2"] \\
\bO^+_{(\cA_1,\sigma_1,f_1)} \ar[r,"\psi"] & \bO^+_{(\cA_2,\sigma_2,f_2)}
\end{tikzcd}
\]
commutes where $\chi_i$ are the respective vector representations.
\end{lem}
\begin{proof}
It is clear from the definitions of $\Spin$ and $\bO^+$ that the algebra automorphisms $\Cl(\psi)$ and $\psi$ restrict to group homomorphisms as in the diagram. Since we may check the commutativity of the diagram over a cover, by working sufficiently locally we may assume that $\cA_i = \cEnd_{\cO}(\cM_i)$ and that $(\sigma_i,f_i)$ is adjoint to a regular quadratic form $q_i \colon \cM_i \to \cO$ for each $i$. Furthermore, we may assume that $\psi$ is induced by an isometry $\phi \colon (\cM_1,q_1) \iso (\cM_2,q_2)$, i.e., that
\begin{align*}
\psi \colon \cEnd_\cO(\cM_1) &\iso \cEnd_\cO(\cM_2) \\
B &\mapsto \phi \circ B \circ \phi^{-1}.
\end{align*}
Additionally, we will have $\Cl(\cA_1,\sigma_1,f_1) \cong \Cl_0(\cM_1,q_1)$ and $\Cl(\psi)(m_1\otimes \ldots \otimes m_k) = \phi(m_1)\otimes \ldots \otimes \phi(m_k)$.

Now, consider $a\in \Spin_{(\cA_1,\sigma_1,f_1)}$. The element $\psi(\chi_1(a))$ is the section of $\bO^+_{(\cA_2,\sigma_2,f_2)}$ which behaves as
\begin{align*}
\psi(\chi_1(a))(n) &= (\phi \circ \chi_1(a) \circ \phi^{-1})(n) \\
&= \phi\big(a \cdot \phi^{-1}(n) \cdot a^{-1}\big)
\end{align*}
where $a \cdot \phi^{-1}(n) \cdot a^{-1}$ lies in $\cM_1\subset \Cl(\cM_1,q_1)$. By definition, the diagram
\[
\begin{tikzcd}
\Cl(\cM_1,q_1) \ar[r,"\Cl(\phi)"] & \Cl(\cM_2,q_2) \\
\cM_1 \ar[u,hookrightarrow] \ar[r,"\phi"] & \cM_2 \ar[u,hookrightarrow]
\end{tikzcd}
\]
commutes and $\Cl(\phi)|_{\Cl_0(\cM_1,q_1)} = \Cl(\psi)$. So, identifying $\cM_1$ with its image in $\Cl(\cM_1,q_1)$ we have
\begin{align*}
\phi\big(a \cdot \phi^{-1}(n) \cdot a^{-1}\big) &= \Cl(\phi)\big(a \cdot \phi^{-1}(n) \cdot a^{-1}\big) \\
&= \Cl(\psi)(a) \cdot \Cl(\phi)(\phi^{-1}(n)) \cdot \Cl(\phi)(a)^{-1} \\
&= \Cl(\psi)(a) \cdot n \cdot \Cl(\psi)(a)^{-1}.
\end{align*}
Therefore, we see that $\psi(\chi_1(a)) = \chi_2(\Cl(\psi)(a))$, which proves the claim.
\end{proof}

\subsubsection{Chevalley Generators}\label{Chevalley_generators}
For the rest of this section we work with $(\HH(\cV),q_{2n})$ of \Cref{split_Clifford}. We fix the ordered basis
\[
\{v_1,v_2,\ldots,v_n,v_n^*,\ldots,v_2^*,v_1^*\}
\]
of $\HH(\cV)$ and use it to identify $\cEnd_{\cO}(\HH(\cV)) = \Mat_{2n}(\cO)$.

The groups $\Spin_{2n}$, $\bO_{2n}^+$, and $\PGO_{2n}^+$ are Chevalley groups. In particular, they have elementary subpresheaves of groups generated by Chevalley generators. For background on Chevalley groups over fields and algebraic relations between Chevalley generators we refer to \cite{Steinberg}.

All three of these groups are of type $D_n = \{\pm e_i \pm e_j \mid 1\leq i<j \leq n\}$. We recall from \cite[1.3.21]{Thesis} that we may choose the following elements as Chevalley generators of $\Spin_{2n} \subseteq \Cl(\HH(\cV),q_{2n})$ where $1\leq i < j \leq n$.
\begin{align*}
X_{e_i-e_j}(t) &= 1+tv_iv_j^* & X_{-e_i+e_j}(t) &= 1+tv_jv_i^* \\
X_{e_i+e_j}(t) &= 1+tv_iv_j & X_{-e_i-e_j}(t) &= 1+tv_j^*v_i^*
\end{align*}
By \cite[1.3.22]{Thesis} (where $\bO^+$ is called $\SO$), these map via $\chi$ to the following choice of Chevalley generators for $\bO_{2n}^+$.
\begin{align*}
x_{e_i-e_j}(t) &= I+t(E_{ij}-E_{\bar{j}\bar{i}}) & x_{-e_i+e_j}(t) &= I+t(E_{ji}-E_{\bar{i}\bar{j}}) \\
x_{e_i+e_j}(t) &= I+t(E_{i\bar{j}}-E_{j\bar{i}}) & x_{-e_i-e_j}(t) &= I+t(E_{\bar{j}i}-E_{\bar{i}j})
\end{align*}
where $\bar{i} = 2n+1-i$ and $E_{ij}$ is the matrix with $1$ in the $(i,j)$ entry and zeroes elsewhere. We choose the Chevalley generators
\[
x'_{\alpha}(t) = \pi_{\bO}(x_\alpha(t))
\]
for $\PGO_{2n}^+$. These choices then define presheaves of subgroups called the \emph{elementary subgroups}.
\begin{defn}\label{elementary_group}
We define the following functors into groups.
\begin{align*}
\bE_{\Spin} \colon \Sch_S &\to \Grp \\
T &\mapsto \langle X_{\alpha}(t) \mid \alpha \in D_n,\; t\in \cO(T)\rangle \\[2ex]
\bE_{\bO^+} \colon \Sch_S &\to \Grp \\
T &\mapsto \langle x_{\alpha}(t) \mid \alpha \in D_n,\; t\in \cO(T)\rangle \\[2ex]
\bE_{\PGO^+} \colon \Sch_S &\to \Grp \\
T &\mapsto \langle x'_{\alpha}(t) \mid \alpha \in D_n,\; t\in \cO(T)\rangle.
\end{align*}
\end{defn}

\begin{prop}\label{Chevalley_sheafification}
The sheafification of $\bE_{\Spin}$ is $\bE_{\Spin}^\sharp = \Spin_{2n}$. Similarly, $\bE_{\bO^+}^\sharp = \bO_{2n}^+$ and $\bE_{\PGO^+}^\sharp = \PGO_{2n}^+$.
\end{prop}
\begin{proof}
It is sufficient to argue that $\bE_{\bO^+}^\sharp = \bO_{2n}^+$. If this is true, then the image of this group under the homomorphism $\pi_{\bO}$ will be both $\bE_{\PGO^+}^\sharp$ and $\PGO_{2n}^+$, thereby implying the last claim. For the first claim, using \cite[Lemma 28(d)]{Steinberg} (which is only stated over fields but holds over any ring), one can compute that the center of $\bE_{\Spin}$ contains the elements
\[
h_{e_{n-1}-e_n}(t)h_{e_{n-1}+e_n}(t)
\]
for $t\in \runity_2$ where the elements $h_{\alpha}(t)$ are as defined in \cite[Lemma 19]{Steinberg}. These elements also form the kernel of $\chi \colon \Spin_{2n} \to \bO_{2n}^+$, and therefore $\Ker(\chi) \subseteq \bE_{\Spin}$. Thus, for every section $a\in \Spin_{2n}$, since $\chi(a)$ locally belongs to $\bE_{\bO^+}$ by assumption and $\bE_{\Spin}(T) \surj \bE_{\bO^+}(T)$ is surjective for all $T \in \Sch_S$, the section $a$ will locally belong to $\bE_{\Spin}$. Hence $\bE_{\Spin}^\sharp = \Spin_{2n}$, verifying the last claim.

Hence, we now argue that $\bE_{\bO^+}^\sharp = \bO_{2n}^+$. The argument hinges on the following fact about the fppf topology. For every ring $R$ and element $r\in R$, the ring extension $R \to R[x]/\langle x^2-r \rangle$ is finitely presented and is also faithfully flat since $R[x]/\langle x^2-r\rangle \cong R\oplus Rx$ is free as an $R$--module. Therefore, $\{\Spec(R[x]/\langle x^2 - r \rangle) \to \Spec(R)\}$ is an fppf covering. This means that when considering an fppf covering of $S$ by affine schemes, we may assume that any element of $R$ has a square root in $R$ by refining the cover if necessary.

Now, let $T\in \Sch_S$ and $A\in \bO_{2n}^+(T)$. Let $\{U_i \to T\}_{i\in I}$ be an affine open cover. Then $A|_{U_i} \in \bO_{2n}^+(U_i) \subset \Mat_{2n}(\cO(U_i))$ is a matrix with entries $A|_{U_i}=[a_{ij}]$. Since $A|_{U_i}$ is invertible, $\det(A|_{U_i}) \in \cO(U_i)^\times$, which means $\langle a_{11},a_{12},\ldots,a_{1n}\rangle = \cO(U_i)$. So, by refining the open cover to include the localizations $\Spec(\cO(U_i)_{a_{1j}})$ if needed, we may assume at least one of the entries in the first row of $A|_{U_i}$ is invertible. Furthermore, since $(q_{2n}\circ (A|_{U_i})^{-1})(v_n^*) = q_{2n}(v_n^*)=0$ and $(A|_{U_i})^{-1} = \sigma(A|_{U_i})$ which has final column $\begin{bmatrix} a_{1,2n} & a_{1,2n-1} & \ldots & a_{11}\end{bmatrix}^T$, we obtain the relation that
\[
\sum_{k=1}^n a_{1,k}a_{1,2n+1-k} = 0.
\]
With these facts, and the square root of elements as necessary, we can repeat the calculations for $\bE_{\SO_{2n}}$ from the proof of \cite[3.4.8]{Thesis} to show that $A|_{U_i}\in \bE_{\bO^+}(U_i)$. Since $A$ is locally in $\bE_{\bO^+}$ we know that $A\in \bE_{\bO^+}^\sharp(T)$, and so $\bE_{\bO^+}^\sharp = \bO_{2n}^+$ as desired.
\end{proof}
\begin{rem}
The above result is also an application of \cite[Expose XXIII, 3.5.4]{SGA3}, which proves the claim more generally for groups of any type with a pinning. In their setup, one has unipotent root subgroups as well as a torus which generate the subsheaf together. One can show that the elements $h_\alpha(t)$ generate an appropriate torus and therefore the result from \cite{SGA3} applies.

Furthermore, the above result is stated for, and the proof inherently uses, the fppf topology. However, in some cases weaker topologies are sufficient. It was known by E.~Abe in \cite{Abe} and H.~Matsumoto in \cite{Matsu} that $\bE_\bG(R) = \bG(R)$ when $\bG$ is simply connected and $R$ is a local ring. This implies that for simply connected groups, $\bG$ is the Zariski sheafification of $\bE_\bG$. In particular this holds for $\Spin_{2n}$. Alternatively, if we assume that $\frac{1}{2} \in \cO^\times$, then whenever we need to add a square root of an element in the above argument, we may do so with an \'etale cover. Therefore in this case, $\bO_{2n}^+$ and $\PGO_{2n}^+$ are the \'etale sheafifications of their elementary subgroups.
\end{rem}
\Cref{Chevalley_sheafification} will allow us to use calculations on Chevalley generators to identify homomorphisms between these groups, since a map of presheaves between the elementary subgroups extends uniquely to a morphism between the sheaves of groups.

\subsection{Cohomology, Stacks, and Twisting}\label{cohomology_twisting}
We will make use of first non-abelian cohomology in the style of \cite[2.4.2]{Gir}. For a sheaf of groups $\bG \colon \Sch_S \to \Grp$, a \emph{$\bG$--torsor} is a sheaf of sets $\cP \colon \Sch_S \to \Grp$ with a right $\bG$--action, i.e., a natural transformation $\cP\times \bG \to \cP$ such that
\begin{enumerate}[label={\rm(\roman*)}]
\item \label{torsor_defn_i} for all $T \in \Sch_S$, the map $\cP(T)\times \bG(T) \to \cP(T)$ defines a simply transitive right $\bG(T)$--action on $\cP(T)$, and
\item \label{torsor_defn_ii} there exists an fppf cover $\{T_i \to S\}_{i\in I}$ over which $\cP(T_i)\neq \O$ for all $i\in I$.
\end{enumerate}
A torsor $\cP$ may have $\cP(T) = \O$ for some $T\in \Sch_S$. In this case we view the map $\O \times \bG(T) \to \O$ as vacuously giving a simply transitive action of $\bG(T)$ on $\O$. A morphism of $\bG$--torsors is a natural transformation of sheaves which is equivariant with respect to the right $\bG$--actions. Any morphism of torsors must be an isomorphism. The sheaf of sets $\bG$ itself viewed as torsor via right multiplication is called the \emph{trivial torsor}. The first cohomology set $H^1(S,\bG)$ is defined to be the pointed set of isomorphism classes of $\bG$--torsors with the class of the trivial torsor as the base point.

Two sheaves $\cF$ and $\cG$ on $\Sch_S$ are called \emph{twisted forms} of each other if they are locally isomorphic, i.e., if there exists a cover $\{T_i \to S\}_{i\in I}$ over which $\cF|_{T_i} \cong \cG|_{T_i}$ for all $i\in I$. In this case, the sheaf of internal isomorphisms
\begin{align*}
\cIsom(\cF,\cG) \colon \Sch_S &\to \Sets \\
T &\mapsto \Isom(\cF|_T,\cG|_T)
\end{align*}
is an $\cAut(\cF)$--torsor. For example, if $(\cA,\sigma,f)$ is a quadratic triple of constant rank $2n$, then the sheaf of quadratic triple isomorphisms
\[
\cP = \cIsom((\cEnd_{\cO}(\HH(\cV)),\sigma_{q_{2n}},f_{q_{2n}}),(\cA,\sigma,f))
\]
is a $\PGO_{2n}$--torsor.

We will also make use of stacks over $\Sch_S$ following \cite{Olsson} and \cite{Vis}. Let $p \colon \fF \to \Sch_S$ be a stack. For $T \in \Sch_S$ we denote by $\fF(T)$ the fiber of $\fF$ over $T$. For a morphism $f\colon T' \to T$ of $S$--schemes and an object $x\in \fF(T)$, a pullback of $x$ will be denoted $f^*(x) \in \fF(T')$ or simply by $x|_{T'}$ when $f$ is clear. The automorphism sheaf of an object $x\in \fF(T)$ is
\begin{align*}
\bAut(x) \colon \Sch_T &\to \Grp \\
T' &\mapsto \Aut_{\fF(T')}(x|_{T'}).
\end{align*}
Two objects $x,y\in \fF(T)$ in the same fiber of a stack are called \emph{twisted forms} if they are locally isomorphic, i.e., there is a cover $\{T_i \to T\}_{i\in I}$ such that $x|_{T_i} \cong y|_{T_i}$ in $\fF(T_i)$ for all $i\in I$. In this case, $\cIsom(x,y)$ is an $\bAut(x)$--torsor over $T$.

When the stack $\fF$ is a gerbe with $x\in \fF(S)$, the isomorphism classes of $\fF(S)$ are classified by $H^1(S,\bAut(x))$. Following \cite[3.1.10]{Olsson}, a morphism $\varphi \colon \fF \to \fG$ of stacks (or more generally fibered categories) is called an \emph{equivalence} if for every $T\in \Sch_S$ the functor between fibers $\varphi_T \colon \fF(T) \to \fG(T)$ is an equivalence of categories.

\begin{lem}[{\cite{Gir},\cite[1.19]{GNR2}}]\label{equivalence_gerbes}
Let $\varphi \colon \fF \to \fG$ be a morphism of gerbes. Assume there is an object $x\in \fF(S)$. If the associated group homomorphism
\[
\bAut(\varphi)\colon \bAut(x) \to \bAut(\varphi(x))
\]
is an isomorphism, then $\varphi$ is an equivalence of stacks.
\end{lem}
We attribute the above lemma to two sources because it is a folklore result which easily follows from results in \cite{Gir}, but does not appear there in this form. A proof which assembles the relevant facts from \cite{Gir} is given in \cite[1.19]{GNR2}.

One motivation for our use of stacks is to take advantage of the following result.
\begin{lem}[{\cite[2.5.3(i)]{Gir}}]\label{stack_cohomology}
Let $\varphi \colon \fF \to \fG$ be a morphism of gerbes. For an object $x\in \fF(S)$, the associated group homomorphism $\bAut(\varphi) \colon \bAut(x) \to \bAut(\varphi(x))$ induces the following map on cohomology
\begin{align*}
H^1(S,\bAut(x)) &\to H^1(S,\bAut(\varphi(x))) \\
[x'] &\mapsto [\varphi(x')]
\end{align*}
where $[x']$ denotes the isomorphism class of $x' \in \fF(S)$ and $[\varphi(x')]$ the isomorphism class of $\varphi(x') \in \fG(S)$.
\end{lem}

\subsection{\'Etale Covers}\label{etale_covers}
By definition, an \emph{\'etale cover of degree $d$} of a scheme $T$ is an affine morphism $f\colon L \to T$ such that $f_*(\cO|_L)$ is an \'etale $\cO|_T$--algebra of rank $d$, i.e., it is a twisted form as commutative $\cO|_T$--algebras of $\cO|_T^d$. As in \cite[2.5.2.9]{CF}, since $\cAut(\cO|_T^d) \cong \SS_d|_T$ is the constant group sheaf associated to the permutation group of $d$ letters, isomorphism classes of degree $d$ \'etale covers of $T$ are classified by $H^1(T,\SS_d)$. If $f \colon L \to T$ is such a cover, since $f_*(\cO|_T)$ is locally isomorphic to $\cO|_T^d$, there will be a cover $\{T_i \to T\}_{i\in I}$ such that $L\times_T T_i \cong T_i^{\sqcup d}$. This provides an equivalent characterization of \'etale covers of degree $d$.

To set our conventions, $\cAut_S(S^{\sqcup d})\cong\SS_d$ acts by permuting components as follows. Let $T \in \Sch_S$ and write $T=\sqcup_{i\in I} T_i$ as the disjoint union of its connected components. Writing $P_d$ for the abstract permutation group of $\{1,\ldots,d\}$, we thus have $\SS_d(T) \cong \prod_{i\in I} P_d$ and we write $\beta=(\beta_i)_{i\in I}$ for a section over $T$. Now,
\[
T^d = \big(\bigsqcup_{i\in I}T_{i,1}\big)\sqcup \ldots \sqcup \big(\bigsqcup_{i\in I}T_{i,d}\big)
\]
where $T_{i,1}=T_{i,2}=\ldots = T_{i,d} = T_i$ for all $i\in I$. Then,
\begin{align}
\beta \colon T^d &\to T^d \label{permutation_convention} \\
T_{i,k} &\xrightarrow{\Id} T_{i,\beta_i^{-1}(k)}, \nonumber
\end{align}
by which we mean, $\beta$ sends the $T_{i,k}$ component to the $T_{i,\beta_i^{-1}(k)}$ component via the identity. Alternatively, $T_{i,\beta_i(k)}$ ends up in the $T_{i,k}$ position. However, since $\SS_d$ is the sheafification of the constant presheaf $\PP_d$ associated to the abstract group $P_d$, we will often be able to assume without loss of generality that $\beta \in \SS_3(T)$ is a \emph{diagonal section}, i.e., that $\beta_i = \beta_j$ for all $i,j\in I$ or equivalently that $\beta \in \PP_3(T) \subseteq \SS_3(T)$. For example, showing that a homomorphism $\varphi \colon \SS_3 \to \SS_3$ restricts to an isomorphism $\PP_3 \iso \PP_3$ is sufficient to conclude that $\varphi$ itself is an isomorphism.

Analogously to what is done over fields in \cite[\S 2.1]{KT03}, we will use the construction of $\Sigma_2(L)$ from a degree $3$ \'etale cover $L \to S$. We review this construction explicitly in this case. Writing the set of size six as $\{(1,2),(2,3),(3,1),(1,3),(2,1),(3,2)\}$ induces a group homomorphisms $P_3 \inj P_6$ by allowing elements of $P_3$ to act componentwise on the pairs. This homomorphism of abstract groups then corresponds to a homomorphism $\SS_3 \to \SS_6$ of the associated constant group sheaves. Now, by the above discussion, a degree $3$ \'etale cover $f\colon L \to S$ corresponds to an $\SS_3$--torsor $\cP_L$ over $S$. Therefore, $\cP_L \wedge^{\SS_3} \SS_6$ is an $\SS_6$--torsor which corresponds to a degree $6$ \'etale cover which we denote by $f' \colon \Sigma_2(L) \to S$.

Now consider the split degree $6$ \'etale cover,
\begin{equation}\label{eq_S6_order}
S^{\sqcup 6} = (S_{1,2} \sqcup S_{2,3} \sqcup S_{3,1}) \sqcup (S_{1,3} \sqcup S_{2,1} \sqcup S_{3,2})
\end{equation}
where each $S_{i,j} = S$ also. We then have two degree $2$ \'etale covers
\begin{align*}
\pi \colon S^{\sqcup 6} &\to S^{\sqcup 3} & \tau \colon S^{\sqcup 6} &\to S^{\sqcup 3} \\
x\in S_{i,j} &\mapsto x\in S_i & x\in S_{i,j} &\mapsto x\in S_j
\end{align*}
by which we mean, $\pi$ sends the $i,j$ labelled component of $S^{\sqcup 6}$ to the $i$ labelled component of $S^{\sqcup 3}$ via the identity map $S \to S$, and likewise $\tau$ sends the $i,j$ component to the $j$ component. These maps correspond to the two morphisms of \'etale $\cO$--algebras
\begin{align*}
\pi' \colon \cO^3 &\to \cO^6 & \tau' \colon \cO^3 &\to \cO^6 \\
(a,b,c) &\mapsto (a,b,c,a,b,c) & (a,b,c)&\mapsto (b,c,a,c,a,b).
\end{align*}
It is straightforward to check that both of these maps are equivariant with respect to the $\SS_3$ action on $\cO^3$ and on $\cO^6$ via the homomorphism $\SS_3 \inj \SS_6$, and therefore they twist into morphisms $\pi_L', \tau_L' \colon f_*(\cO|_L) \to f'_*(\cO|_{\Sigma_2(L)})$ which correspond to two degree $2$ \'etale covers
\[
\pi_L \colon \Sigma_2(L) \to L \text{ and } \tau_L \colon \Sigma_2(L) \to L.
\]
We call $\Sigma_2(L)$ the \emph{trialitarian cover} of $L$.

\begin{lem}\label{Sigma(L)_pullback}
Let $L\to S$ be a degree $3$ \'etale cover. For any scheme $T\in \Sch_S$, the pullback $L\times_S T \to T$ is another degree $3$ \'etale cover. In this case, we have an isomorphism of schemes $\Sigma_2(L)\times_S T \cong \Sigma_2(L\times_S T)$.
\end{lem}
\begin{proof}
This can be seen in terms of cohomology. There is a restriction map $H^1(S,\SS_3) \to H^1(T,\SS_3|_T)$ which sends the class of a torsor $[\cP]$ to the class $[\cP|_T]$. Likewise, there is a restriction map for $\SS_6$. In addition, we have maps $H^1(S,\SS_3) \to H^1(S,\SS_6)$ and $H^1(T,\SS_3|_T) \to H^1(T,\SS_6|_T)$ coming from our group homomorphism $\SS_3 \to \SS_6$. These fit into a commutative diagram
\[
\begin{tikzcd}
H^1(S,\SS_3) \ar[r] \ar[d] & H^1(T,\SS_3|_T) \ar[d] \\
H^1(S,\SS_6) \ar[r] & H^1(T,\SS_6|_T).
\end{tikzcd}
\]
Tracing the class of the cover $L\to S$ through this diagram, we obtain
\[
\begin{tikzcd}
{[L]} \ar[r,mapsto] \ar[d,mapsto] & {[L\times_S T]} \ar[d,mapsto] \\
{[\Sigma_2(L)]} \ar[r,mapsto] & {[\Sigma_2(L)\times_S T]} = {[\Sigma_2(L\times_S T)]}
\end{tikzcd}
\]
which justifies our claim.
\end{proof}

We construct one additional map, an isomorphism of $\Sigma_2(L)$, which will be used to define trialitarian algebras. Consider the isomorphism of $\rho \colon S^{\sqcup 6} \iso S^{\sqcup 6}$ given by the permutation
\begin{equation}\label{eq_rho}
\begin{tikzcd}[column sep=-6pt]
(S_{1,2} \ar[drr] & \sqcup & S_{2,3} \ar[drr] & \sqcup & S_{3,1}) \ar[dllll,start anchor={[xshift=-4pt]south},end anchor={[xshift=4pt]north}] & \sqcup & (S_{1,3} \ar[drrrr,start anchor={[xshift=4pt]south},end anchor={[xshift=-8pt]north}] & \sqcup & S_{2,1} \ar[dll,end anchor=north] & \sqcup & S_{3,2}) \ar[dll,end anchor=north] \\
(S_{1,2} & \sqcup & S_{2,3} & \sqcup & S_{3,1}) & \sqcup & (S_{1,3} & \sqcup & S_{2,1} & \sqcup & S_{3,2})
\end{tikzcd}
\end{equation}
where the arrows are the identity $S\to S$. Viewed as an element in $\SS_6(S)$, this is $(1\;3\;2)(4\;5\;6)$ according to our conventions in \Cref{permutation_convention}. Given an $\cO|_{S^{\sqcup 6}}$--algebra $\cE=(\cA_1,\cA_2,\cA_3,\cB_1,\cB_2,\cB_3)$, the pullback along $\rho$ is
\begin{equation}\label{rho_pullback}
\rho^*(\cE)=(\cA_2,\cA_3,\cA_1,\cB_3,\cB_1,\cB_2).
\end{equation}
The scheme isomorphism $\rho$ corresponds to the $\cO$--algebra isomorphism
\begin{align*}
\rho' \colon \cO^6 &\to \cO^6 \\
((a,b,c),(d,e,f)) &\mapsto ((b,c,a),(f,d,e)).
\end{align*}

\begin{lem}\label{split_rho}
The isomorphism $\rho'$ is $\SS_3$--equivariant.
\end{lem}
\begin{proof}
Without loss of generality, it is sufficient to check this for the diagonal permutations $c =(1\,2\,3)$ and $\lambda =(2\,3)$ of $\SS_3(T)$.

Since $\rho'$ is of the form $(c,c^{-1})$, it is clear it commutes with the action of $c$, which is of the form $(c,c)$.

The action of $\lambda$ performs the permutation
\[
((a,b,c),(d,e,f)) \mapsto ((d,f,e),(a,c,b)).
\]
Therefore we can compute that
\[
(\rho' \circ \lambda)(a,b,c),(d,e,f)) = \rho'((d,f,e),(a,c,b)) = ((f,e,d),(b,a,c))
\]
and that
\[
(\lambda \circ \rho')((a,b,c),(d,e,f)) = \lambda((b,c,a),(f,d,e)) = ((f,e,d),(b,a,c))
\]
and see that $\rho'$ commutes with $\lambda$ as desired.
\end{proof}

Therefore, this map descends to an isomorphism of $\cO$--algebras
\[
\rho_{\Sigma_2(L)}'\colon f'_*(\cO|_{\Sigma_2(L)}) \to f'_*(\cO|_{\Sigma_2(L)})
\]
which corresponds to an isomorphism of schemes $\rho_{\Sigma_2(L)} \colon \Sigma_2(L) \to \Sigma_2(L)$. It is order $3$ since $\rho$ is.

\section{Triality Without Octonions}
In this section we demonstrate that phenomena surrounding triality which are described in \cite[Ch. X]{KMRT} also exist in our context over a scheme. We do so using computations with Chevalley generators in place of the traditional computations with Octonions.

\subsection{Automorphism Groups of Groups}\label{automorphisms_of_groups}
We fix a \emph{pinning} of each of the groups $\Spin_8$, $\bO_8^+$, and $\PGO_8^+$ in the style of \cite[XXIII.1.1]{SGA3}. This first involves a choice of maximal torus for each group. Using the theory of Chevalley generators, if $x_\alpha(t)$ are generators for a split group $\bG$, then setting
\[
w_\alpha(t)=x_\alpha(t)x_{-\alpha}(-t^{-1})x_\alpha(t) \text{ and } h_\alpha(t) = w_\alpha(t)w_\alpha(-1)
\]
for $t\in \cO^\times$ produces a commutative diagonalizable subgroup presheaf generated by sections $h_\alpha(t)$. We let $\bT$ be the sheafification of this presheaf, which is a maximal torus of $\bG$. In $\bO_{2n}^+$ with our choice of generators as in \Cref{Chevalley_generators}, this takes the familiar form of diagonal elements
\[
\bT_{\bO_{2n}^+} = \{\diag(t_1,\ldots,t_n,t_n^{-1},\ldots,t_1^{-1}) \in \Mat_{2n}(\cO) \mid t_i \in \cO^\times\}.
\]
For the simply connected group $\Spin_{2n}$, the character lattice $\Hom(\bT_{\Spin_{2n}},\GG_m) = \Lambda_w$ is the weight lattice of the root system $D_n$, while the adjoint group's character lattice $\Hom(\bT_{\PGO_{2n}^+},\GG_m) = \Lambda_r$ is the root lattice. These lattices are mutual duals and are isomorphic. The intermediate group $\bO_{2n}^+$ has the intermediate character lattice $\Lambda = \Hom(\bT_{\bO_{2n}^+},\GG_m) = \Span_\ZZ(\{e_1,\ldots,e_n\})$. The data of a simple system $\Delta \subseteq D_n$, these character lattices, and our fixed choices of Chevalley generators defines a pinning for each of these three groups. 

Given a graph automorphism of the Dynkin diagram of $D_n$ relative to the simple system $\Delta$, it corresponds to a bijection $\varphi \colon \Delta \iso \Delta$ which extends to an automorphism of the root system and hence also to an automorphism of the root and weight lattices. If this automorphism stabilizes the associated character lattice (which is automatic when $\bG$ is simply connected or adjoint) then it extends uniquely to an automorphism of the elementary subgroup $\bE_{\bG} \iso \bE_{\bG}$ defined by sending the Chevalley generators $x_\alpha(t) \mapsto x_{\varphi(\alpha)}(t)$ and $x_{-\alpha}(t) \mapsto x_{-\varphi(\alpha)}(t)$ for all $\alpha \in \Delta$. The images of the $x_\alpha(t)$ for other roots $\alpha \notin \Delta$ are then determined by the commutator relations of the group. By \Cref{Chevalley_sheafification}, this isomorphism extends to a unique isomorphism of $\bG$, which we also denote $\varphi \colon \bG \iso \bG$. 

Let $\Lambda$ be the character lattice of $\bG$ and let $\bAut(\Delta,\Lambda)$ denote the constant group sheaf associated to the abstract group $\Aut(\Delta,\Lambda)$ of graph automorphisms stabilizing $\Lambda$. By the discussion above, we have a group homomorphisms $a\colon \bAut(\Delta,\Lambda) \to \bAut(\bG)$. For any of the three groups we are considering, the canonical morphism $\bG \to \PGO_{2n}^+$ has a central kernel and the center of $\PGO_{2n}^+$ is trivial, so the group of inner automorphisms of $\bG$ is isomorphic to $\PGO_{2n}^+$. This gives a morphism $\Inn \colon \PGO_{2n}^+ \to \bAut(\bG)$. Then, by \cite[XXIV.1.3]{SGA3} we have an exact sequence
\[
1 \to \PGO_{2n}^+ \xrightarrow{\Inn} \bAut(\bG) \xrightarrow{p} \Out(\bG) \to 1
\]
where $\Out(\bG)=\bAut(\bG)/\Inn(\bG)$ is the quotient sheaf of automorphisms modulo inner automorphisms. Furthermore, \cite[XXIV.1.3(iii)]{SGA3} states that $p\circ a \colon \bAut(\Delta,\Lambda) \iso \Out(\bG)$ is an isomorphism. Finally, \cite[XXIV.1.4]{SGA3} tells us that the automorphism group of $\bG$ is the semidirect product
\[
\cAut(\bG) = \PGO_{2n}^+ \rtimes \bAut(\Delta,\Lambda).
\]
where $\bAut(\Delta,\Lambda)$ acts on $\PGO_{2n}^+$ by acting on Chevalley generators as described above since any graph automorphism fixes the root lattice.

Focusing on $\Spin_8$, $\bO_8^+$, and $\PGO_8^+$, we work with the following non-standard simple system $\Delta$ of the root system $D_4$.
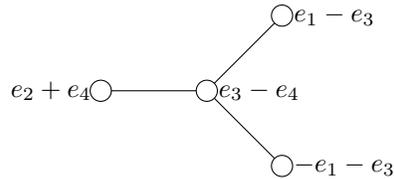
\[
\begin{tikzpicture}
\draw[black] (0,0) -- (1,1);
\draw[black] (0,0) -- (1,-1);
\draw[black] (-1.41,0) -- (0,0);
\draw[black, fill=white] (0,0) circle (4pt) node[anchor=west]{\hspace{1pt}$e_3-e_4$};
\draw[black, fill=white] (1,1) circle (4pt) node[anchor=west]{\hspace{1pt}$e_1-e_3$};
\draw[black, fill=white] (1,-1) circle (4pt) node[anchor=west]{\hspace{1pt}$-e_1-e_3$};
\draw[black, fill=white] (-1.41,0) circle (4pt) node[anchor=east]{$e_2+e_4$\hspace{1pt}}; 
\end{tikzpicture}
\]
The justification for this choice is explained in \Cref{choice_remark}. If $\bG$ is one of $\Spin_8$ or $\PGO_8^+$, then the associated lattice is either the weight or root lattice, and so $\bAut(\Delta,\Lambda) = \bAut(\Delta) = \SS_3$ is the constant group sheaf of permutations of $3$ elements, viewed as the graph isomorphisms of the simple system above. Therefore,
\begin{equation}\label{aut_spin_PGO}
\cAut(\bG) = \PGO_8^+ \rtimes \SS_3.
\end{equation}
We denote by $\theta^+$ the order $3$ automorphism of $\Delta$ coming from a counter-clockwise turn of the graph, and let $\theta^- = (\theta^+)^2$ be its inverse.
\begin{lem}\label{theta+_description}
Let $\bG$ be one of $\Spin_8$ or $\PGO_8^+$ and let $x_\alpha(t)$ denote its Chevalley generators as chosen in \Cref{Chevalley_generators}. The automorphism $\theta^+ \colon \bG \to \bG$ behaves as follows on Chevalley generators.
\begin{align*}
x_{e_1-e_2}(t) &\mapsto x_{e_3+e_4}(-t) & x_{-e_1+e_2}(t) &\mapsto x_{-e_3-e_4}(-t)\\
x_{e_1-e_3}(t) &\mapsto x_{e_2+e_4}(t) & x_{-e_1+e_3}(t) &\mapsto x_{-e_2-e_4}(t) \\
x_{e_1-e_4}(t) &\mapsto x_{e_2+e_3}(-t) & x_{-e_1+e_4}(t) &\mapsto x_{-e_2-e_3}(-t)\\
x_{e_2-e_3}(t) &\mapsto x_{e_2-e_3}(t) & x_{-e_2+e_3}(t) &\mapsto x_{-e_2+e_3}(t) \\
x_{e_2-e_4}(t) &\mapsto x_{e_2-e_4}(t) & x_{-e_2+e_4}(t) &\mapsto x_{-e_2+e_4}(t)\\
x_{e_3-e_4}(t) &\mapsto x_{e_3-e_4}(t) & x_{-e_3+e_4}(t) &\mapsto x_{-e_3+e_4}(t) \\[2ex]
x_{e_1+e_2}(t) &\mapsto x_{-e_1+e_2}(t) & x_{-e_1-e_2}(t) &\mapsto x_{e_1-e_2}(t)\\
x_{e_1+e_3}(t) &\mapsto x_{-e_1+e_3}(t) & x_{-e_1-e_3}(t) &\mapsto x_{e_1-e_3}(t) \\
x_{e_1+e_4}(t) &\mapsto x_{-e_1+e_4}(t) & x_{-e_1-e_4}(t) &\mapsto x_{e_1-e_4}(t)\\
x_{e_2+e_3}(t) &\mapsto x_{-e_1-e_4}(-t) & x_{-e_2-e_3}(t) &\mapsto x_{e_1+e_4}(-t)\\
x_{e_2+e_4}(t) &\mapsto x_{-e_1-e_3}(t) & x_{-e_2-e_4}(t) &\mapsto x_{e_1+e_3}(t) \\
x_{e_3+e_4}(t) &\mapsto x_{-e_1-e_2}(-t) & x_{-e_3-e_4}(t) &\mapsto x_{e_1+e_2}(-t) \\
\end{align*}
\end{lem}
\begin{proof}
We demonstrate one of the calculations for $x_\alpha(t)$ with $\alpha \notin \Delta$. Throughout this calculation we use the commutator relations computed in \cite[Appendix 2]{Rue20}. We can express $x_{e_1-e_2}(t)$ in terms of commutators as
\[
x_{e_1-e_2}(t)=\big((x_{e_1+e_3}(t),x_{-e_3+e_4}(-1)), x_{-e_2-e_4}(1) \big).
\]
This therefore gets mapped to the element
\begin{align*}
&\big((x_{\theta^+(e_1+e_3)}(t),x_{\theta^+(-e_3+e_4)}(-1)), x_{\theta^+(-e_2-e_4)}(1) \big) \\
=& \big((x_{-e_1+e_3}(t),x_{-e_3+e_4}(-1)), x_{e_1+e_3}(1) \big)\\
=& (x_{-e_1+e_4}(t),x_{e_1+e_3}(1)) \\ 
=& x_{e_3+e_4}(-t)
\end{align*}
as noted above. The other calculations proceed similarly.
\end{proof}

Let $\theta$ be the reflection of the graph which exchanges $e_1-e_3$ and $-e_1-e_3$. Since this stabilizes the intermediate lattice associated to $\bO_8^+$, we also have a well defined automorphism $\theta \colon \bO_8^+ \iso \bO_8^+$. As shown in the following lemma, this recovers the familiar fact that
\[
\bAut(\bO_8^+) = \PGO_8^+ \rtimes \SS_2 \cong \PGO_8.
\]
\begin{lem}\label{theta_description}
Let $\varphi \in \bO_8$ be the automorphism of $(\HH(\cV),q_8)$ which sends $v_1 \mapsto -v_1^*$ and $v_1^* \mapsto -v_1$ while fixing all other basis elements. Note $\varphi \notin \bO_8^+$. 
\begin{enumerate}[label={\rm(\roman*)}]
\item \label{theta_description_i} The automorphism $\theta \colon \Spin_8 \iso \Spin_8$ is the restriction of $\bC(\varphi)$.
\item \label{theta_description_ii} The automorphism $\theta \colon \bO_8^+ \iso \bO_8^+$ is the restriction of the inner automorphism $\Inn(\varphi)$ of $\bO_8$.
\item \label{theta_description_iii} We have that $\pi_{\bO}(\varphi) \in \PGO_8$, but not in $\PGO_8^+$, and $\theta \colon \PGO_8^+ \iso \PGO_8^+$ is the restriction of the inner automorphism $\Inn(\pi_{\bO}(\varphi))$ of $\PGO_8$.
\end{enumerate}
\end{lem}
\begin{proof}
\noindent\ref{theta_description_i}: We can compute the following.
\[
\bC(\varphi)(X_{e_1-e_3}(t)) = \bC(\varphi)(1+tv_1v_3^*) = 1-tv_1^*v_3^* = 1+tv_3^* v_1^* = X_{-e_1-e_3}(t).
\]
Because $\varphi$ is order $2$, so is $\bC(\varphi)$, and so we also have $\bC(\varphi)(X_{-e_1-e_3}(t)) = X_{e_1-e_3}(t)$. Further, since $X_{e_2+e_4}(t)$ and $X_{e_3-e_4}(t)$ do not involve $v_1$ or $v_1^*$, they are fixed by $\bC(\varphi)$. Therefore, $\bC(\varphi)$ agrees with $\theta$ on $X_\alpha(t)$ for $\alpha \in \Delta$ and therefore it agrees with $\theta$ on all of $\Spin_8$ as claimed. 

\noindent\ref{theta_description_ii}: Using our identification $\cEnd_{\cO}(\HH(\cV))\cong \Mat_8(\cO)$, the automorphism $\Inn(\varphi)$ acts on the basis elements of $\Mat_8(\cO)$ by
\begin{align*}
E_{1j} &\mapsto -E_{\bar{1}j} & E_{\bar{1}j} &\mapsto -E_{1j} \\
E_{j1} &\mapsto -E_{j\bar{1}} & E_{j\bar{1}} &\mapsto -E_{j1} \\
E_{11} &\mapsto E_{\bar{1}\bar{1}} & E_{\bar{1}\bar{1}} &\mapsto E_{11} \\ 
E_{1\bar{1}} &\mapsto E_{\bar{1}1} & E_{\bar{1}1} &\mapsto E_{1\bar{1}}
\end{align*}
while fixing all other basis elements. This has the effect of mapping $x_{e_1-e_3}(t) \mapsto x_{-e_1-e_3}(t)$ and $x_{-e_1-e_3}(t) \mapsto x_{e_1-e_3}(t)$ while fixing $x_{-e_2+e_3}(t)$ and $x_{e_2+e_4}(t)$. As above, since it agrees with $\theta$ on these Chevalley generators of $\bO_8^+$, we conclude that $\Inn(\varphi)$ restricted to $\bO_8^+$ is $\theta$. 

\noindent\ref{theta_description_iii}: The claim about $\PGO_8^+$ follows because $x'_{\alpha}(t) = \pi_{\bO}(x_\alpha(t))$, and so in $\PGO_8$ we will have
\begin{align*}
\Inn(\pi_{\bO}(\varphi))(x'_\alpha(t)) &= \pi_{\bO}(\varphi) \circ x'_\alpha(t) \circ \pi_{\bO}(\varphi)^{-1} = \pi_{\bO}(\varphi \circ x_\alpha(t) \circ \varphi^{-1})\\
&= \pi_{\bO}(\theta(x_\alpha(t))) = \theta(x'_\alpha(t)),
\end{align*}
which finishes the proof.
\end{proof}

\subsection{Clifford Triality}\label{Clifford_Triality}
Let $\cV = \cO^4$ and consider the regular hyperbolic quadratic form $(\HH(\cV),q_8)$ as in \Cref{split_Clifford}. Notationally, let $(\Mat_8(\cO),\sigma_8,f_8) \cong (\cEnd_{\cO}(\HH(\cV)),\sigma_{q_8},f_{q_8})$ be the adjoint quadratic triple with the isomorphism coming from our standard choice of ordered basis for $\HH(\cV)$. We have the isomorphism
\[
\Phi_0 \colon (\Cl_0(\HH(\cV),q_8),\underline{\sigma}_0,\underline{f}) \iso (\cEnd_{\cO}(\wedge_0 \cV),\sigma_{\wedge 0},f_{\wedge 0})\times (\cEnd_{\cO}(\wedge_1 \cV),\sigma_{\wedge 1},f_{\wedge 1})
\]
of \Cref{eq_Phi_not}, where $(\sigma_{\wedge i},f_{\wedge i})$ are the adjoint quadratic pairs of the quadratic form $q_\wedge$ restricted to each component. We choose the following ordered bases of $\wedge_0\cV$ and $\wedge_1 \cV$.
\begin{align}
\{1,v_{\{1,2\}},v_{\{1,3\}},v_{\{1,4\}},-v_{\{2,3\}},v_{\{2,4\}},-v_{\{3,4\}},v_{[4]}\} &\text{ of } \wedge_0 \cV, \text{ and} \label{eq_V_bases}\\
\{-v_{\{2,3,4\}}, v_2,v_3,v_4,-v_{\{1,2,3\}},v_{\{1,2,4\}},-v_{\{1,3,4\}},-v_1\} &\text{ of } \wedge_1 \cV. \nonumber
\end{align}
These choices provide isomorphisms of quadratic modules $(\wedge_i \cV,q_{\wedge i}) \cong (\HH(\cV),q_8)$. In turn, we have isomorphisms $(\cEnd_{\cO}(\wedge_i \cV),\sigma_{\wedge i},f_{\wedge i}) \cong (\Mat_8(\cO),\sigma_8,f_8)$. Therefore, by applying these isomorphisms and our standard isomorphism $(\Mat_8(\cO),\sigma_8,f_8) \cong (\cEnd_{\cO}(\HH(\cV)),\sigma_{q_8},f_{q_8})$, the isomorphism $\Phi_0$ becomes an isomorphism
\begin{equation}\label{eq_Clifford_triality_iso}
\Psi \colon (\Cl(\Mat_8(\cO),\sigma_8,f_8),\underline{\sigma_8},\underline{f_8}) \iso (\Mat_8(\cO),\sigma_8,f_8)\times (\Mat_8(\cO),\sigma_8,f_8).
\end{equation}

We now analyze the homomorphism
\[
\PGO_8^+ \inj \PGO_8 \xrightarrow{\Cl} \cAut(\Cl(\Mat_8(\cO),\sigma_8,f_8),\underline{\sigma_8},\underline{f_8})
\]
where $\Cl$ is the map of \Cref{eq_action_on_canonical}. By definition of $\PGO_8^+$, the image of this homomorphism will consist of automorphisms which fix the center of the Clifford algebra. In our case, the isomorphism $\Psi$ makes it clear that this center is $\cO\times \cO$. Any automorphism which fixes the center then also stabilizes each factor of the decomposition $(\Mat_8(\cO),\sigma_8,f_8)\times (\Mat_8(\cO),\sigma_8,f_8)$. Therefore, we get a homomorphism
\[
\PGO_8^+ \to \PGO_8 \times \PGO_8.
\]
We describe this homomorphism explicitly using Chevalley generators. Recall that by \Cref{action_on_Clifford}, the image under $\Cl$ of any $\varphi \in \PGO_8^+$ is the inner automorphism of the Clifford algebra given by any element $X\in \Spin_8$ lying above $\varphi$. Therefore, we compute the images of the Chevalley generators of $\Spin_8$ from \Cref{Chevalley_generators} under the homomorphism $\Psi$.
\begin{lem}\label{Spin_Chevalley_computations}
The images of the Chevalley generators of $\Spin_8$ under the homomorphism $\Psi$ of \Cref{eq_Clifford_triality_iso} are as follows.
\begin{align*}
X_{e_1-e_2}(t) &\mapsto (x_{e_3+e_4}(-t),x_{-e_1-e_2}(t)) & X_{-e_1+e_2}(t) &\mapsto (x_{-e_3-e_4}(-t),x_{e_1+e_2}(t))\\
X_{e_1-e_3}(t) &\mapsto (x_{e_2+e_4}(t),x_{-e_1-e_3}(t)) & X_{-e_1+e_3}(t) &\mapsto (x_{-e_2-e_4}(t),x_{e_1+e_3}(t)) \\
X_{e_1-e_4}(t) &\mapsto (x_{e_2+e_3}(-t),x_{-e_1-e_4}(t)) & X_{-e_1+e_4}(t) &\mapsto (x_{-e_2-e_3}(-t),x_{e_1+e_4}(t))\\
X_{e_2-e_3}(t) &\mapsto (x_{e_2-e_3}(t),x_{e_2-e_3}(t)) & X_{-e_2+e_3}(t) &\mapsto (x_{-e_2+e_3}(t),x_{-e_2+e_3}(t)) \\
X_{e_2-e_4}(t) &\mapsto (x_{e_2-e_4}(t),x_{e_2-e_4}(t)) & X_{-e_2+e_4}(t) &\mapsto (x_{-e_2+e_4}(t),x_{-e_2+e_4}(t))\\
X_{e_3-e_4}(t) &\mapsto (x_{e_3-e_4}(t),x_{e_3-e_4}(t)) & X_{-e_3+e_4}(t) &\mapsto (x_{-e_3+e_4}(t),x_{-e_3+e_4}(t)) \\[2ex]
X_{e_1+e_2}(t) &\mapsto (x_{-e_1+e_2}(t),x_{-e_3-e_4}(-t)) & X_{-e_1-e_2}(t) &\mapsto (x_{e_1-e_2}(t),x_{e_3+e_4}(-t))\\
X_{e_1+e_3}(t) &\mapsto (x_{-e_1+e_3}(t),x_{-e_2-e_4}(t)) & X_{-e_1-e_3}(t) &\mapsto (x_{e_1-e_3}(t),x_{e_2+e_4}(t)) \\
X_{e_1+e_4}(t) &\mapsto (x_{-e_1+e_4}(t),x_{-e_2-e_3}(-t)) & X_{-e_1-e_4}(t) &\mapsto (x_{e_1-e_4}(t),x_{e_2+e_3}(-t))\\
X_{e_2+e_3}(t) &\mapsto (x_{-e_1-e_4}(-t),x_{e_1-e_4}(-t)) & X_{-e_2-e_3}(t) &\mapsto (x_{e_1+e_4}(-t),x_{-e_1+e_4}(-t))\\
X_{e_2+e_4}(t) &\mapsto (x_{-e_1-e_3}(t),x_{e_1-e_3}(t)) & X_{-e_2-e_4}(t) &\mapsto (x_{e_1+e_3}(t),x_{-e_1+e_3}(t)) \\
X_{e_3+e_4}(t) &\mapsto (x_{-e_1-e_2}(-t),x_{e_1-e_2}(-t)) & X_{-e_3-e_4}(t) &\mapsto (x_{e_1+e_2}(-t),x_{-e_1+e_2}(-t)) \\
\end{align*}
\end{lem}
\begin{proof}
We demonstrate the calculation of one image and the other calculations proceed similarly. Consider $X_{e_1-e_2}(t)= 1+t v_1 v_2^* \in \Spin_8$. The element $v_1 v_2^* \in \Cl_0(\HH(\cV),q_8)$ maps via $\Phi$ to $\ell_{v_1}\circ d_{v_2^*} \in \cEnd_{\cO}(\wedge \cV)$. We calculate how this endomorphism behaves on the chosen bases of \Cref{eq_V_bases}. It sends
\begin{align*}
-v_{\{2,3\}} &\mapsto -v_{\{1,3\}} & -v_{\{2,3,4\}} &\mapsto -v_{\{1,3,4\}} \\
v_{\{2,4\}} &\mapsto v_{\{1,4\}} & v_2 &\mapsto v_1 
\end{align*}
and maps all other bases elements to zero. Using the isomorphisms $\cEnd_{\cO}(\wedge_i \cV) \cong \Mat_8(\cO)$ coming from the ordered bases, this endomorphism becomes the matrices
\[
(-E_{3,\bar{4}} + E_{4,\bar{3}},E_{\bar{2},1}-E_{\bar{1},2}).
\]
Therefore,
\[
X_{e_1-e_2}(t) \mapsto (I+(-t)(E_{3,\bar{4}} - E_{4,\bar{3}}),I+t(E_{\bar{2},1}-E_{\bar{1},2})) = (x_{e_3+e_4}(-t),x_{-e_1-e_2}(t))
\]
as claimed above.
\end{proof}

\begin{cor}\label{triality_hom}
The homomorphism 
\[
\PGO_8 \xrightarrow{\Cl} \cAut(\Cl(\Mat_8(\cO),\sigma_8,f_8),\underline{\sigma_8},\underline{f_8}) \iso \cAut\big((\Mat_8(\cO),q_8,f_8)\times (\Mat_8(\cO),q_8,f_8)\big)
\]
restricts to the homomorphism
\begin{align*}
\Psi^+ \colon \PGO_8^+ &\to \PGO_8^+ \times \PGO_8^+ \\
\varphi &\mapsto (\theta^+(\varphi),\theta^-(\varphi)).
\end{align*}
\end{cor}
\begin{proof}
Since all of the images described in \Cref{Spin_Chevalley_computations} lie in $\bO_8^+ \times \bO_8^+$, all of $\Spin_8$ must also map into this group. Therefore, when we consider the inner automorphisms by these elements we see that the map between adjoint groups will behave as, for example,
\[
x'_{e_1-e_2}(t) \mapsto (x'_{e_3+e_4}(-t),x'_{-e_1-e_2}(t))
\]
and likewise on other Chevalley generators according to \Cref{Spin_Chevalley_computations}. Therefore it is clear we get a map $\PGO_8^+ \to \PGO_8^+ \times \PGO_8^+$ which by comparison of the images of the Chevalley generators with \Cref{theta+_description}, we see is $\varphi \mapsto (\theta^+(\varphi),\theta^-(\varphi))$ as claimed.
\end{proof}

\begin{rem}\label{choice_remark}
The reader will have noticed that our choice of ordered basis for $\wedge_0 \cV$ is a natural one with respect to which $q_{\wedge 0}$ is obviously hyperbolic. With this natural choice of basis, the automorphism $\theta^+$ appears in the first factor while performing computations, and so this natural choice of basis for $\wedge_0 \cV$ justifies the less natural choice of simple system for $D_4$. The simple system we use is equivalent (after the relabelling $(0,1,2,3)\to(1,3,2,4)$) to the simple system obtained in \cite[Fig. 6.1]{EK} in the context of triality computations for Lie algebras of type $D_4$.

If one uses the natural ordered basis
\[
\{v_1,v_2,v_3,v_4,-v_{\{1,2,3\}},v_{\{1,2,4\}},-v_{\{1,3,4\}},v_{\{2,3,4\}}\}
\]
of $\wedge_1 \cP$, the automorphism $\theta^-$ does not appear in the second factor. Instead, $\theta \circ \theta^-$ appears. This justifies why we use a less natural basis of $\wedge_1 \cV$.
\end{rem}

\begin{cor}\label{split_Clifford_switch_iso}
Consider the element $\pi_{\bO}(\varphi) \in \PGO_8$ of \Cref{theta_description}. Using our standard identification $\cEnd_{\cO}(\HH(\cV))\cong \Mat_8(\cO)$, we view this as an automorphism of $\Mat_8(\cO)$. Then, the following diagram commutes
\[
\begin{tikzcd}
\Cl(\Mat_8(\cO),\sigma_8,f_8) \arrow{r}{\Psi} \arrow{d}{\Cl(\pi_{\bO}(\varphi))} & \Mat_8(\cO)\times \Mat_8(\cO) \arrow{d}{\sw \circ (\pi_{\bO}(\varphi)\times \pi_{\bO}(\varphi))} \\
\Cl(\Mat_8(\cO),\sigma_8,f_8) \arrow{r}{\Psi} & \Mat_8(\cO)\times \Mat_8(\cO)
\end{tikzcd}
\]
where $\sw \colon \Mat_8(\cO)\times \Mat_8(\cO) \to \Mat_8(\cO)\times \Mat_8(\cO)$ is the switch isomorphism, $\sw(b,c)=(c,b)$.
\end{cor}
\begin{proof}
Since the elements $X_\alpha(1)-1$, where $X_\alpha(1)$ is a Chevalley generator of $\Spin_8$, generate $\Cl(\Mat_8(\cO),\sigma_8,f_8)$ as an algebra, it is sufficient to argue that 
\[
\Psi\big(\Cl(\pi_{\bO}(\varphi))(X_{\alpha}(1))\big) = (\sw \circ (\pi_{\bO}(\varphi) \times \pi_{\bO}(\varphi)) \circ \Psi)(X_{\alpha}(1))
\]
for all $\alpha \in D_4$. However, $\Cl(\pi_{\bO}(\varphi)) = \bC(\varphi)$ which by \Cref{theta_description}\ref{theta_description_i} acts on elements of $\Spin_8$ by $\theta \colon \Spin_8 \iso \Spin_8$. One can then inspect the table of \Cref{Spin_Chevalley_computations} to see that
\[
\Psi(\theta(X_{\alpha}(1))) = (\sw \circ (\theta \times \theta) \circ \Psi)(X_{\alpha}(1))
\]
holds for all $\alpha \in D_4$, which is equivalent to the desired formula above. This completes the proof.
\end{proof}

\subsection{Triality with Octonions}\label{without_octonions}
Here we compare our preceding development of triality to the usual method using octonions. First, we will give a brief outline of the development over a scheme in terms of a general symmetric composition algebra analogous to what is done in \cite[\S 35]{KMRT} over fields of characteristic not $2$, \cite[\S 4]{DQ21} over a field of characteristic $2$, or \cite[\S 3.2]{AG19} over any ring. The first two sources also describe triality for similitudes, however in this section we will only describe triality for adjoint groups. Then, by applying this machinery to the split para-Cayley algebra and carefully choosing a basis, we show that we obtain the same order $3$ outer automorphisms of $\PGO^+_8$ as above.

\subsubsection{Symmetric Composition Algebras}
To begin, we follow the definitions given in \cite[\S 34]{KMRT}, adapting them over schemes. Let $(\cS,n)$ be a finite locally free non-associative $\cO$--algebra with a regular multiplicative quadratic form $n\colon \cS \to \cO$. I.e., $n$ is regular and denoting the multiplication in $\cS$ by $x \star y$, we have $n(x\star y) = n(x)n(y)$. The quadratic form $n$ has an associated polar bilinear form $b_n(x,y) = n(x+y)-n(x)-n(y)$ and we call $(\cS,n)$ a \emph{symmetric composition algebra} if
\[
b_n(x\star y,z)=b_n(x,y \star z)
\]
for all appropriate sections $x,y,z\in \cS$. This is equivalent to $(\cS,n)$ satisfying the relations
\[
x\star (y\star x) = n(x)y = (x\star y)\star x
\]
for all appropriate $x,y\in \cS$.

Now, assume that $(\cS,n)$ is a symmetric composition algebra of constant rank $8$. Since symmetric composition algebras come with a quadratic form, we may discuss their Clifford algebra, $\Cl(\cS,n)$, and their even Clifford algebra, $\Cl_0(\cS,n)\cong \Cl(\cEnd_{\cO}(\cS),\sigma_n,f_n)$, where $(\sigma_n,f_n)$ is the quadratic pair adjoint to $n$. There is an isomorphism
\begin{align*}
\alpha \colon \Cl(\cS,n) &\iso \cEnd_{\cO}(\cS\oplus \cS) \\
x &\mapsto \begin{bmatrix} 0 & \ell_x \\ r_x & 0 \end{bmatrix}
\end{align*}
defined on the image of $\cS \to \Cl(\cS,n)$ and extended algebraically. Here, $\ell_x(y)=x\star y$ and $r_x(y) = y\star x$ are the left and right multiplication by $x$ operators respectively. This isomorphism restricts to the even Clifford algebra and the restriction respects canonical involutions, see \cite[3.10]{AG19} for the result over rings, giving an isomorphism
\[
\alpha_0 \colon (\Cl_0(\cS,n),\underline{\sigma_n}) \iso (\cEnd_{\cO}(\cS),\sigma_n)\times (\cEnd_{\cO}(\cS),\sigma_n).
\]
The following lemma is \cite[4.5]{DQ21} over a field of characteristic $2$. We adapt their proof.
\begin{lem}\label{oct_iso_respects_canonical}
The isomorphism $\alpha_0$ respects the canonical semi-trace on $\Cl_0(\cS,n)$. That is, we have an isomorphism of algebras with semi-trace
\[
\alpha_0 \colon (\Cl_0(\cS,n),\underline{\sigma_n},\underline{f_n})\iso (\cEnd_{\cO}(\cS),\sigma_n,f_n)\times (\cEnd_{\cO}(\cS),\sigma_n,f_n)
\]
\end{lem}
\begin{proof}
It is sufficient to show that this hold locally over a cover. Therefore, by considering a fine enough cover such that everything splits, we may assume that $\cS$ is free, $n$ is hyperbolic, and that $\underline{f_n} = \Trd_{\Cl_0(\cS,n)}(c(a)\und)$ for an element $a \in \cEnd_{\cO}(\cS)$ with $\Trd_{\cEnd_{\cO}(\cS)}(a)=1$. Let $\{v_1,\ldots,v_1^*\}$ be a hyperbolic basis of $\cS$ with respect to $n$. Under the isomorphism $\varphi_n \colon \cS\otimes_{\cO} \cS \iso \cEnd_{\cO}(\cS)$ of \Cref{eq_bilinear_iso}, $\varphi_n(v_1^*\otimes v_1)$ sends $v_1 \mapsto v_1$ and all other basis elements to zero, and therefore has trace $1$. We use the element $c(\varphi_n(v_1^*\otimes v_1)) \in \Cl(\cEnd_{\cO}(\cS),\sigma_n,f_n)$ to define $\underline{f_n}$. Viewing it as an element in $\Cl_0(\cS,n)$, it appears simply as $v_1^*v_1$.

The semi-trace $\underline{f_n}\circ \alpha_0^{-1}$ is the map $\Trd_{\cEnd_{\cO}(\cS)\times \cEnd_{\cO}(\cS)}(\alpha_0(v_1^*v_1)\und)$. We compute that
\begin{align*}
\alpha_0(v_1^*v_1) &= \begin{bmatrix} 0 & \ell_{v_1^*} \\ r_{v_1^*} & 0 \end{bmatrix}\begin{bmatrix} 0 & \ell_{v_1} \\ r_{v_1} & 0 \end{bmatrix} \\
&= \begin{bmatrix} \ell_{v_1^*}r_{v_1} & 0 \\ 0 & r_{v_1^*}\ell_{v_1} \end{bmatrix}.
\end{align*}
Hence, we only need to check that $\Trd_{\cEnd_{\cO}(\cS)}(\ell_{v_1^*}r_{v_1}\und)$ and $\Trd_{\cEnd_{\cO}(\cS)}(r_{v_1^*}\ell_{v_1}\und)$ agree with $f_n$. Since the image of any semi-trace is determined on symmetrized elements, it is sufficient to consider symmetric elements of the form $\varphi_n(v_i\otimes v_i)$ and $\varphi_n(v_i^*\otimes v_i^*)$. We demonstrate these computations in the first case. Let $v$ be any basis vector.
\begin{align*}
\Trd_{\cEnd_{\cO}(\cS)}(\ell_{v_1^*}r_{v_1}\varphi_n(v\otimes v)) &= \Trd_{\cEnd_{\cO}(\cS)}(\varphi_n(v\otimes v)\ell_{v_1^*}r_{v_1}) \\
&= \Trd_{\cEnd_{\cO}(\cS)}(b_n(v,\und)v \circ \ell_{v_1^*}\circ r_{v_1}) \\
&= \Trd_{\cEnd_{\cO}(\cS)}(b_n(v,v_1^*\star(\und\star r_{v_1}))v) \\
&= b_n(v,v_1^*\star(v\star v_1))\\
&= b_n(v\star v_1^*, v\star v_1) \\
&= n(v\star v_1^* + v\star v_1) - n(v\star v_1^*) - n(v\star v_1) \\
&= n(v)\big(n(v_1^*+v_1)-n(v_1^*)-n(v_1)\big) \\
&= n(v)\cdot 1 \\
&= n(v)
\end{align*}
where we use the multiplicativity of $n$. Likewise, in the second case
\begin{align*}
\Trd_{\cEnd_{\cO}(\cS)}(r_{v_1^*}\ell_{v_1}\varphi_n(v\otimes v)) &= \Trd_{\cEnd_{\cO}(\cS)}(b_n(v,\und)v \circ r_{v_1^*}\circ \ell_{v_1})\\
&= \Trd_{\cEnd_{\cO}(\cS)}(b_n(v,(v_1\star \und)\star v_1^*)v) \\
&= b_n(v,(v_1\star v)\star v_1^*) \\
&= b_n((v_1\star v)\star v_1^*,v) \\
&= b_n(v_1\star v,v_1^*\star v) \\
&= b_n(v_1,v_1^*)n(v) \\
&= n(v).
\end{align*}
Since these agree with the defining property of $f_n$, we conclude that $\alpha_0$ respects the semi-traces as claimed.
\end{proof}

Triality then arises from the isomorphism $\alpha_0$ as follows. For an automorphism $\varphi \in \PGO^+(\cEnd_{\cO}(\cS),\sigma_n,f_n)$ we obtain an automorphism $\alpha_0 \circ \Cl(\varphi) \circ \alpha_0^{-1}$ of $\cEnd_{\cO}(\cS)\times \cEnd_{\cO}(\cS)$ which fixes the center and hence stabilizes the factors. Therefore, we have a commutative diagram
\[
\begin{tikzcd}
(\Cl_0(\cS,n),\underline{\sigma_n},\underline{f_n}) \ar[r,"\alpha_0"] \ar[d,"\Cl(\varphi)"] & (\cEnd_{\cO}(\cS),\sigma_n,f_n)\times (\cEnd_{\cO}(\cS),\sigma_n,f_n) \ar[d,"\varphi_2 \times \varphi_3"] \\
(\Cl_0(\cS,n),\underline{\sigma_n},\underline{f_n}) \ar[r,"\alpha_0"] & (\cEnd_{\cO}(\cS),\sigma_n,f_n)\times (\cEnd_{\cO}(\cS),\sigma_n,f_n)
\end{tikzcd}
\]
where $\varphi_2,\varphi_3 \in \PGO^+(\cEnd_{\cO}(\cS),\sigma_n,f_n)$. Defining $\theta\oct^+(\varphi) = \varphi_2$ and $\theta\oct^-(\varphi)=\varphi_3$ produces order three outer automorphisms of $\PGO^+(\cEnd_{\cO}(\cS),\sigma_n,f_n)$ with $(\theta\oct^+)^2 = \theta\oct^-$. In the following section we will consider Chevalley generators to identify these automorphisms and align them with the automorphisms $\theta^+$ and $\theta^-$ defined in \Cref{automorphisms_of_groups} which arise in \Cref{triality_hom}.

\subsubsection{The Split para-Cayley Algebra}
The split Cayley algebra, i.e., the split octonions, can be produced from the matrix algebra $\Mat_2(\cO)$ via the Cayley-Dickson construction. In particular, the split Cayley algebra is $\cC_s = \mathrm{CD}(\Mat_2(\cO),-1)$. The construction goes as follow, from \cite[\S 33.C]{KMRT}.

The matrix algebra $\Mat_2(\cO)$, being a quaternion algebra, has its canonical symplectic involution
\[
\psi\left(\begin{bmatrix} a & b \\ c & d \end{bmatrix}\right) = \begin{bmatrix} d & -b \\ -c & a \end{bmatrix}.
\]
This involution satisfies the identities
\[
a+ \psi(a) = \Trd_{\Mat_2(\cO)}(a)\cdot I \text{, and } a\cdot \psi(a) = \det(a)\cdot I
\]
for all $a\in \Mat_2(\cO)$. This gives $\Mat_2(\cO)$ the structure of a Hurwitz algebra. Then, introducing a symbol $\nu$,
\[
\cC_s = \mathrm{CD}(\Mat_2(\cO),-1) = \Mat_2(\cO)\oplus \nu\Mat_2(\cO)
\]
with component wise addition and with multiplication given by
\[
(a+\nu b)(a'+\nu b') = aa' - b'\psi(b) + \nu \big(\psi(a)b' + a'b\big).
\]
This comes with an involution $\overline{a+\nu b} = \psi(a)-\nu b$ and a hyperbolic quadratic form $n(a+\nu b) = \det(a)+\det(b)$. These satisfy $x\cdot \overline{x} = n(x)$ for $x\in \cC_s$. The split Cayley algebra has what is called a \emph{good basis} in \cite{EK}, for example comprised of the elements
\begin{align*}
e_1 &= \begin{bmatrix} 1 & 0 \\ 0 & 0 \end{bmatrix} & u_1 &= \begin{bmatrix} 0 & 1 \\ 0 & 0 \end{bmatrix} & u_2 &= \nu\begin{bmatrix} 0 & 0 \\ 0 & -1 \end{bmatrix} & u_3 &= \nu\begin{bmatrix} 0 & 0 \\ 1 & 0 \end{bmatrix} \\
e_2 &= \begin{bmatrix} 0 & 0 \\ 0 & 1 \end{bmatrix} & v_1 &= \begin{bmatrix} 0 & 0 \\ -1 & 0 \end{bmatrix} & v_2 &= \nu\begin{bmatrix} 1 & 0 \\ 0 & 0 \end{bmatrix} & v_3 &= \nu\begin{bmatrix} 0 & -1 \\ 0 & 0 \end{bmatrix}.
\end{align*}
The multiplication table for $\cC_s$ with respect to this basis is given in \cite[Figure 4.1]{EK}. The involution takes the form $\overline{e_1}=e_2$, $\overline{u_i}=-u_i$, and $\overline{v_i}=-v_i$. Defining the para-multiplication by
\[
x \star y = \overline{x} \cdot \overline{y}
\]
for $x,y\in \cC_s$ makes $(\cC_s,\star,n)$ a symmetric composition algebra of rank $8$, called the \emph{split para-Cayley algebra}. By the previous section, we then have an isomorphism
\[
\alpha_0 \colon (\Cl_0(\cC_s,n),\underline{\sigma_n},\underline{f_n})\iso (\cEnd_{\cO}(\cC_s),\sigma_n,f_n)\times (\cEnd_{\cO}(\cC_s),\sigma_n,f_n).
\]

In order to align the triality which arises from $\alpha_0$ with the preceding results of this paper, we identify $\PGO^+(\cEnd_{\cO}(\cC_s),\sigma_n,f_n)$ with $\PGO^+_8$ by fixing the ordered hyperbolic basis
\[
\{e_1,-v_2,-v_1,-v_3,-u_3,-u_1,-u_2,e_2\}
\]
of $\cC_s$. This also allows us to identify $\Cl_0(\cC_s,n)$ to $\Cl_0(\HH(\cV),q_8)$ of \Cref{split_Clifford} and therefore speak of the Chevalley generators of $\Spin_8$ inside $\Cl_0(\cC_s,n)$ or of Chevalley generators of $\bO^+_8$ inside $\cEnd_{\cO}(\cC_s)$. Now, we may trace the images of these Chevalley generators through the isomorphism $\alpha_0$ in order to identify the resulting automorphism $\theta\oct^+$.
\begin{lem}\label{Spin_Chevalley_computations_octonions}
The images of the Chevalley generators of $\Spin_8$ in $\Cl_0(\cC_s,n)$ under the isomorphism $\alpha_0$ agree with the images stated in \Cref{Spin_Chevalley_computations}.
\end{lem}
\begin{proof}
This result is purely computational. We demonstrate one such computation for the first generator, the others proceed analogously.

The $\Spin_8$ generator $X_{e_1-e_2}(t)=1+tv_1v_2^* \in \Cl_0(\HH(\cV),q_8)$ takes the form
\[
1+te_1(-u_2) = 1-te_1u_2
\]
in $\Cl_0(\cC_s,n)$ and we have that
\[
\alpha_0(e_1u_2) = \begin{bmatrix} \ell_{e_1} \circ r_{u_2} & 0 \\ 0 & r_{e_1} \circ \ell_{u_2} \end{bmatrix}.
\]
We then compute, using the para-multiplication on $\cC_s$, that the linear map $\ell_{e_1} \circ r_{u_2}$ sends
\[
-u_3 \mapsto -v_1 \text{ and } -u_1 \mapsto v_3
\]
and all other basis vectors to zero. Using the ordered basis, this is the matrix $E_{3,\overline{4}}-E_{4,\overline{3}} \in \cEnd_{\cO}(\cC_s)$. Similarly, the linear map $r_{e_1} \circ \ell_{u_2}$ sends
\[
e_1 \mapsto u_2 \text{ and } -v_2 \mapsto e_2
\]
and all other basis vectors to zero. This corresponds to the matrix $-E_{\overline{2},1}+E_{\overline{1},2}$.

Therefore,
\begin{align*}
\alpha_0(X_{e_1-e_2}(t)) &= (I-t(E_{3,\overline{4}}-E_{4,\overline{3}}),I-t(-E_{\overline{2},1}+E_{\overline{1},2})) \\
&= (I + (-t)(E_{3,\overline{4}}-E_{4,\overline{3}}),I+t(E_{\overline{2},1}-E_{\overline{1},2})) \\
&= (x_{e_3+e_4}(-t),x_{-e_1-e_2}(t))
\end{align*}
as claimed in \Cref{Spin_Chevalley_computations}.
\end{proof}

\begin{cor}
Identifying $\PGO^+_8 \cong \PGO^+(\cEnd_{\cO}(\cC_s),\sigma_n,f_n)$ as above, we have that $\theta\oct^+ = \theta^+$ and therefore also $\theta\oct^- = \theta^-$. 
\end{cor}
\begin{proof}
Since the computations in \Cref{Spin_Chevalley_computations_octonions} agree with those in \Cref{Spin_Chevalley_computations}, the group homomorphism
\begin{align*}
\PGO^+_8 &\to \PGO^+_8 \times \PGO^+_8 \\
\varphi &\mapsto (\theta\oct^+(\varphi),\theta\oct^-(\varphi))
\end{align*}
arising from $\alpha_0$ will agree with the map $\Psi^+$ of \Cref{triality_hom}. Hence $\theta\oct^+ = \theta^+$ as claimed.
\end{proof}

\section{Trialitarian Triples}\label{Triples}
Following \cite[42.A]{KMRT}, we define a groupoid of trialitarian triples which is equivalent to the category of $\PGO_8^+$--torsors. Our definition of trialitarian triple, while given over schemes, is essentially identical to the definition given in \cite[\S 4]{BT23} where the authors use \cite{DQ21} in order to define trialitarian triples over fields of any characteristic. We assemble these groupoids of trialitarian triples into a stack equivalent to the stack of $\PGO_8^+$--torsors. Throughout this section, we abbreviate a quadratic triple using the notation $(\cA,\sigma_{\cA},f_{\cA}) = \cA_\qt$.

\subsection{The Groupoid and the Gerbe}
\begin{defn}\label{defn_Trip(S)}
 Let $\fTrip(S)$ be the category whose
\begin{enumerate}[label={\rm(\roman*)}]
\item \label{defn_Trip(S)_objects} objects are quadruples $(\cA_\qt,\cB_\qt, \cC_\qt ,\alpha_{\cA})$ where the first three terms are each quadratic triples of degree $8$ over $S$ and the last term is an isomorphism
\[
\alpha_{\cA} \colon (\Cl(\cA_\qt),\underline{\sigma_{\cA}},\underline{f_{\cA}}) \iso \cB_\qt \times \cC_\qt
\]
of $\cO$--algebras with semi-trace, and whose
\item \label{defn_Trip(S)_morphisms} morphisms are triples
\[
(\phi_1,\phi_2,\phi_3)\colon (\cA_\qt,\cB_\qt,\cC_\qt,\alpha_{\cA}) \to (\cA'_\qt,\cB'_\qt,\cC'_\qt,\alpha_{\cA'})
\]
where $\phi_1 \colon \cA_\qt \iso \cA'_\qt$, $\phi_2 \colon \cB_\qt \iso \cB'_\qt$, and $\phi_3 \colon \cC_\qt \iso \cC'_\qt$ are isomorphisms of quadratic triples such that the diagram
\[
\begin{tikzcd}
(\Cl(\cA_\qt),\underline{\sigma_{\cA}},\underline{f_{\cA}}) \arrow{d}{\Cl(\phi_1)} \arrow{r}{\alpha_{\cA}} & \cB_\qt \times \cC_\qt \arrow{d}{\phi_2\times \phi_3} \\
(\Cl(\cA'_\qt),\underline{\sigma_{\cA'}},\underline{f_{\cA'}}) \arrow{r}{\alpha_{\cA'}} & \cB'_\qt \times \cC'_\qt
\end{tikzcd}
\]
commutes.
\end{enumerate}
\end{defn}
If $(\cA_\qt,\cB_\qt,\cC_\qt,\alpha_{\cA}) \in \fTrip(S)$, then clearly $\Cl(\cA_\qt)$ is a Cartesian product of two Azumaya algebras and therefore the involution $\sigma_{\cA}$ of $\cA_\qt$ is of trivial discriminant. Let $\Mat_8(\cO)_\qt = (\Mat_8(\cO),\sigma_8,f_8)$ be the adjoint quadratic triple of the regular quadratic form $(\HH(\cV),q_8)$ of \Cref{split_Clifford}. The quadruple
\[
(\Mat_8(\cO)_\qt,\Mat_8(\cO)_\qt,\Mat_8(\cO)_\qt,\Psi),
\]
where $\Psi$ is the isomorphism of \Cref{eq_Clifford_triality_iso}, is an object of $\fTrip(S)$. We call it the \emph{split object} and denote it by $\cM_\qt^3$.

We may similarly define a groupoid $\fTrip(T)$ for any $T\in \Sch_S$, and these groupoids can be assembled into the fibers of a stack.
\begin{defn}\label{defn_Trip}
Let $\fTrip$ be the following category.
\begin{enumerate}[label={\rm(\roman*)}]
\item \label{defn_Trip_objects} The objects are pairs $(T,(\cA_\qt,\cB_\qt,\cC_\qt,\alpha_{\cA}))$ consisting of $T\in \Sch_S$ and $(\cA_\qt,\cB_\qt,\cC_\qt,\alpha_{\cA}) \in \fTrip(T)$.
\item \label{defn_Trip_morphisms} The morphisms are pairs 
\[
(g,\phi) \colon (T',(\cA'_\qt,\cB'_\qt,\cC'_\qt,\alpha_{\cA'})) \to (T,(\cA_\qt,\cB_\qt,\cC_\qt,\alpha_{\cA})) 
\]
where $g\colon T' \iso T$ is an isomorphism of $S$--schemes and
\[
\phi \colon (\cA'_\qt,\cB'_\qt,\cC'_\qt,\alpha_{\cA'}) \to (g^*(\cA_\qt),g^*(\cB_\qt),g^*(\cC_\qt),g^*(\alpha_{\cA}))
\]
is an isomorphism in $\fTrip(T')$. Here we use the identification
\[
(\Cl(g^*(\cA_\qt)),\underline{g^*(\sigma_{\cA})},\underline{g^*(f_{\cA})}) = g^*(\Cl(\cA_\qt),\underline{\sigma_{\cA}},\underline{f_{\cA}})
\]
to consider $g^*(\alpha_{\cA})$ as a map with the appropriate domain. Composition is given by $(g_1,\phi_1)\circ (g_2,\phi_2)=(g_1\circ g_2, g_2^*(\phi_1)\circ \phi_2)$.

\item \label{defn_Trip_structure} We equip $\fTrip$ with a structure functor $p \colon \fTrip \to \Sch_S$ sending
\[
(T,(\cA_\qt,\cB_\qt,\cC_\qt,\alpha_{\cA}))\mapsto T \text{ and } (g,\phi)\mapsto g
\]
on objects and morphisms respectively.
\end{enumerate}
\end{defn}
Given a morphism $g\colon T' \to T$ of $S$--schemes and $(T,(\cA_\qt,\cB_\qt,\cC_\qt,\alpha_{\cA})) \in \fTrip$, the morphism
\[
(g,\Id) \colon (T',(\cA_\qt|_{T'},\cB_\qt|_{T'},\cC_\qt|_{T'},\alpha_{\cA}|_{T'})) \to (T,(\cA_\qt,\cB_\qt,\cC_\qt,\alpha_{\cA}))
\]
is cartesian and makes the object $(T',(\cA_\qt|_{T'},\cB_\qt|_{T'},\cC_\qt|_{T'},\alpha_{\cA}|_{T'}))$ a pullback of $(T,(\cA_\qt,\cB_\qt,\cC_\qt,\alpha_{\cA}))$. We will use the notation $(T,(\cA_\qt,\cB_\qt,\cC_\qt,\alpha_{\cA}))|_{T'}$ to refer to this specific pullback. Therefore, the structure functor $p\colon \fTrip \to \Sch_S$ makes $\fTrip$ a fibered category whose fiber over $T$ is the groupoid $\fTrip(T)$.

\begin{lem}\label{Trip_properties}
Consider the fibered category $p \colon \fTrip \to \Sch_S$ of \Cref{defn_Trip}.
\begin{enumerate}[label={\rm(\roman*)}]
\item \label{Trip_properties_i} $p \colon \fTrip \to \Sch_S$ is a gerbe.
\item \label{Trip_properties_ii} The automorphism sheaf of the split object $(S,\cM_\qt^3)\in \fTrip(S)$ is
\[
\bAut(S,\cM_\qt^3) \cong \PGO_8^+.
\]
\end{enumerate}
\end{lem}
\begin{proof}
\noindent\ref{Trip_properties_i}: First, we must argue that $p\colon \fTrip \to \Sch_S$ is a stack. To see that the $\cHom$ functors are sheaves, consider two objects in the same fiber,
\[
(T,(\cA_\qt,\cB_\qt,\cC_\qt,\alpha_{\cA})),(T,(\cA'_\qt,\cB'_\qt,\cC'_\qt,\alpha_{\cA'})) \in \fTrip(T)
\]
for some $T\in \Sch_S$. The functor $\cIsom(\cA_\qt,\cA'_\qt) \colon \Sch_T \to \Sets$ of internal isomorphisms between $\cA_\qt$ and $\cA'_\qt$ is a sheaf. Likewise $\cIsom(\cB_\qt,\cB'_\qt)$ and $\cIsom(\cC_\qt, \cC'_\qt)$ are sheaves. Consider the two morphisms
\begin{align*}
\cIsom(\cA_\qt,\cA'_\qt)\times\cIsom(\cB_\qt,\cB'_\qt)\times\cIsom(\cC_\qt,\cC'_\qt) &\to \cIsom(\Cl(\cA_\qt),\Cl(\cA'_\qt)) \\
\rho_1\colon (\phi_1,\phi_2,\phi_3) &\mapsto \Cl(\phi_1) \\
\rho_2\colon (\phi_1,\phi_2,\phi_3) &\mapsto \alpha_{\cA'}^{-1} \circ (\phi_2\times \phi_3) \circ \alpha_{\cA}.
\end{align*}
The functor $\cHom((T,(\cA_\qt,\cB_\qt,\cC_\qt,\alpha_{\cA})),(T,(\cA'_\qt,\cB'_\qt,\cC'_\qt,\alpha_{\cA'})))$ is the equalizer of the two morphisms $\rho_1$ and $\rho_2$ between sheaves, and is therefore a sheaf itself as desired.

To see that $\fTrip$ allows gluing of objects, let $T\in \Sch_S$ and let $\{T_i \to T\}_{i\in I}$ be a cover. Consider objects $(T_i,(\cA_{i\qt},\cB_{i\qt},\cC_{i\qt},\alpha_{\cA_{i}})) \in \fTrip$ with isomorphisms
{\small
\[
\rho_{ij} \colon (T_{ij},(\cA_{i\qt}|_{T_{ij}},\cB_{i\qt}|_{T_{ij}},\cC_{i\qt}|_{T_{ij}},\alpha_{\cA_{i}}|_{T_{ij}})) \overset{\sim}{\to} (T_{ij},(\cA_{j\qt}|_{T_{ij}},\cB_{j\qt}|_{T_{ij}},\cC_{j\qt}|_{T_{ij}},\alpha_{\cA_j}|_{T_{ij}}))
\]
}
satisfying $\rho_{jk}\circ \rho_{ij} = \rho_{ik}$ over $T_{ijk}$. First consider the quadratic triples $\cA_{i\qt}$. The isomorphisms $\rho_{ij}$ contain isomorphisms $(\phi_1)_{ij} \colon \cA_{i\qt}|_{T_{ij}} \iso \cA_{j\qt}|_{T_{ij}}$ satisfying $(\phi_1)_{jk} \circ (\phi_1)_{ij} = (\phi_1)_{ik}$ over $T_{ijk}$. Therefore, since quadratic triples allow gluing, there exists a quadratic triple $\cA_\qt$ over $T$ and isomorphisms $\gamma_{\cA,i} \colon \cA_{i\qt} \iso \cA_\qt|_{T_i}$ such that $(\gamma_{\cA,j})|_{T_{ij}}^{-1}\circ \gamma_{\cA,i}|_{T_{ij}} = (\phi_1)_{ij}$. Similarly, there will be quadratic triples $\cB_\qt$ and $\cC_\qt$ over $T$ with isomorphisms $\gamma_{\cB,i} \colon \cB_{i\qt} \iso \cB_\qt|_{T_i}$ and $\gamma_{\cC,i} \colon \cC_{i\qt} \iso \cC_\qt|_{T_i}$. The isomorphisms $\alpha_{\cA_i}$ are also compatible on overlaps via the isomorphisms $\rho_{ij}$ by assumption, and therefore they glue into an isomorphism $\alpha_{\cA} \colon (\Cl(\cA_\qt),\underline{\sigma_{\cA}},\underline{f_{\cA}}) \iso \cB_\qt \times \cC_\qt$ which makes $(T,(\cA_\qt,\cB_\qt,\cC_\qt,\alpha_{\cA}))$ an object in $\fTrip(T)$. Locally there will be isomorphisms
{\small
\[
\gamma_i = (\Id,(\gamma_{\cA,i},\gamma_{\cB,i},\gamma_{\cC,i})) \colon (T_i,(\cA_{i\qt},\cB_{i\qt},\cC_{i\qt},\alpha_{\cA_i})) \overset{\sim}{\to} (T_i,(\cA_\qt|_{T_i},\cB_\qt|_{T_i},\cC_\qt|_{T_i},\alpha_{\cA}|_{T_i}))
\]
}
which by construction satisfy the relation $(\gamma_j)|_{T_{ij}}^{-1} \circ \gamma_i|_{T_{ij}} = \rho_{ij}$. Hence, we see that $(T,(\cA_\qt,\cB_\qt,\cC_\qt,\alpha_{\cA}))$ is a glued object as required and so $\fTrip$ is a stack.

Towards showing that $\fTrip$ is a gerbe, since $\fTrip(S)$ is non-empty it is clear that all fibers are non-empty and so every object trivially has a cover over which the fibers are non-empty. It is also clear from construction that $\fTrip$ is fibered in groupoids. Hence, we must only show that all objects in a fiber are twisted forms of one another. We show that for any $T\in \Sch_S$, every object $(T,(\cA_\qt,\cB_\qt,\cC_\qt,\alpha_{\cA})) \in \fTrip(T)$ is a twisted form of the split object $(T,(\Mat_8(\cO)_\qt|_T,\Mat_8(\cO)_\qt|_T,\Mat_8(\cO)_\qt|_T,\Psi|_T)) = (T,\cM_\qt^3|_T) \in \fTrip(T)$. Since quadratic triples are locally split, there exists a cover $\{T_i \to T\}_{i\in I}$ such that for each $i\in I$ there is an isomorphism $\phi_{1,i} \colon \cA_\qt|_{T_i} \iso \Mat_8(\cO)_\qt|_{T_i}$. We consider the diagram
\[
\begin{tikzcd}
(\Cl(\cA_\qt),\underline{\sigma_{\cA}},\underline{f_{\cA}})|_{T_i} \arrow{d}{\Cl(\phi_{1,i})} \arrow{r}{\alpha_{\cA|_{T_i}}} & \cB_\qt|_{T_i} \times \cC_\qt|_{T_i} \arrow[dashed]{d} \\
(\Cl(\Mat_8(\cO)_\qt),\underline{\sigma_8},\underline{f_8})|_{T_i} \arrow{r}{\Psi|_{T_i}} & \Mat_8(\cO)_\qt|_{T_i} \times \Mat_8(\cO)_\qt|_{T_i}.
\end{tikzcd}
\]
There are two cases. If $\Cl(\phi_{1,i})(\alpha_{\cA}|_{T_i}^{-1}(1,0)) = \Psi|_{T_i}^{-1}(1,0)$ and likewise for $(0,1)$, then defining $\phi_{2,i} \colon \cB_\qt|_{T_i} \to \Mat_8(\cO)_\qt|_{T_i}$ by
\[
(\phi_{2,i}(b),0) = (\Psi \circ \Cl(\phi_{1,i}) \circ \alpha_{\cA}^{-1})(b,0)
\]
and similarly for $\phi_{3,i} \colon \cC_\qt|_{T_i} \to \Mat_8(\cO)_\qt|_{T_i}$ will define an isomorphism
\[
(\Id,(\phi_{1,i},\phi_{2,i},\phi_{3,i})) \colon (T_i,(\cA_\qt|_{T_i},\cB_\qt|_{T_i},\cC_\qt|_{T_i},\alpha_{\cA}|_{T_i})) \iso (T_i,\cM_\qt^3|_{T_i})
\]
in $\fTrip(T_i)$ as desired. However, in the second case we will have
\begin{align*}
\Cl(\phi_{1,i})(\alpha_{\cA}|_{T_i}^{-1}(1,0)) &= \Psi|_{T_i}^{-1}(0,1) \text{, and} \\
\Cl(\phi_{1,i})(\alpha_{\cA}|_{T_i}^{-1}(0,1)) &= \Psi|_{T_i}^{-1}(1,0).
\end{align*}
Here, we may use the improper isometry $\varphi \in \bO_8$ with $\varphi \notin \bO_8^+$ of \Cref{theta_description} and instead consider the isomorphism $\pi_{\bO}(\varphi)|_{T_i}\circ \phi_{1,i} \colon \cA_\qt|_{T_i} \to \Mat_8(\cO)_\qt|_{T_i}$. Since $\varphi$ is improper, the isomorphism $\bC(\varphi) = \Cl(\pi_{\bO}(\varphi))$ will swap the factors in the Clifford algebra and therefore return us to the first case. We then proceed as above to finish the proof and conclude that $\fTrip$ is a gerbe.

\noindent\ref{Trip_properties_ii}: For $T\in \Sch_S$, a section of $\bAut(S,\cM_\qt^3)(T)$ is an automorphism of $\cM_\qt^3|_T$ in $\fTrip(T)$. Let $(\phi_1,\phi_2,\phi_3) \in \Aut_{\fTrip(T)}(\cM_\qt^3|_T)$. Since the diagram
\[
\begin{tikzcd}
(\Cl(\Mat_8(\cO)_\qt),\underline{\sigma_8},\underline{f_8})|_T \arrow{r}{\Psi|_T} \arrow{d}{\Cl(\phi_1)} & \Mat_8(\cO)_\qt|_T \times \Mat_8(\cO)_\qt|_T \arrow{d}{\phi_2\times \phi_3} \\
(\Cl(\Mat_8(\cO)_\qt),\underline{\sigma_8},\underline{f_8})|_T \arrow{r}{\Psi|_T} & \Mat_8(\cO)_\qt|_T \times \Mat_8(\cO)_\qt|_T
\end{tikzcd}
\]
commutes, we know from \Cref{triality_hom} that we must have $\phi_2 = \theta^+(\phi_1)$ and $\phi_3 = \theta^-(\phi_1)$. Therefore, the isomorphism is completely determined by the choice of $\phi_1 \in \PGO_{8}(T)$, which must be an isomorphism such that $\Cl(\phi_1)$ preserves the order of the factors in the Clifford algebra. This is the definition of $\PGO_8^+$ and therefore gives a clear isomorphism
\begin{align*}
\bAut(S,\cM_\qt^3)(T) &\iso \PGO_8^+(T) \\
(\phi_1,\phi_2,\phi_3) &\mapsto \phi_1.
\end{align*}
These isomorphisms clearly assemble into an isomorphism of sheaves
\[
\bAut(S,\cM_\qt^3) \iso \PGO_8^+
\]
as desired.
\end{proof}

\Cref{Trip_properties} means that $\fTrip$ is equivalent to the gerbe of $\PGO_8^+$--torsors. In particular, for every object $(\cA_\qt,\cB_\qt,\cC_\qt,\alpha_{\cA})\in \fTrip$, the sheaf of isomorphisms
\[
\cIsom(\cM_\qt^3,(\cA_\qt,\cB_\qt,\cC_\qt,\alpha_{\cA}))
\]
is a $\PGO_8^+$--torsor and this provides the equivalence of stacks. Looking globally, this means we have an equivalence of categories between the category of $\PGO_8^+$--torsors over $S$ and $\fTrip(S)$. In terms of cohomology, we view $H^1(S,\PGO_8^+)$ as the set of isomorphism classes in $\fTrip(S)$.

\subsection{Endofunctors of $\fTrip(S)$}\label{endo_Trip(S)}
We now define an action of the abstract group $\SS_3$ as endofunctors on the groupoid $\fTrip(S)$. Let $(\cA_\qt,\cB_\qt,\cC_\qt,\alpha_{\cA}) \in \fTrip(S)$. Notice that the isormorphism $\alpha_{\cA}$ is uniquely determined by the $\PGO_8^+$--torsor $\cP = \cIsom(\cM_\qt^3,(\cA_\qt,\cB_\qt,\cC_\qt,\alpha_{\cA}))$. Indeed, let $\{T_i \to S\}_{i \in I}$ be a cover over which there are sections $(\phi_{i1},\phi_{i2},\phi_{i3})\in \cP(T_i)$. Commutativity of the diagram
\[
\begin{tikzcd}
(\Cl(\Mat_8(\cO)_\qt),\underline{\sigma_8},\underline{f_8})|_{T_i} \arrow{r}{\Psi|_{T_i}} \arrow{d}{\Cl(\phi_{i1})} & \Mat_8(\cO|_{T_i})_\qt \times \Mat_8(\cO|_{T_i})_\qt \arrow{d}{\phi_{i2}\times \phi_{i3}} \\
(\Cl(\cA_\qt),\underline{\sigma_{\cA}},\underline{f_{\cA}})|_{T_i} \arrow{r}{\alpha_{\cA}|_{T_i}} & \cB_\qt|_{T_i} \times \cC_\qt|_{T_i} 
\end{tikzcd}
\]
means that
\begin{equation}\label{eq_alpha_A}
\alpha_{\cA}|_{T_i} = (\phi_{i2}\times \phi_{i3})\circ \Psi|_{T_i} \circ \Cl(\phi_{i1})^{-1}.
\end{equation}
Also, this formula does not depend on the choice of isomorphism $(\phi_{i1},\phi_{i2},\phi_{i3})$. We now consider the diagram
\[
\begin{tikzcd}
(\Cl(\Mat_8(\cO)_\qt),\underline{\sigma_8},\underline{f_8})|_{T_i} \arrow{r}{\Psi|_{T_i}} \arrow{d}{\Cl(\phi_{i2})} & \Mat_8(\cO|_{T_i})_\qt \times \Mat_8(\cO|_{T_i})_\qt \arrow{d}{\phi_{i3}\times \phi_{i1}} \\
(\Cl(\cB_\qt),\underline{\sigma_{\cB}},\underline{f_{\cB}})|_{T_i} \arrow[dashed]{r} & \cC_\qt|_{T_i} \times \cA_\qt|_{T_i} 
\end{tikzcd}
\]
and define $\alpha_{\cB}|_{T_i}$ to match the dashed map, in particular
\begin{equation}\label{eq_alpha_B}
\alpha_{\cB}|_{T_i} = (\phi_{i3}\times \phi_{i1})\circ \Psi|_{T_i} \circ \Cl(\phi_{i2})^{-1}.
\end{equation}
In the following lemma, we argue that these local definitions do not depend on the choice of isomorphism in $\cP(T_i)$ and that they glue into a well-defined isomorphism $\alpha_{\cB} \colon (\Cl(\cB_\qt),\underline{\sigma_{\cB}},\underline{f_{\cB}}) \iso \cC_\qt \times \cA_\qt$.

\begin{lem}\label{build_alpha_B}
Let $(\cA_\qt,\cB_\qt,\cC_\qt,\alpha_{\cA})$ be an object in $\fTrip(S)$ and set
\[
\cP=\cIsom(\cM_\qt^3,(\cA_\qt,\cB_\qt,\cC_\qt,\alpha_{\cA})).
\]
\begin{enumerate}[label={\rm(\roman*)}]
\item \label{build_alpha_B_i} For all $T \in \Sch_S$ such that there are isomorphisms $(\phi_1,\phi_2,\phi_3),(\phi_1',\phi_2',\phi_3') \in \cP(T)$, we have $\phi_i^{-1}\circ \phi_i' \in \PGO_8^+(T)$ for each $i=1,2,3$ and furthermore
\begin{align*}
\phi_2^{-1}\circ \phi_2'&=\theta^+(\phi_1^{-1}\circ \phi_1')\\
\phi_3^{-1}\circ \phi_3'&=\theta^-(\phi_1^{-1}\circ \phi_1').
\end{align*}
\item \label{build_alpha_B_ii} For all $T\in \Sch_S$ as in \ref{build_alpha_B_i}, the maps
\begin{align*}
&(\phi_3\times \phi_1)\circ \Psi|_T \circ \Cl(\phi_2)^{-1}, \text{ and}\\
&(\phi'_3\times \phi'_1)\circ \Psi|_T \circ \Cl(\phi'_2)^{-1}
\end{align*}
are equal.
\item \label{build_alpha_B_iii} There exists a unique isomorphism $\alpha_{\cB} \colon (\Cl(\cB_\qt),\underline{\sigma_{\cB}},\underline{f_{\cB}}) \iso \cC_\qt \times \cA_\qt$ such that for all $T\in \Sch_S$ as in \ref{build_alpha_B_i}, we have
\[
\alpha_{\cB}|_T = (\phi_3\times \phi_1)\circ \Psi|_T \circ \Cl(\phi_2)^{-1}.
\]
\end{enumerate}
\end{lem}
\begin{proof}
\noindent\ref{build_alpha_B_i}: Using \Cref{eq_alpha_A} with each isomorphism, and then taking inverses on the first equation, we have
\begin{align*}
\alpha_{\cA}^{-1} &= \Cl(\phi_1) \circ \Psi|_T^{-1} \circ (\phi_2^{-1} \times \phi_3^{-1}), \text{ and} \\
\alpha_{\cA} &= (\phi'_{2}\times \phi'_{3})\circ \Psi|_T \circ \Cl(\phi'_{1})^{-1}.
\end{align*}
This of course implies that
\[
\Id_{\Cl(\cA_\qt)} = \Cl(\phi_1) \circ \Psi|_T^{-1} \circ (\phi_2^{-1} \times \phi_3^{-1}) \circ (\phi'_{2}\times \phi'_{3})\circ \Psi|_T \circ \Cl(\phi'_{1})^{-1}
\]
which rearranges to
\begin{align*}
\Cl(\phi_1^{-1})\circ \Cl(\phi'_{1})&= \Psi|_T^{-1} \circ (\phi_2^{-1} \times \phi_3^{-1}) \circ (\phi'_{2}\times \phi'_{3})\circ \Psi|_T \\
\Leftrightarrow \Cl(\phi_1^{-1}\circ \phi'_{1})&= \Psi|_T^{-1} \circ ((\phi_2^{-1}\circ \phi_2') \times (\phi_3^{-1}\circ \phi_3')) \circ \Psi|_T.
\end{align*}
This shows that $\phi_1^{-1}\circ \phi_1'$ preserves the center of the Clifford algebra and therefore is an element of $\PGO_8^+$. The final two claimed formulas then follow from \Cref{triality_hom} and this also implies that $\phi_i^{-1}\circ \phi_i' \in \PGO_8^+$ for $i=2,3$.

\noindent\ref{build_alpha_B_ii}: Using the formulas from \ref{build_alpha_B_i}, we know that $\theta^+(\phi_2^{-1}\circ \phi_2') = \phi_3^{-1}\circ \phi_3'$ and $\theta^-(\phi_2^{-1}\circ \phi_2') = \phi_1^{-1}\circ \phi_1'$. Therefore by \Cref{triality_hom},
\[
\Cl(\phi_2^{-1}\circ \phi_2') = \Psi|_T^{-1} \circ ((\phi_3^{-1}\circ \phi_3') \times (\phi_1^{-1}\circ \phi_1')) \circ \Psi|_T.
\]
We rearrange this in the reverse manner as above to obtain that
\[
\Id_{\Cl(\cB_\qt)} = \Cl(\phi_2) \circ \Psi|_T^{-1} \circ (\phi_3^{-1} \times \phi_1^{-1}) \circ (\phi'_{3}\times \phi'_{1})\circ \Psi|_T \circ \Cl(\phi'_{2})^{-1}
\]
meaning that $\Cl(\phi_2) \circ \Psi|_T^{-1} \circ (\phi_3^{-1} \times \phi_1^{-1})$ and $(\phi'_{3}\times \phi'_{1})\circ \Psi|_T \circ \Cl(\phi'_{2})^{-1}$ are inverses, or equivalently that
\[
(\phi_3\times \phi_1)\circ \Psi|_T \circ \Cl(\phi_2)^{-1} = (\phi'_3\times \phi'_1)\circ \Psi|_T \circ \Cl(\phi'_2)^{-1}
\]
as claimed.

\noindent\ref{build_alpha_B_iii}: For each $T \in \Sch_S$ as in \ref{build_alpha_B_i}, we define $\alpha_{\cB,T} \colon (\Cl(\cB_\qt),\underline{\sigma_{\cB}},\underline{f_{\cB}})|_T \to \cC_\qt|_T \times \cA_\qt|_T$ by the formula
\[
\alpha_{\cB,T} = (\phi_3\times \phi_1)\circ \Psi|_T \circ \Cl(\phi_2)^{-1}.
\]
This is of course an isomorphism and it is unique by \ref{build_alpha_B_ii}. Furthermore, given two such $T_1,T_2 \in \Sch_S$, the restrictions $\alpha_{\cB,T_1}|_{T_1\times_S T_2}$ and $\alpha_{\cB,T_2}|_{T_1\times_S T_2}$ will each be constructed from an element of $\cP(T_1\times_S T_2)$. Once again by \ref{build_alpha_B_ii}, the choice of these elements does not matter and therefore $\alpha_{\cB,T_1}|_{T_1\times_S T_2} = \alpha_{\cB,T_2}|_{T_1\times_S T_2}$. Thus, since there exists a cover $\{T_i \to S\}_{i\in I}$ consisting of such $T_i \in \Sch_S$, the various isomorphisms $\alpha_{\cB,T}$ glue into a global isomorphism $\alpha_{\cB}$. It has the claimed property by construction.
\end{proof}

Given $(\cA_\qt,\cB_\qt,\cC_\qt,\alpha_{\cA})\in \fTrip(S)$, constructing $\alpha_{\cB}$ as in \Cref{build_alpha_B}\ref{build_alpha_B_iii} above produces another object $(\cB_\qt,\cC_\qt,\cA_\qt,\alpha_{\cB})\in \fTrip(S)$. This construction is functorial.
\begin{lem}\label{Theta+_morphisms}
Consider a morphism
\[
(\phi_1,\phi_2,\phi_3) \colon (\cA_\qt,\cB_\qt,\cC_\qt,\alpha_{\cA}) \to (\cA'_\qt,\cB'_\qt,\cC'_\qt,\alpha_{\cA'})
\]
in $\fTrip(S)$. Then,
\[
(\phi_2,\phi_3,\phi_1) \colon (\cB_\qt,\cC_\qt,\cA_\qt,\alpha_{\cB}) \to (\cB'_\qt,\cC'_\qt,\cA'_\qt,\alpha_{\cB'})
\]
where $\alpha_{\cB}$ and $\alpha_{\cB'}$ are constructed as in \Cref{build_alpha_B}\ref{build_alpha_B_iii}, is another morphism in $\fTrip(S)$.
\end{lem}
\begin{proof}
We only need to check that the diagram
\[
\begin{tikzcd}
(\Cl(\cB_\qt),\underline{\sigma_{\cB}},\underline{f_{\cB}}) \arrow{d}{\Cl(\phi_2)} \arrow{r}{\alpha_{\cB}} & \cC_\qt \times \cA_\qt \arrow{d}{\phi_3\times \phi_1} \\
(\Cl(\cB'_\qt),\underline{\sigma_{\cB'}},\underline{f_{\cB'}}) \arrow{r}{\alpha_{\cB'}} & \cC'_\qt \times \cA'_\qt
\end{tikzcd}
\]
commutes. To do so, let $\cP = \cIsom(\cM_\qt^3,(\cA_\qt,\cB_\qt,\cC_\qt,\alpha_{\cA}))$ and let $T\in \Sch_S$ be such that there exists a section $(\varphi_1,\varphi_2,\varphi_3)\in \cP(T)$. By \Cref{build_alpha_B}\ref{build_alpha_B_iii}, we can describe $\alpha_{\cB}$ over $T$ by
\[
\alpha_{\cB}|_T = (\varphi_3 \times \varphi_1)\circ \Psi|_T \circ \Cl(\varphi_2)^{-1}.
\]
Because we have an isomorphism $(\phi_1,\phi_2,\phi_3)$ between our two objects of $\fTrip(S)$, we will also have a section
\[
(\phi_1|_T \circ \varphi_1,\phi_2|_T \circ \varphi_2, \phi_3|_T \circ \varphi_3) \in \cIsom(\cM_\qt^3,(\cA_\qt',\cB_\qt',\cC_\qt',\alpha_{\cA'}))(T)
\]
and therefore
\[
\alpha_{\cB'}|_T = ((\phi_3|_T\circ \varphi_3) \times (\phi_1|_T \circ \varphi_1)) \circ \Psi|_T \circ \Cl(\phi_2|_T\circ \varphi_2)^{-1}.
\]
These formulas then make it clear that $(\phi_3\times\phi_1)|_T \circ \alpha_{\cB}|_T = \alpha_{\cB'}|_T \circ \Cl(\phi_2)|_T$. Since there exists a cover $\{T_i \to S\}_{i\in I}$ over which $\cP$ has sections, these local equalities over each $T_i$ imply that globally we have $(\phi_3\times \phi_1)\circ \alpha_{\cB} = \alpha_{\cB'}\circ \Cl(\phi_2)$ as desired, finishing the proof.
\end{proof}

\begin{defn}\label{defn_Theta+_functor}
We define an endofunctor on $\fTrip(S)$ by
\begin{align*}
\Theta^+_S \colon \fTrip(S) &\to \fTrip(S) \\
(\cA_\qt,\cB_\qt,\cC_\qt,\alpha_{\cA}) &\mapsto (\cB_\qt, \cC_\qt,\cA_\qt,\alpha_{\cB}) \\
(\phi_1,\phi_2,\phi_3) &\mapsto (\phi_2,\phi_3,\phi_1)
\end{align*}
on objects and morphisms respectively, where $\alpha_{\cB}$ is defined as in \Cref{build_alpha_B}. Notationally, we may also write $\alpha_{\cB} = \Theta^+_S(\alpha_{\cA})$.
\end{defn}

\begin{lem}\label{Theta+_order_3}
Consider the endofunctor $\Theta^+_S$ of \Cref{defn_Theta+_functor}.
\begin{enumerate}[label={\rm(\roman*)}]
\item \label{Theta+_order_3_i} The split object $\cM_\qt^3 \in \fTrip(S)$ is a fixed point of $\Theta^+_S$.
\item \label{Theta+_order_3_ii} $\Theta^+_S$ is order $3$ up to equality, i.e., it satisfies $(\Theta^+_S)^3 = \Id$. In particular, it is fully faithful and surjective.
\end{enumerate} 
\end{lem}
\begin{proof}
\noindent\ref{Theta+_order_3_i}: Let $\Theta^+_S(\cM_\qt^3) = (\Mat_8(\cO)_\qt,\Mat_8(\cO)_\qt,\Mat_8(\cO)_\qt,\Psi')$. We must argue that $\Psi'=\Psi$. However, this is immediate since we have a global section $(\Id,\Id,\Id) \in \cIsom(\cM_\qt^3,\cM_\qt^3)$, and therefore by \Cref{build_alpha_B}\ref{build_alpha_B_iii} we have
\[
\Psi' = (\Id \times \Id)\circ \Psi \circ \Cl(\Id)^{-1} = \Psi.
\]

\noindent\ref{Theta+_order_3_ii}: It is clear from the definition that $\Theta^+_S$ is order $3$ on morphisms. Therefore, given $(\cA_\qt,\cB_\qt,\cC_\qt,\alpha_{\cA})\in \fTrip(S)$ and setting notation
\begin{align*}
\Theta^+_S(\cA_\qt,\cB_\qt,\cC_\qt,\alpha_{\cA}) &= (\cB_\qt,\cC_\qt,\cA_\qt,\alpha_{\cB}), \\
\Theta^+_S(\cB_\qt,\cC_\qt,\cA_\qt,\alpha_{\cB}) &= (\cC_\qt,\cA_\qt,\cB_\qt,\alpha_{\cC}) \text{, and} \\
\Theta^+_S(\cC_\qt,\cA_\qt,\cB_\qt,\alpha_{\cC}) &= (\cA_\qt,\cB_\qt,\cC_\qt,\alpha'_{\cA}),
\end{align*}
we only need to check that $\alpha'_{\cA} = \alpha_{\cA}$. Since $\cM_\qt^3$ is a fixed point of $\Theta^+_S$ by \ref{Theta+_order_3_i}, if we have a section $(\phi_1,\phi_2,\phi_3) \in \cIsom(\cM_\qt^3,(\cA_\qt,\cB_\qt,\cC_\qt,\alpha_{\cA}))(T)$ for some $T\in \Sch_S$, then we also have that
\[
(\phi_3,\phi_1,\phi_2) \in \cIsom(\cM_\qt^3,(\cC_\qt,\cA_\qt,\cB_\qt,\alpha_{\cC}))(T).
\]
Thus by \Cref{build_alpha_B}\ref{build_alpha_B_iii}, locally we have
\[
\alpha'_{\cA}|_T = (\phi_2\times \phi_3) \circ \Psi|_T \circ \Cl(\phi_1)^{-1}
\]
which by \Cref{eq_alpha_A} means that $\alpha'_{\cA} = \alpha_{\cA}$ globally as required.
\end{proof}
In light of \Cref{Theta+_order_3}\ref{Theta+_order_3_ii}, we denote $\Theta^-_S = (\Theta^+_S)^2$. We also define a second, more straightforward, endofunctor on $\fTrip(S)$. Given two algebras $\cB$ and $\cC$, let $\sw \colon \cB\times \cC \to \cC \times \cB$ be the switch isomorphism, $\sw(b,c)=(c,b)$.
\begin{defn}\label{defn_Theta_functor}
We define an endofunctor on $\fTrip(S)$ by
\begin{align*}
\Theta_S \colon \fTrip(S) &\to \fTrip(S) \\
(\cA_\qt,\cB_\qt,\cC_\qt,\alpha_{\cA}) &\mapsto (\cA_\qt,\cC_\qt,\cB_\qt,\sw \circ \alpha_{\cA}) \\
(\phi_1,\phi_2,\phi_3) &\mapsto (\phi_1,\phi_3,\phi_2).
\end{align*}
on objects and morphisms respectively.
\end{defn}
It is clear that this defines an order two endofunctor and that $\Theta_S$ is fully faithful and surjective. It also satisfies the following.

\begin{lem}\label{Theta_functor_properties}
Consider the functor $\Theta_S$ of \ref{defn_Theta_functor}.
\begin{enumerate}[label={\rm(\roman*)}]
\item \label{Theta_functor_properties_i} There are isomorphisms
\begin{align*}
(\pi_{\bO}(\varphi),\pi_{\bO}(\varphi),\pi_{\bO}(\varphi)) &\colon \cM_\qt^3 \iso \Theta_S(\cM_\qt^3)\text{, and} \\
(\pi_{\bO}(\varphi),\pi_{\bO}(\varphi),\pi_{\bO}(\varphi)) &\colon \Theta_S(\cM_\qt^3) \iso \cM_\qt^3
\end{align*}
where $\pi_{\bO}(\varphi)$ is as in \Cref{theta_description}. These isomorphisms are inverses.
\item \label{Theta_functor_properties_ii} Also consider $\Theta^+_S$ of \Cref{defn_Theta+_functor}. We have $\Theta^+_S \circ \Theta_S = \Theta_S \circ \Theta^-_S$.
\end{enumerate}
\end{lem}
\begin{proof}
\noindent\ref{Theta_functor_properties_i}: This follows immediately from \Cref{split_Clifford_switch_iso} and the fact that $\pi_{\bO}(\varphi)$ is order $2$.

\noindent\ref{Theta_functor_properties_ii}: Let $(\cA_\qt,\cB_\qt,\cC_\qt,\alpha_{\cA})\in \fTrip(S)$ be an object and let
\[
\cP = \cIsom(\cM_\qt^3,(\cA_\qt,\cB_\qt,\cC_\qt,\alpha_{\cA})).
\]
Whenever we have a section $(\phi_1,\phi_2,\phi_3) \in \cP$, we will also have a section
\[
(\phi_1,\phi_3,\phi_2) \in \cIsom(\Theta_S(\cM_\qt^3),(\cA_\qt,\cC_\qt,\cB_\qt,\sw\circ \alpha_{\cA}))
\]
which by composing with the isomorphism from \ref{Theta_functor_properties_i} means we will have a section
\[
(\phi_1\circ \pi_{\bO}(\varphi),\phi_3\circ \pi_{\bO}(\varphi),\phi_2\circ \pi_{\bO}(\varphi)) \in \cIsom(\cM_\qt^3,(\cA_\qt,\cC_\qt,\cB_\qt,\sw\circ \alpha_{\cA})).
\]
Therefore by \Cref{build_alpha_B}\ref{build_alpha_B_iii}, $\Theta^+_S(\sw\circ \alpha_{\cA})$ will be locally described by
\begin{align*}
\Theta^+_S(\sw\circ \alpha_{\cA}) &= (\phi_2\circ \pi_{\bO}(\varphi))\times (\phi_1\circ \pi_{\bO}(\varphi)) \circ \Psi \circ \Cl(\phi_3\circ \pi_{\bO}(\varphi))^{-1} \\
&= (\phi_2\times \phi_1)\circ (\pi_{\bO}(\varphi) \times \pi_{\bO}(\varphi)) \circ \Psi \circ \Cl(\pi_{\bO}(\varphi))^{-1} \circ \Cl(\phi_3)^{-1} \\
&= (\phi_2\times \phi_1)\circ (\sw \circ \Psi) \circ \Cl(\phi_3)^{-1} \\
&= \sw \circ (\phi_1\times \phi_2)\circ \Psi \circ \Cl(\phi_3)^{-1}.
\end{align*}
On the other hand, $\Theta^-_S(\alpha_{\cA})$ will be locally described by
\[
\Theta^-_S(\alpha_{\cA}) = (\phi_1\times \phi_2)\circ \Psi \circ \Cl(\phi_3)^{-1}
\]
and so $\Theta_S(\Theta^-_S(\alpha_{\cA})) = \sw \circ (\phi_1\times \phi_2)\circ \Psi \circ \Cl(\phi_3)^{-1}$. Therefore,
\begin{align*}
\Theta^+_S(\Theta_S(\cA_\qt,\cB_\qt,\cC_\qt,\alpha_{\cA})) &= (\cC_\qt,\cB_\qt,\cA_\qt,\Theta^+_S(\sw \circ \alpha_{\cA})) \\
&= (\cC_\qt,\cB_\qt,\cA_\qt,\sw \circ \Theta^-_S(\alpha_{\cA}))\\
&= \Theta_S(\Theta^-_S(\cA_\qt,\cB_\qt,\cC_\qt,\alpha_{\cA})),
\end{align*}
which finishes the proof.
\end{proof}

\subsection{Stack Endomorphisms and Cohomology}
Since our base scheme $S$ is arbitrary, the endofunctors $\Theta^+_S, \Theta_S \colon \fTrip(S) \to \fTrip(S)$ constructed in \Cref{endo_Trip(S)} have obvious analogues $\Theta^+_T, \Theta_T \colon \fTrip(T) \to \fTrip(T)$ for any $T\in \Sch_S$. These functors can be combined into endomorphisms of the gerbe $\fTrip$ of \Cref{defn_Trip}.
\begin{lem}\label{endo_compatibility}
Consider a morphism in $\fTrip$
\[
(g,\phi) \colon (T',(\cA'_\qt,\cB'_\qt,\cC'_\qt,\alpha_{\cA'})) \to (T,(\cA_\qt,\cB_\qt,\cC_\qt,\alpha_{\cA})). 
\]
Recall that $g\colon T'\to T$ is a morphism of $S$--schemes and $\phi \colon (\cA'_\qt,\cB'_\qt,\cC'_\qt,\alpha_{\cA'}) \to (\cA_\qt|_{T'},\cB_\qt|_{T'},\cC_\qt|_{T'},\alpha_{\cA}|_{T'})$ is an isomorphism in $\fTrip(T')$.
\begin{enumerate}[label={\rm(\roman*)}]
\item \label{endo_compatibility_i} There is an equality
\[
\Theta^+_{T'}(\cA_\qt|_{T'},\cB_\qt|_{T'},\cC_\qt|_{T'},\alpha_{\cA}|_{T'}) = (\Theta^+_T(\cA_\qt,\cB_\qt,\cC_\qt,\alpha_{\cA}))|_{T'}.
\]
\item \label{endo_compatibility_ii} $(g,\Theta^+_{T'}(\phi)) \colon (T',(\cB'_\qt,\cC'_\qt,\cA'_\qt,\alpha_{\cB'})) \to (T,(\cB_\qt,\cA_\qt,\cC_\qt,\alpha_{\cB}))$ is a morphism in $\fTrip$.
\item \label{endo_compatibility_iii} There is an equality
\[
\Theta_{T'}(\cA_\qt|_{T'},\cB_\qt|_{T'},\cC_\qt|_{T'},\alpha_{\cA}|_{T'}) = (\Theta_T(\cA_\qt,\cB_\qt,\cC_\qt,\alpha_{\cA}))|_{T'}.
\]
\item \label{endo_compatibility_iv} $(g,\Theta_{T'}(\phi)) \colon (T',(\cA'_\qt,\cC'_\qt,\cA'_\qt,\sw\circ\alpha_{\cA'})) \to (T,(\cA_\qt,\cC_\qt,\cB_\qt,\sw\circ\alpha_{\cA}))$ is a morphism in $\fTrip$.
\end{enumerate}
\end{lem}
\begin{proof}
\noindent\ref{endo_compatibility_i}: First, we know that $\Theta^+_T(\cA_\qt,\cB_\qt,\cC_\qt,\alpha_{\cA}) = (\cB_\qt,\cC_\qt,\cA_\qt,\alpha_{\cB})$ and therefore
\[
(\Theta^+_T(\cA_\qt,\cB_\qt,\cC_\qt,\alpha_{\cA}))|_{T'} = (\cB_\qt|_{T'},\cC_\qt|_{T'},\cA_\qt|_{T'},\alpha_{\cB}|_{T'}).
\]
Thus, we need to show that $\alpha_{\cB}|_{T'} = \Theta^+_{T'}(\alpha_{\cA}|_{T'})$. Let $\{T_i \to T\}_{i\in I}$ be a cover over which $\cIsom(\cM_\qt^3|_T, (\cA_\qt,\cB_\qt,\cC_\qt,\alpha_{\cA}))$ has sections. Over any such $T_i$, there will be a section $(\phi_1,\phi_2,\phi_3)$ and by definition of $\alpha_{\cB}$ from \Cref{build_alpha_B}\ref{build_alpha_B_iii} we have
\[
\alpha_{\cB}|_{T_i} = (\phi_3\times \phi_1)\circ \Psi|_{T_i} \circ \Cl(\phi_2)^{-1}.
\]
However, $\{T'\times_T T_i \to T'\}_{i\in I}$ is a cover of $T'$ and the sheaf
\[
\cIsom(\cM_\qt^3|_{T'},(\cA_\qt|_{T'},\cB_\qt|_{T'},\cC_\qt|_{T'},\alpha_{\cA}|_{T'}))
\]
will have sections over this cover, namely those sections coming from restriction along the morphisms $T'\times_T T_i \to T_i$. Therefore, for such a $T'\times_T T_i$, we will have a section $(\phi_1|_{T'\times_T T_i},\phi_2|_{T'\times_T T_i},\phi_3|_{T'\times_T T_i})$. Thus,
\begin{align*}
\Theta^+_{T'}(\alpha_{\cA}|_{T'})|_{T'\times_T T_i} &= (\phi_3|_{T'\times_T T_i}\times \phi_1|_{T'\times_T T_i})\circ \Psi|_{T'\times_T T_i} \circ \Cl(\phi_2|_{T'\times_T T_i})^{-1} \\
&= (\alpha_{\cB}|_{T_i})|_{T'\times_T T_i} \\
&= (\alpha_{\cB}|_{T'})|_{T'\times_T T_i}
\end{align*}
which means that $\Theta^+_{T'}(\alpha_{\cA}|_{T'})$ and $\alpha_{\cB}|_{T'}$ agree on a cover of $T'$. Hence $\Theta^+_{T'}(\alpha_{\cA}|_{T'}) = \alpha_{\cB}|_{T'}$ as desired.

\noindent\ref{endo_compatibility_ii}: This follows immediately from \ref{endo_compatibility_i}.

\noindent\ref{endo_compatibility_iii}: This is straightforward since the switch isomorphism is stable under base change. In particular, $(\sw\circ \alpha_{\cA})|_{T'} = \sw \circ \alpha_{\cA}|_{T'}$.

\noindent\ref{endo_compatibility_iv}: This follows immediately from \ref{endo_compatibility_iii}.
\end{proof}

\begin{defn}\label{defn_endo_Trip}
We define the following two endomorphisms of the gerbe $\fTrip$.
\begin{enumerate}[label={\rm(\roman*)}]
\item \label{defn_endo_Trip_i}
\hfill $\begin{aligned}[t]
\Theta^+ \colon \fTrip &\to \fTrip \\
(T,(\cA_\qt,\cB_\qt,\cC_\qt,\alpha_{\cA})) &\mapsto (T,\Theta^+_T(\cA_\qt,\cB_\qt,\cC_\qt,\alpha_{\cA})) \\
(g,\phi) &\mapsto (g,\Theta_{T'}^+(\phi))
\end{aligned}$ \hfill\null \vspace{2ex}
\item \label{defn_endo_Trip_ii}
\hfill $\begin{aligned}[t]
\Theta \colon \fTrip &\to \fTrip \\
(T,(\cA_\qt,\cB_\qt,\cC_\qt,\alpha_{\cA})) &\mapsto (T,\Theta_T(\cA_\qt,\cB_\qt,\cC_\qt,\alpha_{\cA})) \\
(g,\phi) &\mapsto (g,\Theta_{T'}(\phi)).
\end{aligned}$\hfill\null
\end{enumerate}
\end{defn}
It is clear from construction that $\Theta^+|_{\fTrip(T)} \colon \fTrip(T) \to \fTrip(T)$ is equal to $\Theta^+_T$, and so these endomorphisms agree with our previous endofunctors on fibers. The same holds for $\Theta$, it restricts to $\Theta_T$ on $\fTrip(T)$.

\begin{prop}\label{endo_Trip_group_homs}
Consider the endomorphism $\Theta^+$ of \Cref{defn_endo_Trip}\ref{defn_endo_Trip_i}.
\begin{enumerate}[label={\rm(\roman*)}]
\item \label{endo_Trip_group_homs_i} $\Theta^+$ is order three, i.e., $(\Theta^+)^3 = \Id$. In particular, $\Theta^+$ is an isomorphism.
\item \label{endo_Trip_group_homs_ii} Viewing the automorphism sheaf of the split object $\bAut(S,\cM_\qt^3)$ as $\PGO_8^+$ via \Cref{Trip_properties}\ref{Trip_properties_ii}, the induced group automorphism
\[
\bAut(\Theta^+)\colon \PGO_8^+ \to \PGO_8^+
\]
is $\theta^+$.
\item \label{endo_Trip_group_homs_iii} $\Theta$ is order two and hence an isomorphism. Furthermore, $\Theta^+ \circ \Theta = \Theta \circ \Theta^-$, where we define $\Theta^- = (\Theta^+)^2$.
\item \label{endo_Trip_group_homs_iv} Using the isomorphisms of \Cref{Theta_functor_properties}\ref{Theta_functor_properties_i}, we also identify $\bAut(\Theta_S(\cM_\qt^3))\cong \PGO_8^+$. Then, the induced group automorphism
\[
\bAut(\Theta) \colon \PGO_8^+ \to \PGO_8^+
\]
is $\theta$.
\end{enumerate}
\end{prop}
\begin{proof}
\noindent\ref{endo_Trip_group_homs_i}: Since $\Theta^+$ is given by $\Theta^+_T$ on fibers, this follows from \Cref{Theta+_order_3}\ref{Theta+_order_3_ii}.

\noindent\ref{endo_Trip_group_homs_ii}: For $T\in \Sch_S$, the group homomorphism $\bAut(\Theta^+)(T)$ will be the composition
\[
\begin{tikzcd}
\PGO_8^+(T) \arrow{r}{\sim} & \bAut(S,\cM_\qt^3)(T) \arrow{r}{\Theta^+_T} & \bAut(S,\cM_\qt^3)(T) \arrow{r}{\sim} & \PGO_8^+(T) \\[-4.5ex]
\phi \arrow[mapsto]{r} & (\phi,\theta^+(\phi),\theta^-(\phi)) \arrow[mapsto]{r} & (\theta^+(\phi),\theta^-(\phi),\phi) \arrow[mapsto]{r} & \theta^+(\phi).
\end{tikzcd}
\]
Hence, we see that we obtain the map $\theta^+ \colon \PGO_8^+ \iso \PGO_8^+$ overall.

\noindent\ref{endo_Trip_group_homs_iii}: Similarly as in \ref{endo_Trip_group_homs_i}, the first claim holds because $\Theta$ restricts to $\Theta_T$ on fibers, each of which is order two. Likewise, the second claim follows because it holds over each fiber by \Cref{Theta_functor_properties}\ref{Theta_functor_properties_ii}.

\noindent\ref{endo_Trip_group_homs_iv}: For $T\in \Sch_S$, the group homomorphism $\bAut(\Theta)(T)$ will be the composition
\begin{align*}
\PGO_8^+(T) &\iso \bAut(S,\cM_\qt^3)(T) \xrightarrow{\Theta}  \bAut(S,\Theta_S(\cM_\qt^3))(T) \\
&\iso \bAut(S,\cM_\qt^3)(T) \iso \PGO_8^+(T)
\end{align*}
which sends
\begin{align*}
\phi &\mapsto (\phi,\theta^+(\phi),\theta^-(\phi)) \\
&\mapsto (\phi,\theta^-(\phi),\theta^+(\phi)) \\ 
&\mapsto (\pi_{\bO}(\varphi)\circ\phi\circ\pi_{\bO}(\varphi),\pi_{\bO}(\varphi)\circ\theta^-(\phi)\circ\pi_{\bO}(\varphi),\pi_{\bO}(\varphi)\circ\theta^+(\phi)\circ\pi_{\bO}(\varphi)) \\
&\mapsto \pi_{\bO}(\varphi)\circ\phi\circ\pi_{\bO}(\varphi).
\end{align*}
Since $\pi_{\bO}(\varphi)\circ\phi\circ\pi_{\bO}(\varphi) = \theta(\phi)$ by \Cref{theta_description}\ref{theta_description_iii}, these calculations for each $T \in \Sch_S$ show that $\bAut(\Theta)=\theta$ as morphisms of sheaves.
\end{proof}

\begin{cor}\label{PGO+_cohomology}
View $H^1(S,\PGO_8^+)$ as the set of isomorphism classes of $\fTrip(S)$. The maps on cohomology induced by the automorphisms $\theta^+$ and $\theta$ of $\PGO_8^+$ are
\begin{align*}
\widetilde{\theta^+} \colon H^1(S,\PGO_8^+) &\to H^1(S,\PGO_8^+) \\
[(\cA_\qt,\cB_\qt,\cC_\qt,\alpha_{\cA})] &\mapsto [(\cB_\qt,\cC_\qt,\cA_\qt,\alpha_{\cB})] 
\end{align*}
and
\begin{align*}
\widetilde{\theta} \colon H^1(S,\PGO_8^+) &\to H^1(S,\PGO_8^+) \\
[(\cA_\qt,\cB_\qt,\cC_\qt,\alpha_{\cA})] &\mapsto [(\cA_\qt,\cC_\qt,\cB_\qt,\sw \circ \alpha_{\cA})] 
\end{align*}
respectively.
\end{cor}
\begin{proof}
Since we know the induced group homomorphism from \Cref{endo_Trip_group_homs}\ref{endo_Trip_group_homs_ii} and \ref{endo_Trip_group_homs_iv}, the claim follows by applying \Cref{stack_cohomology}.
\end{proof}

\section{Trialitarian Algebras}\label{Trialitarian_Algebras}
In the same style as the previous section, but now considering the group $\PGO_8^+\rtimes \SS_3$, we define groupoids and a stack of trialitarian algebras and show that we obtain categories equivalent to the groupoid or stack of $\PGO_8^+\rtimes \SS_3$--torsors. In \cite[\S 43]{KMRT}, underlying a trialitarian $\FF$--algebra is an Azumaya algebra over a degree $3$ \'etale extension of $\FF$. Likewise, our definition of a trialitarian algebra over a scheme $S$ involves a degree $3$ \'etale extension $L \to S$. We first define some morphisms associated to such a cover and then use those morphisms to define trialitarian algebras. Our newly defined stack of trialitarian algebras will be equivalent to both the stack of simply connected groups of type $D_4$ as well as the stack of adjoint groups of type $D_4$ since they are all equivalent to the stack of $\PGO_8^+\rtimes \SS_3$--torsors by definition. In each case, we provide a concrete realization of these equivalences by defining the adjoint and simply connected group associated to a trialitarian algebra. All of our definitions in this section are inspired by their analogues in \cite{KMRT}, with modifications made to work in our setting.

\subsection{The Definition via the Trialitarian Cover}
We recall the background on finite \'etale covers from \Cref{etale_covers}. In particular, recall the morphisms $\pi,\tau \colon S^{\sqcup 6} \to S^{\sqcup 3}$. Let $\cE$ be an Azumaya $\cO|_{S^{\sqcup 3}}$--algebra, so $\cE$ is of the form
\begin{align*}
\cE \colon \Sch_{S^{\sqcup 3}} &\to \Ab \\
U\sqcup V \sqcup W &\mapsto \cA(U)\times \cB(V) \times \cC(W)
\end{align*} 
where $\cA$, $\cB$, and $\cC$ are Azumaya $\cO$--algebras. We denote this by $\cE = (\cA,\cB,\cC)$.

\begin{lem}\label{split_E_twist}
Let $\cE = (\cA,\cB,\cC)$ be an Azumaya $\cO|_{S^{\sqcup 3}}$--algebra. Then
\[
\pi_*(\tau^*(\cE)) = (\cB\times\cC, \cC\times\cA, \cA \times \cB).
\]
\end{lem}
\begin{proof}
First, it is easy to compute that
\[
\tau^*(\cE) = (\cB,\cC,\cA,\cC,\cA,\cB).
\]
Next, for $U\sqcup V \sqcup W \in \Sch_{S^{\sqcup 3}}$ we have that
\[
(U\sqcup V \sqcup W)\times_\pi S^{\sqcup 6} = (U\sqcup V \sqcup W) \sqcup (U \sqcup V \sqcup W).
\]
Therefore,
\begin{align*}
\pi_*(\tau^*(\cE))(U\sqcup V \sqcup W) &= \tau^*(\cE)\big((U\sqcup V \sqcup W)\times_\pi S^{\sqcup 6}\big) \\
&= \tau^*(\cE)\big((U\sqcup V \sqcup W) \sqcup (U \sqcup V \sqcup W)\big) \\
&= \cB(U)\times \cC(V) \times \cA(W) \times \cC(U) \times \cA(V) \times \cB(W) \\
&= (\cB\times \cC)(U) \times (\cC\times \cA)(V) \times (\cA\times \cB)(W)
\end{align*}
and $\pi_*(\tau^*(\cE)) \cong (\cB\times\cC, \cC\times\cA, \cA \times \cB)$ as claimed.
\end{proof}

Now, for a degree $3$ \'etale cover $L\to S$, recall the twisted versions of $\pi$ and $\tau$, namely $\pi_L,\tau_L \colon \Sigma_2(L) \to L$.
\begin{lem}\label{gamma_E-morphism}
Let $L \to S$ be a degree $3$ \'etale cover and let $\cE$ be an Azumaya algebra over $L$. For any $T \in \Sch_S$ which splits $L$, we write $\cE|_{L\times_S T} = (\cA_T,\cB_T, \cC_T)$ for Azumaya $\cO|_T$--algebras $\cA_T$, $\cB_T$, and $\cC_T$. There is a morphism of $\cO|_L$--algebras
\[
\gamma_\cE \colon \cE \to (\pi_L)_*(\tau_L^*(\cE))
\]
that is given over any $T \in \Sch_S$ which splits $L$ by
\begin{align*}
\gamma_{\cE}|_{L\times_S T} \colon (\cA_T,\cB_T,\cC_T) &\to (\cB_T\times\cC_T,\cC_T\times\cA_T,\cA_T\times\cB_T) \\
(a,b,c) &\mapsto ((b,c),(c,a),(a,b))
\end{align*}
where we use \Cref{split_E_twist}.
\end{lem}
\begin{proof}
If such a map exists then it is uniquely defined by the formula above since there exists a cover $\{T_i \to S\}_{i\in I}$ which splits $L$ and $\{L\times_S T_i \to L\}_{i\in I}$ is also a cover. Thus, we must argue that these local descriptions are compatible and glue into a global morphism. Let $T_1, T_2 \in \Sch_S$ be two schemes which split $L$, i.e., $L\times_S T_i \cong T_i \sqcup T_i \sqcup T_i$. Then, $\cE|_{T_i} = (\cA_{T_i},\cB_{T_i},\cC_{T_i})$ for $i=1,2$. We have a transition isomorphism
\[
(L\times_S T_1)\times_{T_1} T_{12} \cong T_{12} \sqcup T_{12} \sqcup T_{12} \iso T_{12} \sqcup T_{12} \sqcup T_{12} \cong (L\times_S T_2)\times_{T_2} T_{12}
\]
which is a permutation in $\SS_3(T_{12})$. Without loss of generality, it is sufficient to assume this permutation is a diagonal section of $\SS_3(T_{12})$, i.e., it is a section of the presheaf $\PP_3$. Further, it is sufficient to assume it is either the cycle $c = (1\; 2\; 3)$ or the transposition $\lambda = (2\; 3)$.

First, assume it is $c$. Since the pullback
\[
c^*\big((\cE|_{L\times_S T_2})|_{L\times_S T_{12}}\big) = c^*(\cA_{T_2}|_{T_{12}},\cB_{T_2}|_{T_{12}},\cC_{T_2}|_{T_{12}}) = (\cB_{T_2}|_{T_{12}},\cC_{T_2}|_{T_{12}},\cA_{T_2}|_{T_{12}})
\]
agrees with $(\cE|_{L\times_S T_1})|_{L\times_S T_{12}} = (\cA_{T_1}|_{T_{12}},\cB_{T_1}|_{T_{12}},\cC_{T_1}|_{T_{12}})$, we have equalities
\begin{align*}
\cA_{T_1}|_{T_{12}} &= \cB_{T_2}|_{T_{12}} \\
\cB_{T_1}|_{T_{12}} &= \cC_{T_2}|_{T_{12}} \\
\cC_{T_1}|_{T_{12}} &= \cA_{T_2}|_{T_{12}}.
\end{align*}
Thus, the map $c^*\big((\gamma_{\cE}|_{L\times_S T_2})|_{L\times_S T_{12}}\big)$ acts as $(b',c',a')\mapsto ((c',a'),(a',b'),(b',c'))$ which clearly matches with $(\gamma_{\cE}|_{L\times_S T_1})|_{L\times_S T_{12}}$ after setting $a=b'$, $b=c'$, and  $c=a'$. Hence, the local maps agree when the transition isomorphism is $c$.

Now, assume the transition isomorphism is $\lambda$. In this case, we obtain local equalities
\begin{align*}
\cA_{T_1}|_{T_{12}} &= \cA_{T_2}|_{T_{12}} \\
\cB_{T_1}|_{T_{12}} &= \cC_{T_2}|_{T_{12}} \\
\cC_{T_1}|_{T_{12}} &= \cB_{T_2}|_{T_{12}}
\end{align*}
and the map $\lambda^*\big((\gamma_{\cE}|_{L\times_S T_2})|_{L\times_S T_{12}}\big)$ acts as $(a',c',b') \mapsto ((c',b'),(b',a'),(a',c'))$. After setting $a=a'$, $b=c'$, and $c=b'$, this too clearly agrees with $(\gamma_{\cE}|_{L\times_S T_1})|_{L\times_S T_{12}}$ which of course acts as $(a,b,c) \mapsto ((b,c),(c,a),(a,b))$. Hence, the local maps also agree when the transition morphism is $\lambda$. This means that the local morphisms agree on overlaps and glue into a global morphism with the stated property, so we are done.
\end{proof}

\begin{lem}\label{gamma_respects_autos}
Let $L\to S$ be a degree $3$ \'etale cover and let $\psi \colon \cE_1 \iso \cE_2$ be an isomorphism of Azumaya algebras over $L$. Then, the diagram
\[
\begin{tikzcd}[column sep = 10ex]
\cE_1 \ar[r,"\psi"] \ar[d,"\gamma_{\cE_1}"] & \cE_2 \ar[d,"\gamma_{\cE_2}"] \\
(\pi_L)_*(\tau_L^*(\cE_1)) \ar[r,"(\pi_L)_*(\tau_L^*(\psi))"] & (\pi_L)_*(\tau_L^*(\cE_2))
\end{tikzcd}
\]
commutes, where $\gamma_{\cE_i}$ are the maps of \Cref{gamma_E-morphism}.
\end{lem}
\begin{proof}
This may be checked locally, and therefore we may assume that everything is split. In this case, $\cE_i = (\cA_i,\cB_i,\cC_i)$ and the isomorphism $\psi$ consists of three isomorphisms $\psi_{\cA} \colon \cA_1 \iso \cA_2$ and similarly $\psi_{\cB}$ and $\psi_{\cC}$. We know the form of $(\pi_L)_*(\tau_L^*(\cE_i))$ from \Cref{split_E_twist}, from which we also see that
\[
(\pi_L)_*(\tau_L^*(\psi)) = (\psi_{\cB}\times\psi_{\cC}, \psi_{\cC}\times\psi_{\cA},\psi_{\cA}\times\psi_{\cB}).
\]
Then, using the description of $\gamma_{\cE_i}$ from \Cref{gamma_E-morphism} when things are split, we can simply trace an element through the diagram, yielding
\[
\begin{tikzcd}
(a,b,c) \ar[r,mapsto] \ar[d,mapsto] & (\psi_{\cA}(a),\psi_{\cB}(b),\psi_{\cC}(c)) \ar[d,mapsto] \\
((b,c),(c,a),(a,b)) \ar[r,mapsto] & ((\psi_{\cB}(b),\psi_{\cC}(c)),(\psi_{\cC}(c),\psi_{\cA}(a)),(\psi_{\cA}(a),\psi_{\cB}(b)))
\end{tikzcd}
\]
which justifies the claim that the diagram commutes.
\end{proof}

We can now use the morphism $\rho_{\Sigma_2(L)} \colon \Sigma_2(L) \to \Sigma_2(L)$ to construct a trialitarian triple.
\begin{lem}\label{triple_over_Sigma(L)}
Let $(\cE,\sigma,f)$ be a quadratic triple of degree $8$ over $L$ and consider an isomorphism
\[
\alpha_\cE \colon (\Cl(\cE_\qt),\underline{\sigma},\underline{f}) \iso (\pi_L)_*(\tau_L^*(\cE_\qt))
\]
of algebras with semi-trace. For any $T\in \Sch_S$ which splits $L$, we write $(\cE_\qt)|_{L\times_S T} = (\cA_{T\qt},\cB_{T\qt},\cC_{T\qt})$ and $\alpha_{\cE}|_{L\times_S T} = (\alpha_{T\cA},\alpha_{T\cB},\alpha_{T\cC})$ for isomorphisms
\begin{align*}
\alpha_{T\cA} \colon \Cl(\cA_{T\qt})&\iso \cB_{T\qt}\times \cC_{T\qt} \\
\alpha_{T\cB} \colon \Cl(\cB_{T\qt})&\iso \cC_{T\qt}\times \cA_{T\qt} \\
\alpha_{T\cC} \colon \Cl(\cC_{T\qt})&\iso \cA_{T\qt}\times \cB_{T\qt}.
\end{align*}
Then, there is trialitarian triple over $\Sigma_2(L)$,
\[
(\pi_L^*(\cE_\qt),\rho_{\Sigma_2(L)}^*(\pi_L^*(\cE_\qt)),\rho_{\Sigma_2(L)}^{*2}(\pi_L^*(\cE_\qt)),\Xi(\alpha_\cE)) \in \fTrip(\Sigma_2(L))
\]
where $\Xi(\alpha_\cE)$ is the isomorphism locally defined by
\[
\Xi(\alpha_\cE)|_{\Sigma_2(L)\times_S T} = (\alpha_{T\cA},\alpha_{T\cB},\alpha_{T\cC},\sw\circ\alpha_{T\cA},\sw\circ\alpha_{T\cB},\sw\circ\alpha_{T\cC}).
\]
\end{lem}
\begin{proof}
First, we check that our claim is well-defined, i.e., that locally $\Xi(\alpha_\cE)|_{\Sigma_2(L)\times_S T}$ has the appropriate domain and codomain. Because $\pi_L$ is the standard degree $2$ \'etale cover of $L$, when $L$ is split we will have $\pi_L|_{L\times_S T} \colon (T \sqcup T \sqcup T) \sqcup (T \sqcup T \sqcup T) \to T \sqcup T \sqcup T$ and
\[
\pi_L^*(\cE_\qt)|_{\Sigma_2(L)\times_S T} = \pi_L|_{L\times_S T}^*(\cE_\qt|_{L\times_S T}) = (\cA_{T\qt},\cB_{T\qt},\cC_{T\qt},\cA_{T\qt},\cB_{T\qt},\cC_{T\qt}).
\]
The Clifford algebra $\Cl(\pi_L^*(\cE_\qt)|_{\Sigma_2(L)\times_S T})$ is then
\[
(\Cl(\cA_{T\qt}),\Cl(\cB_{T\qt}),\Cl(\cC_{T\qt}),\Cl(\cA_{T\qt}),\Cl(\cB_{T\qt}),\Cl(\cC_{T\qt})).
\]
By the definition of $\rho_{\Sigma_2(L)}$, we know from \Cref{rho_pullback} that
\begin{align*}
\rho_{\Sigma_2(L)}^*(\pi_L^*(\cE_\qt)|_{\Sigma_2(L)\times_S T}) &= (\cB_{T\qt},\cC_{T\qt},\cA_{T\qt},\cC_{T\qt},\cA_{T\qt},\cB_{T\qt})\text{, and} \\
\rho_{\Sigma_2(L)}^{*2}(\pi_L^*(\cE_\qt)|_{\Sigma_2(L)\times_S T}) &= (\cC_{T\qt},\cA_{T\qt},\cB_{T\qt},\cB_{T\qt},\cC_{T\qt},\cA_{T\qt})
\end{align*}
so we have that $\big(\rho_{\Sigma_2(L)}^*(\pi_L^*(\cE_\qt))\times \rho_{\Sigma_2(L)}^{*2}(\pi_L^*(\cE_\qt))\big)|_{\Sigma_2(L)\times_S T}$ is the $\cO|_{T^{\sqcup 6}}$--algebra
\[
(\cB_{T\qt}\times\cC_{T\qt},\cC_{T\qt}\times\cA_{T\qt},\cA_{T\qt}\times\cB_{T\qt},\cC_{T\qt}\times\cB_{T\qt},\cA_{T\qt}\times\cC_{T\qt},\cB_{T\qt}\times\cA_{T\qt}).
\]
Hence $\Xi(\alpha_{\cE})|_{\Sigma_2(L)\times_S T}$ is the morphism between the appropriate $\cO|_{T^{\sqcup 6}}$--algebras.

Now, if the isomorphism $\Xi(\alpha_\cE)$ exists, then it is uniquely defined by the formula above over those $T\in \Sch_S$ which split $L$ and hence also split $\Sigma_2(L)$. Therefore, we only need to show that the given local formulas are compatible with gluing and hence produce a global isomorphism as desired. Let $T_1,T_2\in \Sch_S$ be two schemes which split $L$ and set $T_{12} = T_1\times_S T_2$. We have an isomorphism of $S$--schemes
\[
\beta \colon (T_1\sqcup T_1 \sqcup T_1)\times_{T_1} T_{12} = T_{12}\sqcup T_{12} \sqcup T_{12} \iso T_{12}\sqcup T_{12} \sqcup T_{12} = (T_2\sqcup T_2 \sqcup T_2)\times_{T_2} T_{12}
\]
given by a permutation $\beta \in \SS_3(T_{12})$. As usual, we assume that $\beta$ is one of the diagonal sections $c=(1\; 2\; 3)$ or $\lambda = (2\; 3)$.

Assume first that $\beta = c$. In this case, because the pullback of $(\cE|_{L\times_S T_2})|_{L\times_S T_{12}}$ along $c$ must agree with $(\cE|_{L\times _S T_1})|_{L\times_S T_{12}}$, we obtain equalities
\begin{align*}
\cA_{T_1}|_{T_{12}} &= \cC_{T_2}|_{T_{12}} \\
\cB_{T_1}|_{T_{12}} &= \cA_{T_2}|_{T_{12}} \\
\cC_{T_1}|_{T_{12}} &= \cB_{T_2}|_{T_{12}}.
\end{align*}
The isomorphisms $(\alpha_\cE|_{L\times_S T_1})|_{L\times_S T_{12}}$ and $(\alpha_\cE|_{L\times_S T_2})|_{L\times_S T_{12}}$ must also be compatible along this pullback, and so we obtain commutative diagrams
\[
\begin{tikzcd}[column sep=10ex]
\Cl(\cA_{T_1})|_{T_{12}} \ar[r,"\alpha_{T_1\cA}|_{T_{12}}"] \ar[d,equals] & \cB_{T_1}|_{T_{12}} \times \cC_{T_1}|_{T_{12}} \ar[d,equals] \\
\Cl(\cC_{T_2})|_{T_{12}} \ar[r,"\alpha_{T_2\cC}|_{T_{12}}"] & \cA_{T_2}|_{T_{12}} \times \cB_{T_2}|_{T_{12}},
\end{tikzcd}\;
\begin{tikzcd}[column sep=10ex]
\Cl(\cB_{T_1})|_{T_{12}} \ar[r,"\alpha_{T_1\cB}|_{T_{12}}"] \ar[d,equals] & \cC_{T_1}|_{T_{12}} \times \cA_{T_1}|_{T_{12}} \ar[d,equals] \\
\Cl(\cA_{T_2})|_{T_{12}} \ar[r,"\alpha_{T_2\cA}|_{T_{12}}"] & \cB_{T_2}|_{T_{12}} \times \cC_{T_2}|_{T_{12}},
\end{tikzcd}
\]
and
\[
\begin{tikzcd}[column sep=10ex]
\Cl(\cC_{T_1})|_{T_{12}} \ar[r,"\alpha_{T_1\cC}|_{T_{12}}"] \ar[d,equals] & \cA_{T_1}|_{T_{12}} \times \cB_{T_1}|_{T_{12}} \ar[d,equals] \\
\Cl(\cB_{T_2})|_{T_{12}} \ar[r,"\alpha_{T_2\cB}|_{T_{12}}"] & \cC_{T_2}|_{T_{12}} \times \cA_{T_2}|_{T_{12}}.
\end{tikzcd}
\]
Now, the pullback of $(\Xi(\alpha_\cE)|_{\Sigma_2(L)\times_S T_2})|_{\Sigma_2(L)\times_S T_{12}}$ along the permutation of $(T_{12})^{\sqcup 6}$ corresponding to the action of $c$ is
\[
(\alpha_{T_2\cC}|_{T_{12}},\alpha_{T_2\cA}|_{T_{12}},\alpha_{T_2\cB}|_{T_{12}},\sw\circ\alpha_{T_2\cC}|_{T_{12}},\sw\circ\alpha_{T_2\cA}|_{T_{12}},\sw\circ\alpha_{T_2\cB}|_{T_{12}})
\]
which must agree with $(\Xi(\alpha_\cE)|_{\Sigma_2(L)\times_S T_1})|_{\Sigma_2(L)\times_S T_{12}}$ which is simply
\[
(\alpha_{T_1\cA}|_{T_{12}},\alpha_{T_1\cB}|_{T_{12}},\alpha_{T_1\cC}|_{T_{12}},\sw\circ\alpha_{T_1\cA}|_{T_{12}},\sw\circ\alpha_{T_1\cB}|_{T_{12}},\sw\circ\alpha_{T_1\cC}|_{T_{12}}).
\]
The commutative diagrams above show precisely that these two isomorphisms agree, and so in this case we are done.

Assume now that $\beta = \lambda$. In this case we obtain equalities
\begin{align*}
\cA_{T_1}|_{T_{12}} &= \cA_{T_2}|_{T_{12}} \\
\cB_{T_1}|_{T_{12}} &= \cC_{T_2}|_{T_{12}} \\
\cC_{T_1}|_{T_{12}} &= \cB_{T_2}|_{T_{12}}.
\end{align*}
and commutative diagrams
\[
\begin{tikzcd}[column sep=10ex]
\Cl(\cA_{T_1})|_{T_{12}} \ar[r,"\alpha_{T_1\cA}|_{T_{12}}"] \ar[d,equals] & \cB_{T_1}|_{T_{12}} \times \cC_{T_1}|_{T_{12}} \ar[d,equals] \\
\Cl(\cA_{T_2})|_{T_{12}} \ar[r,"\sw\circ (\alpha_{T_2\cA}|_{T_{12}})" {yshift=0.5ex}] & \cC_{T_2}|_{T_{12}} \times \cB_{T_2}|_{T_{12}},
\end{tikzcd}\;
\begin{tikzcd}[column sep=10ex]
\Cl(\cB_{T_1})|_{T_{12}} \ar[r,"\alpha_{T_1\cB}|_{T_{12}}"] \ar[d,equals] & \cC_{T_1}|_{T_{12}} \times \cA_{T_1}|_{T_{12}} \ar[d,equals] \\
\Cl(\cC_{T_2})|_{T_{12}} \ar[r,"\sw\circ (\alpha_{T_2\cC}|_{T_{12}})" {yshift=0.5ex}] & \cB_{T_2}|_{T_{12}} \times \cA_{T_2}|_{T_{12}},
\end{tikzcd}
\]
and
\[
\begin{tikzcd}[column sep=10ex]
\Cl(\cC_{T_1})|_{T_{12}} \ar[r,"\alpha_{T_1\cC}|_{T_{12}}"] \ar[d,equals] & \cA_{T_1}|_{T_{12}} \times \cB_{T_1}|_{T_{12}} \ar[d,equals] \\
\Cl(\cB_{T_2})|_{T_{12}} \ar[r,"\sw\circ (\alpha_{T_2\cB}|_{T_{12}})" {yshift=0.5ex}] & \cA_{T_2}|_{T_{12}} \times \cC_{T_2}|_{T_{12}}.
\end{tikzcd}
\]
We then require compatibility between the pullback of $(\Xi(\alpha_\cE)|_{\Sigma_2(L)\times_S T_2})|_{\Sigma_2(L)\times_S T_{12}}$ along the action of $\lambda$, which is
\[
(\sw\circ\alpha_{T_2\cA}|_{T_{12}},\sw\circ\alpha_{T_2\cC}|_{T_{12}},\sw\circ\alpha_{T_2\cB}|_{T_{12}},\alpha_{T_2\cA}|_{T_{12}},\alpha_{T_2\cC}|_{T_{12}},\alpha_{T_2\cB}|_{T_{12}}),
\]
and $(\Xi(\alpha_\cE)|_{\Sigma_2(L)\times_S T_1})|_{\Sigma_2(L)\times_S T_{12}}$, which is still as above in the previous case. Once again, the commutative diagrams show that these two morphisms agree, so we are done.
\end{proof}

Now that we have a trialitarian triple
\[
(\pi_L^*(\cE_\qt),\rho_{\Sigma_2(L)}^*(\pi_L^*(\cE_\qt)),\rho_{\Sigma_2(L)}^{*2}(\pi_L^*(\cE_\qt)),\Xi(\alpha_\cE))
\]
over $\Sigma_2(L)$, we have two ways of generating an isomorphism
\[
\Cl\big(\rho_{\Sigma_2(L)}^*(\pi_L^*(\cE_\qt))\big) \iso \rho_{\Sigma_2(L)}^{*2}(\pi_L^*(\cE_\qt)) \times \pi_L^*(\cE_\qt),
\]
namely by applying $\Theta_{\Sigma_2(L)}^+$ to $\Xi(\alpha_\cE)$ or by pulling $\Xi(\alpha_\cE)$ back along $\rho_{\Sigma_2(L)}$. We use these in the following definition.

\begin{defn}\label{defn_TriAlg}
A \emph{trialitarian algebra} (over $S$) is a triple $(L,(\cE,\sigma,f),\alpha_{\cE})$ where
\begin{enumerate}[label={\rm(\roman*)}]
\item \label{defn_TriAlg_i} $L \to S$ is a degree $3$ \'etale cover,
\item \label{defn_TriAlg_ii} $(\cE,\sigma,f)$ is a quadratic triple over $L$ where $\cE$ is a degree $8$ Azumaya $\cO|_L$--algebra, and
\item \label{defn_TriAlg_iii} $\alpha_{\cE} \colon (\Cl(\cE,\sigma,f),\Csigma,\Cf) \iso (\pi_L)_*(\tau_L^*(\cE,\sigma,f))$ is an isomorphism of algebras with semi-trace such that
\[
\Theta_{\Sigma_2(L)}^+(\Xi(\alpha_\cE)) = \rho_{\Sigma_2(L)}^*(\Xi(\alpha_\cE))
\]
as isomorphisms of $\cO|_{\Sigma_2(L)}$--algebras, where $\Xi(\alpha_\cE)$ is the isomorphism appearing in the trialitarian triple of \Cref{triple_over_Sigma(L)}.
\end{enumerate}
\end{defn}

\begin{ex}\label{split_rho_calculation}
Let $L=S^{\sqcup 3}$ be the split degree $3$ \'etale cover, so $\Sigma_2(L)=S^{\sqcup 6}$, $\pi_L = \pi$, $\tau_L = \tau$, and $\rho_{\Sigma_2(L)} = \rho$. A degree $8$ quadratic triple over $L$ is of the form
\[
(\cE,\sigma,f) = ((\cA,\sigma_{\cA},f_{\cA}),(\cB,\sigma_{\cB},f_{\cB}),(\cC,\sigma_{\cC},f_{\cC})).
\]
By applying \Cref{split_E_twist} while also keeping track of semi-traces, we know that
\[
\pi_*(\tau^*(\cE,\sigma,f)) \cong (\cB_\qt \times \cC_\qt, \cC_\qt \times \cA_\qt, \cA_\qt \times \cB_\qt).
\]
The Clifford algebra of $(\cE,\sigma,f)$ is
\[
\Cl(\cE_\qt) = (\Cl(\cA_\qt),\Cl(\cB_\qt),\Cl(\cC_\qt)).
\]
where we omit writing the canonical involutions and semi-traces on the Clifford algebras. An isomorphism $\alpha_{\cE} \colon (\Cl(\cE_\qt),\underline{\sigma},\underline{f})\iso \pi_*(\tau^*(\cE_\qt))$ is the data of three isomorphisms
\begin{align*}
\alpha_1 \colon (\Cl(\cA_\qt),\underline{\sigma_{\cA}},\underline{f_{\cA}}) &\iso \cB_\qt \times \cC_\qt \\
\alpha_2 \colon (\Cl(\cB_\qt),\underline{\sigma_{\cB}},\underline{f_{\cB}}) &\iso \cC_\qt \times \cA_\qt \\
\alpha_3 \colon (\Cl(\cC_\qt),\underline{\sigma_{\cC}},\underline{f_{\cC}}) &\iso \cA_\qt \times \cB_\qt.
\end{align*}
Going back up to $\Sigma_2(L)$, we have that
\[
\Cl(\pi^*(\cE_\qt)) = (\Cl(\cA_\qt),\Cl(\cB_\qt),\Cl(\cC_\qt),\Cl(\cA_\qt),\Cl(\cB_\qt),\Cl(\cC_\qt))
\]
and
\begin{align*}
&\rho^*(\pi^*(\cE_\qt))\times\rho^{*2}(\pi^*(\cE_\qt))\\
= &(\cB_\qt\times\cC_\qt,\cC_\qt\times\cA_\qt,\cA_\qt\times\cB_\qt,\cC_\qt\times\cB_\qt,\cA_\qt\times\cC_\qt,\cB_\qt\times\cA_\qt).
\end{align*}
Then, $\Xi(\alpha_\cE) = (\alpha_1,\alpha_2,\alpha_3,\sw\circ\alpha_1,\sw\circ\alpha_2,\sw\circ\alpha_3)$ and so it is clear that
\[
\rho^*(\Xi(\alpha_\cE)) = (\alpha_2,\alpha_3,\alpha_1,\sw\circ\alpha_3,\sw\circ\alpha_1,\sw\circ\alpha_2).
\]
Now, we must identify $\Theta_{\Sigma_2(L)}^+(\Xi(\alpha_\cE))$. Notice that, because $\Sigma_2(L)=S^{\sqcup 6}$, the trialitarian triple
\[
(\pi_L^*(\cE_\qt),\rho_{\Sigma_2(L)}^*(\pi_L^*(\cE_\qt)),\rho_{\Sigma_2(L)}^{*2}(\pi_L^*(\cE_\qt)),\Xi(\alpha_\cE)) \in \fTrip(\Sigma_2(L))
\]
is simply the data of six triples in $\fTrip(S)$, namely the objects
\begin{align*}
&(\cA_\qt,\cB_\qt,\cC_\qt,\alpha_1) & &(\cB_\qt,\cC_\qt,\cA_\qt,\alpha_2) & &(\cC_\qt,\cA_\qt,\cB_\qt,\alpha_3) \\
&(\cA_\qt,\cC_\qt,\cB_\qt,\sw\circ\alpha_1) & &(\cB_\qt,\cA_\qt,\cC_\qt,\sw\circ\alpha_2) & &(\cC_\qt,\cB_\qt,\cA_\qt,\sw\circ\alpha_3).
\end{align*}
Since any isomorphism with $\cM_{\qt,\Sigma_2(L)}^3 \in \fTrip(\Sigma_2(L))$ will also just be six different isomorphisms between the above triples and $\cM_\qt^3 \in \fTrip(S)$, we see that applying $\Theta_{\Sigma_2(L)}^+$, which produces the isomorphism from \Cref{build_alpha_B}\ref{build_alpha_B_iii}, is the same as applying $\Theta_S^+$ to each of the six triples. Therefore,
\begin{align*}
& \Theta_{\Sigma_2(L)}^+(\Xi(\alpha_\cE)) \\
=& (\Theta_S^+(\alpha_1),\Theta_S^+(\alpha_2),\Theta_S^+(\alpha_3),\Theta_S^+(\sw\circ\alpha_1),\Theta_S^+(\sw\circ\alpha_2),\Theta_S^+(\sw\circ\alpha_1)) \\
=& (\Theta_S^+(\alpha_1),\Theta_S^+(\alpha_2),\Theta_S^+(\alpha_3),\sw\circ\Theta_S^-(\alpha_1),\sw\circ\Theta_S^-(\alpha_2),\sw\circ\Theta_S^-(\alpha_3)).
\end{align*}
Thus, $(L,\cE_\qt,\alpha_{\cE})=(S^{\sqcup 3},(\cA_\qt,\cB_\qt,\cC_\qt),(\alpha_1,\alpha_2,\alpha_3))$ is a trialitarian algebra if and only if $\alpha_2 = \Theta_S^+(\alpha_1)$ and $\alpha_3 = \Theta_S^+(\alpha_2)$.
\end{ex}

\subsection{The Groupoid and the Gerbe}
The definition appearing below of an isomorphism between trialitarian algebras is not surprising, and with it we obtain a groupoid of trialitarian algebras. Ultimately, we will assemble these groupoids into a gerbe which is equivalent to the gerbe of $(\PGO_8^+ \rtimes \SS_3)$--torsors. Recall, over a connected scheme $T$ we view $\SS_3(T)$ as generated by the automorphisms $\theta^+|_T$ and $\theta|_T$ of \Cref{automorphisms_of_groups}. The group $\SS_3$ acts on $\PGO_8^+$ within the semi-direct product by $\phi \cdot \beta = \beta \cdot \beta^{-1}(\phi)$ for $\beta \in \SS_3$ and $\phi \in \PGO_8^+$.

\begin{defn}\label{defn_TriAlg_groupoid}
Let $\fTrial(S)$ be the groupoid whose
\begin{enumerate}
\item \label{defn_TriAlg_groupoid_i} objects are trialitarian algebras $(L,\cE_\qt,\alpha_{\cE})$ over $S$, and whose
\item \label{defn_TriAlg_groupoid_ii} morphisms are pairs $(h,\psi)\colon (L',\cE'_\qt,\alpha_{\cE'}) \to (L,\cE_\qt,\alpha_{\cE})$ where $h\colon L' \iso L$ is an isomorphism of $S$--schemes and $\psi \colon \cE'_\qt \iso h^*(\cE_\qt)$ is an isomorphism of quadratic triples over $L'$ such that
\[
\begin{tikzcd}
(\Cl(\cE'_\qt),\underline{\sigma'},\underline{f'}) \arrow{r}{\alpha_{\cE'}} \arrow{d}{\Cl(\psi)} & (\tau_{L'})_*(\pi_{L'}^*(\cE'_\qt)) \arrow{d}{(\tau_{L'})_*(\pi_{L'}^*(\psi))}\\
(\Cl(h^*(\cE_\qt)),\underline{h^*(\sigma)},\underline{h^*(f)}) \arrow[equals]{d} & (\tau_{L'})_*(\pi_{L'}^*(h^*(\cE_\qt))) \arrow[equals]{d}{\rotatebox{90}{$\sim$}} \\[-3ex]
h^*(\Cl(\cE_\qt),\underline{\sigma},\underline{f}) \arrow{r}{h^*(\alpha_{\cE})} & h^*((\tau_L)_*(\pi_L^*(\cE_\qt)))
\end{tikzcd}
\]
commutes. Composition is given by $(h_1,\psi_1)\circ(h_2,\psi_2) = (h_1\circ h_2, h_2^*(\psi_1)\circ \psi_2)$.
\end{enumerate}
\end{defn}
We designate $\cT\spl = (S^{\sqcup 3},(\Mat_8(\cO)_\qt,\Mat_8(\cO)_\qt,\Mat_8(\cO)_\qt),(\Psi,\Psi,\Psi)) \in \fTrial(S)$ as the split object. The notion of trialitarian algebra is stable under base change because finite \'etale covers, Azumaya algebras, and Clifford algebras are all stable. If $g\colon T' \to T$ is any morphism of $S$--schemes and $(L,\cE_\qt,\alpha_{\cE}) \in \fTriAlg(T)$, then $T'\times_T L \to T'$ will be a degree $3$ \'etale cover of $T'$ and $(T'\times_T L, \cE_\qt|_{T'\times_T L},\alpha_{\cE}|_{T'\times_T L})$ will be a trialitarian algebra over $T'$. We denote this by $g^*(L,\cE_\qt,\alpha_{\cE})$ or simply by $(L,\cE_\qt,\alpha_{\cE})|_{T'}$. It is then clear that the definition below produces a fibered category.

\begin{defn}\label{defn_TriAlg_stack}
Let $\fTriAlg \to \Sch_S$ be the fibered category defined as follows.
\begin{enumerate}
\item The objects are of the form $(T,L,\cE_\qt,\alpha_{\cE})$ where $T\in\Sch_S$ is any $S$--scheme and $(L,\cE_\qt,\alpha_{\cE})\in \fTriAlg(T)$ is a trialitarian algebra over $T$.
\item The morphisms are pairs $(g,\phi)\colon (T',L',\cE_\qt',\alpha_{\cE'}) \to (T,L,\cE_\qt,\alpha_{\cE})$ where $g\colon T' \to T$ is any $S$-scheme morphism and
\[
\phi \colon (L',\cE'_\qt,\alpha_{\cE'}) \iso g^*(L,\cE_\qt,\alpha_{\cE}) = (T'\times_T L, \cE_\qt|_{T'\times_T L},\alpha_{\cE}|_{T'\times_T L})
\]
is an isomorphism in $\fTriAlg(T')$ of trialitarian algebras over $T'$. Composition is given by $(g_1,\phi_1)\circ(g_2,\phi_2) = (g_1\circ g_2, g_2^*(\phi_1)\circ \phi_2)$.
\item The structure functor $\fTriAlg \to \Sch_S$ is given by
\[
(T,L,\cE_\qt,\alpha_{\cE}) \mapsto T \text{ and } (g,\phi) \mapsto g.
\]
\end{enumerate}
\end{defn}
By construction, the fiber over a scheme $T\in \Sch_S$ is the groupoid $\fTriAlg(T)$ of \Cref{defn_TriAlg_groupoid}. At this point, we can define a morphism of fibered categories from trialitarian triples into trialitarian algebras which is intuitively an inclusion. Recall that for $(T,(\cA_\qt,\cB_\qt,\cC_\qt,\alpha_{\cA})) \in \fTrip$ we write $\alpha_{\cB}=\Theta^+_T(\alpha_{\cA})$ for the isomorphism appearing in the object $\Theta^+_T(\cA_\qt,\cB_\qt,\cC_\qt,\alpha_{\cA})$ and similarly $\alpha_{\cC}=\Theta^{+2}_T(\alpha_{\cA})$.
\begin{lem}\label{inc_Trip_TriAlg}
There is a morphism of fibered categories $\frc \colon \fTrip \to \fTriAlg$ defined as follows.
\begin{align*}
(T,(\cA_\qt,\cB_\qt,\cC_\qt,\alpha_{\cA}))&\mapsto (T,T^{\sqcup 3},\cE_\qt = (\cA_\qt,\cB_\qt,\cC_\qt),\alpha_{\cE}=(\alpha_{\cA},\alpha_{\cB},\alpha_{\cC})) \\
(g,(\phi_1,\phi_2,\phi_3)) &\mapsto \big(g,\big(h=g^{\sqcup 3},\psi=(\phi_1,\phi_2,\phi_3)\big)\big) 
\end{align*}
on objects and morphisms respectively. Furthermore, this is a faithful functor and it is essentially surjective onto objects of the form $(T',T'^{\sqcup 3},\cE_{\qt},\alpha_{\cE})$, i.e., those with split degree $3$ \'etale cover.
\end{lem}
\begin{proof}
First, it is clear this is well-defined on objects, i.e., that $\frc$ indeed produces objects in $\fTriAlg$, since $\frc(T,(\cA_\qt,\cB_\qt,\cC_\qt,\alpha_{\cA}))$ matches the description of trialitarian algebras with split \'etale cover given in \Cref{split_rho_calculation}. It is clear as well from \Cref{split_rho_calculation} that, if $\frc$ is a functor, then it is essentially surjective onto such trialitarian algebras.

Now, consider a morphism
\[
(g,(\phi_1,\phi_2,\phi_3)) \colon (T',(\cA'_\qt,\cB'_\qt,\cC'_\qt,\alpha_{\cA'})) \to (T,(\cA_\qt,\cB_\qt,\cC_\qt,\alpha_{\cA}))
\]
in $\fTrip$. Using the identification $\Cl(g^*(\cA_\qt))_\qt = g^*(\Cl(\cA_\qt)_\qt)$, this means we have a commutative diagram
\[
\begin{tikzcd}
\Cl(\cA'_\qt)_\qt \ar[r,"\alpha_{\cA'}"] \ar[d,"\Cl(\phi_1)"] & \cB'_\qt \times \cC'_\qt \ar[d,"\phi_2\times \phi_3"] \\
g^*(\Cl(\cA_\qt)_\qt) \ar[r,"g^*(\alpha_{\cA})"] & \cB_\qt \times \cC_\qt.
\end{tikzcd}
\]
In turn, through the application of $\Theta^+_{T'}$, this produces two more commutative diagrams with the roles of $\cA_\qt$, $\cB_\qt$, and $\cC_\qt$ cyclically permuted and $\alpha_{\cA}$ replaced with $\alpha_{\cB}$ and $\alpha_{\cC}$ as appropriate. Thus, we combine them into the diagram
\[
\begin{tikzcd}[column sep = 4ex]
(\Cl(\cA'_\qt)_\qt,\Cl(\cB'_\qt)_\qt,\Cl(\cC'_\qt)_\qt) \ar[r,"{(\alpha_{\cA'},\alpha_{\cB'},\alpha_{\cC'})}"] \ar[d,"{(\Cl(\phi_1),\Cl(\phi_2),\Cl(\phi_3)}" description] & (\cB'_\qt\times\cC'_\qt,\cC'_\qt\times\cA'_\qt,\cA'_\qt\times\cB'_\qt) \ar[d,"{(\phi_2\times\phi_3,\phi_3\times\phi_1,\phi_1\times\phi_2)}" description]\\[2ex]
(g^{\sqcup 3})^*(\Cl(\cA_\qt)_\qt,\Cl(\cB_\qt)_\qt,\Cl(\cC_\qt)_\qt) \ar[r,"{(\alpha_{\cA},\alpha_{\cB},\alpha_{\cC})}" yshift=1ex] & (g^{\sqcup 3})^*(\cB_\qt\times\cC_\qt,\cC_\qt\times\cA_\qt,\cA_\qt\times\cB_\qt)
\end{tikzcd}
\]
which is the required commutative diagram of \ref{defn_TriAlg_groupoid}\ref{defn_TriAlg_groupoid_ii}, thus our proposed morphism $\frc(g,(\phi_1,\phi_2,\phi_3))$ is a morphism in $\fTriAlg$.

Given two composable morphisms $(g,(\phi_1,\phi_2,\phi_3)$ and $(g',(\phi'_1,\phi'_2,\phi'_3))$ in $\fTrip$, their composition is
\[
(g\circ g', (g'^*(\phi_1)\circ \phi'_1,g'^*(\phi_2)\circ \phi'_2,g'^*(\phi_3)\circ \phi'_3)).
\]
After passing through $\frc$, the composition $\frc(g,(\phi_1,\phi_2,\phi_3))\circ \frc(g',(\phi'_1,\phi'_2,\phi'_3))$ is
\begin{align*}
&\big(g\circ g',\big(g^{\sqcup 3}\circ g'^{\sqcup 3},(g'^{\sqcup 3})^*(\phi_1,\phi_2,\phi_3) \circ (\phi'_1,\phi'_2,\phi'_3)\big)\big) \\
=& \big(g\circ g',\big((g\circ g')^{\sqcup 3},(g'^*(\phi_1)\circ \phi'_1,g'^*(\phi_2)\circ \phi'_2,g'^*(\phi_3)\circ \phi'_3)\big)\\
=& \frc\big((g,(\phi_1,\phi_2,\phi_3)\circ(g',(\phi'_1,\phi'_2,\phi'_3))\big)
\end{align*}
and therefore $\frc$ is a functor.

If $\frc(g,(\phi_1,\phi_2,\phi_3)) = \frc(g',(\phi'_1,\phi'_2,\phi'_3))$ for two morphisms then it is immediate that $g=g'$ and thus $\phi_i=\phi'_i$, so $\frc$ is faithful as claimed. 

It is obvious that $\frc$ respects the structure functors $\fTrip \to \Sch_S$ and $\fTriAlg \to \Sch_S$, and since $\fTriAlg$ is fibered in groupoids this is sufficient for $\frc$ to be a morphism of fibered categories. This finishes the proof.
\end{proof}

\begin{lem}\label{TriAlg_gerbe}
The fibered category $\fTriAlg$ of \Cref{defn_TriAlg_stack} is a gerbe.
\end{lem}
\begin{proof}
As before, we must first convince ourselves that $\fTriAlg$ is a stack. Beginning with the homomorphism presheaves, let $(T,\cT_1)=(T,L_1,\cE_{1\qt},\alpha_{\cE_1})$ and $(T,\cT_2)=(T,L_2,\cE_{2\qt},\alpha_{\cE_2})$ be two objects in the same fiber. Let $\{T_i \to T\}_{i\in I}$ be a cover of $T$. Assume we are given two morphisms
\begin{align*}
(g_1,\phi_1),(g_2,\phi_2) &\in \cIsom((T,\cT_1),(T,\cT_2))(T)\\
 &= \Hom_{\fTriAlg(T)}((L_1,\cE_{1\qt},\alpha_{\cE_1}),(L_2,\cE_{2\qt},\alpha_{\cE_2}))
\end{align*}
which agree over all $T_i$. By \cite[Tag 040L]{Stacks}, the functor
\begin{align*}
\cHom_{\Sch_T}(L_1,L_2)\colon \Sch_T &\to \Sets \\
T' &\to \Hom_{T'}(L_1\times_T T', L_2\times_T T')
\end{align*}
is an fppf--sheaf and therefore we conclude $g_1=g_2$ and denote them simply as $g$. But then, we have two maps of algebras
$\phi_1,\phi_2 \colon \cE_{1\qt} \iso g^*(\cE_{2\qt})$ which agree locally, and so we also have $\phi_1=\phi_2$. Hence, we have a separated presheaf.

If we are given local morphisms
\[
(g_i,\phi_i)\colon (L_1\times_T T_i,\cE_{1\qt}|_{L_1\times_T T_i},\alpha_{\cE_1}|_{L_1\times_T T_i})\iso (L_2\times_T T_i,\cE_{2\qt}|_{L_2\times_T T_i},\alpha_{\cE_2}|_{L_2\times_T T_i})
\]
which are compatible on overlaps, then once again \cite[Tag 040L]{Stacks} guarantees that the $g_i$ glue into a global $g\colon L_1 \iso L_2$ whose restriction to $T_i$ is $g_i$. Therefore, the maps $\phi_i$ are compatible local isomorphisms between the algebras $\cE_{1\qt}$ and $g^*(\cE_{2\qt})$, so we also have a glued isomorphism $\phi \colon \cE_{1\qt} \iso g^*(\cE_{2\qt})$. It is an isomorphism of quadratic triples, but we still must check that is qualifies as an isomorphism of trialitarian algebras. Recall that by \Cref{Sigma(L)_pullback} we have that $\Sigma_2(L_1\times_T T_i)\cong \Sigma_2(L_1)\times_T T_i$. The maps $\tau_{L_1},\pi_{L_1} \colon \Sigma_2(L_1) \to L_1$ are compatible with restrictions, i.e., $\tau_{L_1}|_{L_1\times_T T_i} = \tau_{L_1\times_T T_i}$ and likewise for $\pi_{L_1}$. Of course, analogous statements hold for $L_2$. Therefore, if we consider the diagram of \Cref{defn_TriAlg_groupoid}\ref{defn_TriAlg_groupoid_ii} for this case and restrict all maps to $L_1\times_T T_i$, we will obtain the commutative diagrams corresponding to \Cref{defn_TriAlg_groupoid}\ref{defn_TriAlg_groupoid_ii} for the maps $(g_i,\phi_i) \in \fTriAlg(T_i)$. Hence, since our desired diagram commutes locally it must commute globally as well and so $(g,\phi)$ is a morphism of trialitarian algebras.

Second, we must argue that $\fTriAlg$ allows gluing of objects. Still working with a cover $\{T_i\to T\}_{i\in I}$ in $\Sch_S$, assume we have objects $(T_i,L_i,\cE_{i\qt},\alpha_{\cE_i}) \in \fTriAlg(T_i)$ and gluing data isomorphisms
\[
(g_{ij},(h_{ij},\phi_{ij}))\colon (T_i,L_i,\cE_{i\qt},\alpha_{\cE_i})|_{T_{ij}} \iso (T_j,L_j,\cE_{j\qt},\alpha_{\cE_j})|_{T_{ij}}
\]
in $\fTriAlg$. If we restrict just to the information of the degree $3$ \'etale covers, i.e., $(T_i,L_i)$ and morphisms $(g_{ij},h_{ij})$, then we have diagrams
\[
\begin{tikzcd}
L_i\times_{T_i} T_{ij} \ar[r,"h_{ij}"] \ar[dr] & (L_j\times_{T_j} T_{ij})\times_{g_{ij}} T_{ij} \ar[r,"p"] \ar[d] & L_j\times_{T_j} T_{ij} \ar[d] \\
& T_{ij} \ar[r,"g_{ij}"] & T_{ij}.
\end{tikzcd}
\]
where $p$ is the projection coming from the fiber product. This makes the collection of $(p\circ h_{ij},g_{ij})$ a family of morphisms in the stack of finite \'etale covers $F\acute{E}t$ of \cite[0BLY]{Stacks}, which by \cite[0BLZ]{Stacks} is indeed a stack. In fact, we have gluing data in this stack and therefore we have a glued degree $3$ \'etale cover $L\to T$. Furthermore, since $L\times_T T_i \cong L_i$, we have a cover $\{L_i \to L\}_{i\in I}$ of $L$. Then, the $\phi_{ij}$ becomes gluing data in the stack of quadratic triples over $L$ and hence there is a glued quadratic triple $\cE_\qt$. Locally it is degree $8$, so it is a degree $8$ $\cO|_L$--algebra. Since morphisms between sheaves also allow gluing and the $\phi_{ij}$ are compatible with the isomorphisms $\alpha_{\cE_i}|_{T_{ij}}$ and $\alpha_{\cE_j}|_{T_{ij}}$, we will have a glued isomorphism $\alpha_{\cE}$ over $L$. We must verify that $(L,\cE_\qt,\alpha_{\cE})$ is a trialitarian algebra, i.e., that is satisfies \Cref{defn_TriAlg}\ref{defn_TriAlg_iii}. However, the isomorphism $\Xi(\alpha_\cE)$ of \Cref{triple_over_Sigma(L)} is constructed by gluing local constructions, and therefore it is clear that $\Xi(\alpha_{\cE})|_{\Sigma_2(L_i)} = \Xi(\alpha_{\cE_i})$. Then, because $\Theta^+$ of \Cref{defn_endo_Trip}\ref{defn_endo_Trip_i} is a stack morphisms and also respects base change,
\[
\Theta^+_{\Sigma_2(L)}(\Xi(\alpha_{\cE}))|_{\Sigma_2(L_i)} = \Theta^+_{\Sigma_2(L_i)}(\Xi(\alpha_{\cE_i}))=\rho_{\Sigma_2(L_i)}^*(\Xi(\alpha_{\cE_i})) = \rho_{\Sigma_2(L)}^*(\Xi(\alpha_{\cE}))|_{\Sigma_2(L_i)}.
\]
Since $\{L_i \cong L\times_T T_i \to L\}_{i\in I}$ is a cover and
\[
\Sigma_2(L_i)\cong \Sigma_2(L\times_T T_i) \cong \Sigma_2(L)\times_T T_i \cong \Sigma_2(L)\times_L L_i,
\]
where the second isomorphism is by \Cref{Sigma(L)_pullback}, we see that $\{\Sigma_2(L_i) \to \Sigma_2(L)\}_{i\in I}$ is another cover. By above, the maps $\Theta^+_{\Sigma_2(L)}(\Xi(\alpha_{\cE}))$ and $\rho_{\Sigma_2(L)}^*(\Xi(\alpha_{\cE}))$ agree locally over this cover and so we conclude that
\[
\Theta^+_{\Sigma_2(L)}(\Xi(\alpha_{\cE})) = \rho_{\Sigma_2(L)}^*(\Xi(\alpha_{\cE})).
\]
Thus, $(L,\cE_\qt,\alpha_{\cE})$ is a trialitarian algebra, meaning that $\fTriAlg$ allows gluing. This concludes the proof that $\fTriAlg$ is a stack.

The claim that $\fTriAlg$ is a gerbe is now almost immediate. Given any object $(T,L,\cE_\qt,\alpha_{\cE})\in \fTriAlg$, there is a cover which splits $L$, meaning that it locally becomes of the form $(T',T'^{\sqcup 3},(\cA_\qt, \cB_\qt,\cC_\qt),(\alpha,\Theta^+(\alpha),\Theta^{+2}(\alpha)))$, which is the image of the trialitarian triple $(T',(\cA_\qt,\cB_\qt,\cC_\qt,\alpha)) \in \fTrip(T')$ under the morphism $\frc$ of \Cref{inc_Trip_TriAlg}. Then, since $\fTrip$ is a gerbe by \ref{Trip_properties}\ref{Trip_properties_i}, all such objects are locally isomorphic to $\frc(T',\cM_\qt^3|_{T'}) \cong \frc(S,\cM_\qt^3)|_{T'}$. Thus all objects of $\fTriAlg$ are locally isomorphic to the split object $\frc(S,\cM_\qt^3) = (S,\cT\spl)$, hence $\fTriAlg$ is a gerbe and we are done.
\end{proof}

\begin{lem}\label{TriAlg_automorphism_sheaf}
The automorphism sheaf of the split object $(S,\cT\spl) \in \fTriAlg(S)$ is
\[
\bAut(S,\cT\spl) \cong \PGO_8^+\rtimes \SS_3.
\]
In particular, $\fTriAlg$ is equivalent to the gerbe of $\PGO_8^+\rtimes \SS_3$--torsors.
\end{lem}
\begin{proof}
To ease notation, let $\cE_\qt = (\Mat_8(\cO)_\qt,\Mat_8(\cO)_\qt,\Mat_8(\cO)_\qt)$ be the split degree $8$ Azumaya algebra over $S^{\sqcup 3}$. To keep track of positions, let $\cE_\qt = (\cA_1,\cA_2,\cA_3)$ with all $\cA_i = \Mat_8(\cO)_\qt$. Recall that $\tau_{S^{\sqcup 3}} = \tau$, $\pi_{S^{\sqcup 3}} = \pi$, and we have that
\[
\tau_*(\pi^*(\cA_1,\cA_2,\cA_3)) = (\cA_2\times \cA_3,\cA_3\times \cA_1,\cA_1\times \cA_2)
\]
from \Cref{split_E_twist}. As part of the definition of morphisms in \ref{defn_TriAlg_groupoid}\ref{defn_TriAlg_groupoid_ii}, we use a canonical isomorphism
\[
\tau_*(\pi^*(h^*(\cE_\qt))) \cong h^*(\tau_*(\pi^*(\cE_\qt)))
\]
for an element $h\in \SS_3(S)$. 

To investigate this more closely, we may assume without loss of generality that we are dealing with a diagonal section and first consider the cycle $c = (1\; 2\; 3) \in \SS_3(S)$. Then,
\begin{align*}
\tau_*(\pi^*(c^*(\cA_1,\cA_2,\cA_3))) &= \tau_*(\pi^*(\cA_2,\cA_3,\cA_1)) \\ 
&= (\cA_3\times\cA_1,\cA_1\times\cA_2,\cA_2\times\cA_3)
\intertext{and}
c^*(\tau_*(\pi^*(\cA_1,\cA_2,\cA_3))) &= c^*(\cA_2\times \cA_3,\cA_3\times \cA_1,\cA_1\times \cA_2) \\
&= (\cA_3\times\cA_1,\cA_1\times\cA_2,\cA_2\times\cA_3)
\end{align*}
so in this case the canonical isomorphism is an equality. Furthermore, this is also clearly the case for the identity permutation and $c^2$, i.e., the diagonal permutations in $A_3(S) \subset \SS_3(S)$. However, now consider the transposition $\lambda = (2\; 3)$. We have
\begin{align*}
\tau_*(\pi^*(\lambda^*(\cA_1,\cA_2,\cA_3))) &= \tau_*(\pi^*(\cA_1,\cA_3,\cA_2)) \\ 
&= (\cA_3\times\cA_2,\cA_2\times\cA_1,\cA_1\times\cA_3)
\intertext{and}
\lambda^*(\tau_*(\pi^*(\cA_1,\cA_2,\cA_3))) &= \lambda^*(\cA_2\times \cA_3,\cA_3\times \cA_1,\cA_1\times \cA_2) \\
&= (\cA_2\times\cA_3,\cA_1\times\cA_2,\cA_3\times\cA_1)
\end{align*}
and so now the canonical isomorphism is $(\sw,\sw,\sw)$. This is also the canonical isomorphism for the other diagonal sections $c\lambda$ and $c^2\lambda$ outside of $A_3(S)$ in $\SS_3(S)$. 

Now, let $h\in \SS_3(S)$ be any diagonal section. Since we have $h^*(\cE_\qt)=\cE_\qt$, $\tau_*(\pi^*(\cE_\qt))=\cE_\qt \times \cE_\qt$, and $h^*(\Cl(\cE_\qt)_\qt) = \Cl(\cE_\qt)_\qt$ (in fact these equations hold even for non-diagonal sections), an automorphism $(\phi_1,\phi_2,\phi_3) \colon \cE_\qt \to \cE_\qt$ of quadratic triples over $S^{\sqcup 3}$ will give a section of $\bAut(S,\cT\spl)(S)$ if and only if the diagram
\[
\begin{tikzcd}[column sep=15ex]
\Cl(\cE_\qt)_\qt \ar[r,"\alpha"] \ar[d,"{(\Cl(\phi_1),\Cl(\phi_2),\Cl(\phi_3))}" description] & \cE_\qt \times \cE_\qt \ar[d,"{(\phi_2\times\phi_3,\phi_3\times\phi_1,\phi_1\times\phi_2)}" description] \\
\Cl(\cE_\qt)_\qt \ar[r,"\beta"] & \cE_\qt \times \cE_\qt
\end{tikzcd}
\]
commutes, where $\beta = \alpha$ if $h\in A_3(S)$ or $\beta = \sw\circ\alpha$ if $h\notin A_3(S)$. In other words, this is the information of
\begin{enumerate}
\item \label{TriAlg_automorphism_sheaf_i} an automorphism $(\phi_1,\phi_2,\phi_3) \in \bAut_{\fTrip(S)}(\cM_\qt^3)(S)$ if $h\in A_3(S)$, or
\item \label{TriAlg_automorphism_sheaf_ii} an isomorphism $(\phi_1,\phi_2,\phi_3) \in \cIsom_{\fTrip(S)}(\cM_\qt^3,\Theta_S(\cM_\qt^3))(S)$ if $h\notin A_3(S)$.
\end{enumerate}
Note that, in case \ref{TriAlg_automorphism_sheaf_i} this means all $\phi_i \in \PGO_8^+(S)$ and we have $\phi_2 = \theta^+(\phi_1)$ and $\phi_3=\theta^+(\phi_1)$. In case \ref{TriAlg_automorphism_sheaf_ii}, they are in $\PGO_8(S)\backslash \PGO_8^+(S)$. Using the isomorphisms of \Cref{Theta_functor_properties}\ref{Theta_functor_properties_i}, where we recall $\pi_{\bO}(\varphi)$ from \Cref{theta_description}\ref{theta_description_iii}, we have that
\[
(\pi_{\bO}(\varphi)\cdot \phi_1,\pi_{\bO}(\varphi)\cdot \phi_2,\pi_{\bO}(\varphi)\cdot \phi_3) \in \bAut_{\fTrip(S)}(\cM_\qt^3)(S)
\]
and so $\theta^+(\pi_{\bO}(\varphi)\cdot \phi_1) = \pi_{\bO}(\varphi)\cdot \phi_2$ and $\theta^+(\pi_{\bO}(\varphi)\cdot \phi_2) = \pi_{\bO}(\varphi)\cdot \phi_3$. Since an arbitrary $h\in \SS_3(S)$ is locally constant, after passing to a cover this discussion is sufficient to describe all of $\bAut(S,\cT\spl)(S)$. The above discussion is also stable under base change and therefore describes $\bAut(T,\cT\spl|_T)$ for any $T \in \Sch_S$ as well. 

Now, we define a morphism of sheaves of groups $\bAut(S,\cT\spl) \to \PGO_8^+ \rtimes \SS_3$ by defining it to act as
\begin{align}
\bAut(S,\cT\spl)(T) &\to (\PGO_8^+ \rtimes \SS_3)(T) \label{eq_TriAlg_automorphism_sheaf} \\
(\Id_S,h,(\phi_1,\phi_2,\phi_3)) &\mapsto \begin{cases} (1,h)(\phi_1,1) & h \in A_3(T) \\ (1,h)(\pi_{\bO}(\varphi) \cdot \phi_1,1) & h \notin A_3(T) \end{cases} \nonumber
\end{align}
for diagonal sections $h\in \SS_3(T)$, which then uniquely defines the behaviour on all sections. We claim that this is an isomorphism of groups. It is sufficient to argue that it is an isomorphism between the subpresheaf of automorphisms where $h$ is a diagonal section and the subpresheaf $\PGO_8^+\rtimes \PP_3$ where $\PP_3$ is the constant presheaf associated to $P_3$. The discussion above makes it clear that it is a bijection of presheaves with an inverse given by sending
\[
(1,h)(\phi,1) \mapsto \begin{cases} (\Id_S,h,(\phi,\theta^+(\phi),\theta^-(\phi))) & h \in A_3(T) \\ (\Id_S,h,(\pi_{\bO}(\varphi)\cdot\phi,\pi_{\bO}(\varphi)\cdot\theta^+(\phi),\pi_{\bO}(\varphi)\cdot\theta^-(\phi))) & h \notin A_3(T) \end{cases}
\]
Thus, it only remains to show that it is a group homomorphism. It is sufficient to consider compositions $(\Id_S,h_1,(\phi_1,\phi_2,\phi_3))\circ (\Id_S,h_2,(\psi_1,\psi_2,\psi_3))$ where $h_2$ is $c$ or $\lambda$.

First, consider
\[
(\Id_S,h_1,(\phi_1,\phi_2,\phi_3))\circ (\Id_S,c,(\psi_1,\psi_2,\psi_3)) = (\Id_S,h_1c,(\phi_2\psi_1,\phi_3\psi_2,\phi_1\psi_3)).
\]
If $h_1 \in A_3(T)$, then the left hand side maps to
\begin{align*}
(1,h_1)(\phi_1,1)(1,c)(\psi_1,1) &= (1,h_1)(1,c)(\theta^+(\phi_1),1)(\psi_1,1) \\
&= (1,h_1c)(\theta^+(\phi_1)\psi_1,1) \\
&= (1,h_1c)(\phi_2\psi_1,1)
\end{align*}
which is the image of the right hand side since $h_1c \in A_3(T)$ also. If $h_1 \notin A_3(T)$, then the left hand side maps to
\begin{align*}
(1,h_1)(\pi_{\bO}(\varphi)\phi_1,1)(1,c)(\psi_1,1) &= (1,h_1c)(\theta^+(\pi_{\bO}(\varphi)\phi_1)\psi_1,1) \\
&= (1,h_1c)(\pi_{\bO}(\varphi)\phi_2\psi_1,1)
\end{align*}
which is the image of the right hand side since $h_1c \notin A_3(T)$ also.

Second, consider 
\[
(\Id_S,h_1,(\phi_1,\phi_2,\phi_3))\circ (\Id_S,\lambda,(\psi_1,\psi_2,\psi_3)) = (\Id_S,h_1\lambda,(\phi_1\psi_1,\phi_3\psi_2,\phi_2\psi_3)).
\]
If $h_1\in A_3(T)$ then the left side maps to
\begin{align*}
(1,h_1)(\phi_1,1)(1,\lambda)(\pi_{\bO}(\varphi)\cdot \psi_1) &= (1,h_1\lambda)(\theta(\phi_1),1)(\pi_{\bO}(\varphi)\cdot \psi_1,1) \\
&= (1,h_1\lambda)(\pi_{\bO}(\varphi)\cdot \phi_1 \cdot \pi_{\bO}(\varphi),1)(\pi_{\bO}(\varphi)\cdot \psi_1,1) \\
&= (1,h_1\lambda)(\pi_{\bO}(\varphi)\cdot \phi_1\psi_1,1)
\end{align*}
which is the image of the right side since $h_1\lambda \notin A_3(T)$. Here we used the description of the automorphism $\theta$ of $\PGO_8^+$ from \Cref{theta_description}\ref{theta_description_iii}. If instead, $h_1 \notin A_3(T)$ then the left side maps to
\begin{align*}
(1,h_1)(\pi_{\bO}(\varphi)\cdot \phi_1,1)(1,\lambda)(\pi_{\bO}(\varphi)\cdot \psi_1) &= (1,h_1\lambda)(\theta(\pi_{\bO}(\varphi)\cdot \phi_1),1)(\pi_{\bO}(\varphi)\cdot \psi_1,1) \\
&= (1,h_1\lambda)(\phi_1\cdot \pi_{\bO}(\varphi),1)(\pi_{\bO}(\varphi)\cdot \psi_1,1) \\
&= (1,h_1\lambda)(\phi_1\psi_1,1)
\end{align*}
which is the image of the right side since $h_1\lambda \in A_3(T)$. Therefore the proposed isomorphism is a group homomorphism and we are done.
\end{proof}

\begin{cor}
The morphism $\frc \colon \fTrip \to \fTriAlg$ of \Cref{inc_Trip_TriAlg} is a morphism of gerbes which sends the split object to the split object, i.e., $\frc(S,\cM_\qt^3) = (S,\cT\spl)$. The corresponding induced homomorphism between automorphism sheaves is the inclusion
\[
\PGO_8^+ \inj \PGO_8^+ \rtimes \SS_3.
\]
Under the corresponding map on cohomology $H^1(S,\PGO_8^+) \to H^1(S,\PGO_8^+\rtimes \SS_3)$, the images of the classes
\begin{align*}
&[\big(S,(\cA_\qt,\cB_\qt,\cC_\qt,\alpha_{\cA})\big)], \\
&[\big(S,\big(\cB_\qt,\cC_\qt,\cA_\qt,\Theta^+_S(\alpha_{\cA})\big)\big)],\text{ and }\\
&[\big(S,\big(\cA_\qt,\cC_\qt,\cB_\qt,\Theta_S(\alpha_{\cA})\big)\big)]
\end{align*}
are all equal.
\end{cor}
\begin{proof}
It is a morphism between gerbes by \Cref{Trip_properties}\ref{Trip_properties_i} and \Cref{TriAlg_gerbe}. The fact that the homomorphism between automorphism sheaves is as claimed is clear from the definition of $\frc$, the isomorphism described in \Cref{Trip_properties}\ref{Trip_properties_ii}, and the description of the automorphism sheaf in \Cref{TriAlg_automorphism_sheaf}. The objects $\frc(S,(\cB_\qt,\cC_\qt,\cA_\qt,\Theta^+_S(\alpha_{\cA})))$ and $\frc(S,(\cA_\qt,\cB_\qt,\cC_\qt,\alpha_{\cA}))$ are isomorphic in $\fTriAlg(S)$ via the map $(\Id_S,c=(1\; 2\; 3),(\Id_{\cB},\Id_{\cC},\Id_{\cA}))$. This shows that the images of the first and second cohomology classes are equal. Similarly, the map $(\Id_S,\lambda=(2\; 3),(\Id_{\cA},\Id_{\cC},\Id_{\cB}))$ will be an isomorphism in $\fTriAlg(S)$, showing that the images of the first and third classes are equal.
\end{proof}

\section{Groups of Type $D_4$}\label{Groups}
In this final section, we define the simply connected and adjoint groups associated to a trialitarian algebra. These constructions will provide equivalences of stacks between the stack of simply connected groups of type $D_4$ and $\fTrial$ or the stack of adjoint groups of type $D_4$ and $\fTrial$. Our definition of the adjoint groups is a direct analogue of the definition in \cite[44.A]{KMRT} and our approach to the simply connected groups is heavily motivated by \cite[44.A]{KMRT}.

\subsection{Adjoint Groups}
The adjoint groups of type $D_4$ are exactly the twisted forms of $\PGO_8^+$. It is therefore clear that the stack of all such groups, which we denote $\fD_4^\adj$ and which has
\begin{enumerate}
\item objects which are pairs $(T,\bG)$ where $T\in \Sch_S$ is any scheme and $\bG$ is a sheaf of groups on $\Sch_T$ for which there exists a cover $\{T_i \to T\}_{i\in I}$ such that $\bG|_{T_i} \cong \PGO_8^+|_{T_i}$,
\item morphisms are pairs $(g,\varphi)\colon (T',\bG') \to (T,\bG)$ where $g\colon T' \to T$ is any morphism in $\Sch_S$ and $\varphi \colon \bG' \iso g^*(\bG)$ is an isomorphism of groups, and
\item structure functor $(T,\bG) \mapsto T$ and $(g,\varphi)\mapsto g$,
\end{enumerate}
is a gerbe which is equivalent to the gerbe of $\bAut(\PGO_8^+)=\PGO_8^+\rtimes \SS_3$--torsors.

We now define a subgroup of the automorphisms of a trialitarian algebra. Recall that for a trialitarian algebra $(T,L,\cE_\qt,\alpha_\cE) \in \fTrial(T)$, the sections of its automorphism sheaf over $T'\in \Sch_T$ are of the form $(g=\Id_{T'},h,\psi)$ where $h\colon L\times_T T' \iso L\times_T T'$ is an automorphism in $\Sch_{T'}$ and $\psi \colon \cE_\qt|_{L\times_T T'} \iso h^*(\cE_\qt|_{L\times_T T'})$ is an automorphism of quadratic triples which respects $\alpha_{\cE}|_{L\times_T T'}$.
\begin{defn}\label{defn_PGO+_T}
Let $(T,\cT)=(T,L,\cE_\qt,\alpha_\cE) \in \fTrial$ be a trialitarian algebra over $T\in \Sch_S$. We define the group $\PGO^+_\cT$ to be the subgroup $\bAut_{\fTrial}(T,\cT)$ consisting of automorphisms which fix $L$. In detail, for $T'\in \Sch_T$
\[
\PGO^+_\cT(T') = \{(\Id_{T'},h,\psi)\in \bAut_{\fTrial}(T,\cT)(T') \mid h=\Id_{L\times_T T'}\}.
\]
\end{defn}

\begin{lem}\label{adjoint_group_prelim}
With the definition of $\PGO^+_\cT$ of a trialitarian algebra given above, we have the following.
\begin{enumerate}
\item \label{adjoint_group_prelim_i} $\PGO^+_{\cT\spl} \cong \PGO_8^+$.
\item \label{adjoint_group_prelim_ii} A morphism $(g,\phi)\colon (T',\cT')\iso (T,\cT)$ in $\fTrial$ induces an isomorphism 
\begin{align*}
\phi^\adj \colon \PGO^+_{\cT'} &\iso g^*(\PGO^+_\cT) = \PGO^+_{g^*(\cT)} \\
\varphi &\mapsto \phi \circ \varphi \circ \phi^{-1}
\end{align*}
of group sheaves over $T'$.
\item \label{adjoint_group_prelim_iii} Given $(T,\cT)\in \fTrial$, the group $\PGO^+_\cT$ is a twisted form of $\PGO_8^+|_T$.
\end{enumerate}
\end{lem}
\begin{proof}
\noindent\ref{adjoint_group_prelim_i}: It is clear that the automorphism of \Cref{TriAlg_automorphism_sheaf}, in particular \Cref{eq_TriAlg_automorphism_sheaf}, restricts to an isomorphism $\PGO^+_{\cT\spl}\iso \PGO_8^+$.

\noindent\ref{adjoint_group_prelim_ii}: The morphism $(g,\phi)$ consist of a scheme morphism $g \colon T' \to T$ and an isomorphism $\phi \colon \cT' \to g^*(\cT)$ in $\fTriAlg(T')$. Therefore, if $\varphi$ is an automorphism of $\cT'$ in $\fTriAlg(T')$, it is clear that $\phi \circ \varphi \circ \phi^{-1}$ is an automorphism of $g^*(\cT)$. Furthermore, let $\cT' = (L',\cE'_\qt,\alpha_{\cE'})$ and $\cT = (L,\cE_\qt,\alpha_{\cE})$. Then, $g^*(\cT) = (L\times_T T',\cE_\qt|_{L\times_T T'},\alpha_{\cE}|_{L\times_T T'})$ and the isomorphism $\phi$ contains a component $h\colon L' \iso L\times_T T'$ which is a isomorphism of $T'$--schemes. Any $\varphi \in \PGO^+_{\cT'}$ acts as the identity on $L'$, so it is clear that $\phi \circ \varphi \circ \phi^{-1}$ acts as $h \circ \Id_{L'} \circ h^{-1} = \Id_{L\times_T T'}$ on $L\times_T T'$ and therefore $\phi \circ \varphi \circ \phi^{-1} \in \PGO^+_{g^*(\cT)}$. Hence, the given morphism is well-defined. It is clear it is an isomorphism and so we are done.

\noindent\ref{adjoint_group_prelim_iii}: Since $\fTriAlg$ is a gerbe by \Cref{TriAlg_gerbe}, the object $(T,\cT)$ is locally isomorphic to $(T,\cT\spl|_T)$, i.e., there is a cover $\{T_i \to T\}_{i\in I}$ over which $(T_i,\cT|_{T_i})\cong (T_i,\cT\spl|_{T_i})$. Therefore by \ref{adjoint_group_prelim_ii}, there are isomorphisms $\PGO^+_{(\cT|_{T_i})} \cong \PGO^+_{(\cT\spl|_{T_i})}$, where this second group is $\PGO^+_{(\cT\spl|_{T_i})} = (\PGO^+_{\cT\spl})|_{T_i} \cong \PGO_8^+|_{T_i}$ by \ref{adjoint_group_prelim_i}. This shows that $\PGO^+_\cT$ is a twisted form of $\PGO_8^+|_T$.
\end{proof}

\Cref{adjoint_group_prelim} shows that the groups $\PGO^+_\cT$ are adjoint group of type $D_4$ and that we have a functor
\begin{align}
\cF^\adj \colon \fTrial &\to \fD_4^\adj \label{F^adj} \\
(T,\cT) &\mapsto (T,\PGO^+_\cT) \nonumber \\
(g,\phi) &\mapsto (g,\phi^\adj). \nonumber
\end{align}
It is a morphism of stacks since it is clearly compatible with the structure functors and $\fD_4^\adj$ is a gerbe. 

\begin{thm}\label{F^ad_equiv}
The morphism $\cF^\adj$ above is an equivalence of gerbes.
\end{thm}
\begin{proof}
We argue using \Cref{equivalence_gerbes}. Consider the object $(S,\cT\spl) \in \fTriAlg(S)$. Recall that by \Cref{TriAlg_automorphism_sheaf}, its automorphism group is $\bAut_{\fTriAlg}(S,\cT\spl)\cong \PGO_8^+\rtimes \SS_3$. By \Cref{adjoint_group_prelim}\ref{adjoint_group_prelim_i}, we have $\cF^\adj(S,\cT\spl) = \PGO_8^+$ which is the obvious subgroup of $\bAut_{\fTriAlg}(S,\cT\spl)$. By \Cref{aut_spin_PGO}, the automorphism group of $\PGO_8^+$ is $\bAut_{\fD_4^\adj}(\PGO_8^+) \cong \PGO_8^+\rtimes \SS_3$. The morphism $\cF^\adj$ induces a group homomorphism between automorphism groups
\[
\bAut(\cF^\adj)\colon \PGO_8^+\rtimes \SS_3 \to \PGO_8^+\rtimes \SS_3.
\]
Applying \Cref{adjoint_group_prelim}\ref{adjoint_group_prelim_ii}, we see that for $\phi \in \PGO_8^+\rtimes \SS_3$, its image under this group homomorphism is the automorphism
\begin{align*}
\bAut(\cF^\adj)(\phi) \colon \PGO_8^+ &\iso \PGO_8^+ \\
\varphi &\mapsto \phi \varphi \phi^{-1}.
\end{align*}
However, by the description of $\bAut_{\fD_4^\adj}(\PGO_8^+)$ in \Cref{aut_spin_PGO}, the elements of $\PGO_8^+\subset \PGO_8^+ \rtimes \SS_3$ represent the inner automorphisms of $\PGO_8^+$, and therefore $\bAut(\cF^\adj)$ acts as the identity on this subgroup. Furthermore, for $\sigma \in \SS_3$ and $\varphi \in \PGO_8^+$ we have
\[
\sigma \varphi \sigma^{-1} = (\sigma\sigma^{-1})\cdot \sigma(\varphi) = \sigma(\varphi)
\]
and so $\sigma$ maps to itself as an automorphism of $\PGO_8^+$, i.e., $\bAut(\cF^\adj)$ is the identity on $\SS_3$ also. This makes it clear that $\bAut(\cF^\adj)$ is an isomorphism and so by \Cref{equivalence_gerbes} we conclude that $\cF^\adj$ is an equivalence as claimed.
\end{proof}

\subsection{Simply Connected Groups}
The simply connected groups of type $D_4$ are exactly the twisted forms of $\Spin_8$. Therefore, the stack denoted $\fD_4^\simpc$ consisting of
\begin{enumerate}
\item objects which are pairs $(T,\bG)$ where $T\in \Sch_S$ is any scheme and $\bG$ is a sheaf of groups on $\Sch_T$ for which there exists a cover $\{T_i \to T\}_{i\in I}$ such that $\bG|_{T_i} \cong \Spin_8|_{T_i}$,
\item morphisms which are pairs $(g,\varphi)\colon (T',\bG') \to (T,\bG)$ where $g\colon T' \to T$ is any morphism in $\Sch_S$ and $\varphi \colon \bG' \iso g^*(\bG)$ is an isomorphism of groups, and
\item structure functor which sends $(T,\bG) \mapsto T$ and $(g,\varphi) \mapsto g$
\end{enumerate}
is a gerbe which is equivalent to the gerbe of $\bAut(\Spin_8)=\PGO_8^+ \rtimes \SS_3$--torsors.

Given a trialitarian algebra $(T,\cT)=(T,L,\cE_\qt,\alpha_\cE)\in \fTriAlg$, the quadratic triple $\cE_\qt$ has an associated spin group as in \Cref{revelant_definitions}, namely $\Spin_{\cE_\qt} \subset \Cl(\cE_\qt)$. The quadratic triple $\cE_\qt$ also has its orthogonal group, $\bO^+_{\cE_\qt}\subset \cE_\qt^\times$ and we have the vector representation $\chi_\cE \colon \Spin_{\cE_\qt} \to \bO^+_{\cE_\qt}$. Thus, we have the following two compositions from $\Spin_{\cE_\qt}$ into $(\pi_L)_*(\tau_L^*(\cE_\qt))^\times$,
\[
\begin{tikzcd}
\Spin_{\cE_\qt} \ar[r,hookrightarrow] & \Cl(\cE_\qt)^\times \ar[r,"\alpha_\cE"] & (\pi_L)_*(\tau_L^*(\cE_\qt))^\times \textnormal{ and} \\[-4ex]
\Spin_{\cE_\qt} \ar[r,"\chi_\cE"] & \bO^+_{\cE_\qt} \subset \cE_\qt^\times \ar[r,"\gamma_\cE"] & (\pi_L)_*(\tau_L^*(\cE_\qt))^\times
\end{tikzcd}
\]
where $\gamma_\cE$ is the map of \Cref{gamma_E-morphism}. Both these maps are group homomorphisms and thus their equalizer is a group sheaf over $L$, which we denote as
\begin{align*}
\preSpin_{\cT} \colon \Sch_L &\to \Grp \\
L' &\mapsto \{a \in \Spin_{\cE_\qt}(L') \mid \alpha_\cE(a) = (\gamma_\cE \circ \chi_\cE)(a)\}.
\end{align*} 
After a pushforward, it is the group we are after.
\begin{defn}\label{defn_Spin_T}
Let $(T,\cT)=(T,L,\cE_\qt,\alpha_\cE)\in \fTriAlg$ be a trialitarian algebra over $T\in \Sch_S$ and let $f\colon L \to T$ be the degree $3$ \'etale cover. We define the spin group of the trialitarian algebra to be $\Spin_\cT = f_*(\preSpin_\cT)$. It is a group sheaf over $T$.
\end{defn}

We keep in mind that there is seemingly some ambiguity in the definition of pushforwards coming from the non-uniqueness of fiber products. Given a trialitarian algebra $(T,\cT)=(T,L,\cE_\qt,\alpha_\cE)$ and a scheme $T' \in \Sch_T$, we may have two fiber products $L'$ and $L\times_T T'$. Of course, these come with an isomorphism $L' \iso L\times_T T'$ of $T$--schemes due to their universal properties. This isomorphism yields an isomorphism of groups $\preSpin_\cT(L\times_T T')\iso \preSpin_\cT(L')$. When discussing the pushforward, we view this isomorphism as an identification and abusively write that $\Spin_\cT(T') = \preSpin_\cT(L\times_T T') = \preSpin_\cT(L')$.

\begin{lem}\label{simpc_group_prelim}
With the definition of $\Spin_\cT$ of a trialitarian algebra given above, we have the following.
\begin{enumerate}
\item \label{simpc_group_prelim_i} $\Spin_{\cT\spl} \cong \Spin_8$.
\item \label{simpc_group_prelim_ii} Let $(g,\phi) \colon (T',\cT') \iso (T,\cT)$ be a morphism in $\fTriAlg$. In detail, $\cT' = (L',\cE'_\qt,\alpha_{\cE'})$, $\cT =(L,\cE_\qt,\alpha_\cE)$, and $\phi=(h,\psi)$ where $h\colon L' \iso L\times_T T'$ is an isomorphism of $T'$--schemes and $\psi \colon \cE'_\qt \iso h^*(\cE_\qt|_{L\times_T T'})$ is an isomorphism of quadratic triples over $L'$ which respects the trialitarian algebra structures. This induces an isomorphism of group sheaves over $L'$
\begin{align*}
\preSpin_{\cT'} &\iso h^*(\preSpin_\cT|_{L\times_T T'}) \\
a &\mapsto \Cl(\psi)(a)
\end{align*}
whose pushforward is an isomorphism of group sheaves over $T'$
\[
\phi^\simpc \colon \Spin_{\cT'} \iso g^*(\Spin_\cT).
\]
\item \label{simpc_group_prelim_iii} Given $(T,\cT) \in \fTriAlg$, the group $\Spin_\cT$ is a twisted form of $\Spin_8|_T$.
\end{enumerate}
\end{lem}
\begin{proof}
\noindent\ref{simpc_group_prelim_i}: In the case of the split trialitarian algebra
\[
(S,\cT\spl) = (S,S^{\sqcup 3},\cE_\qt = (\Mat_8(\cO)_\qt,\Mat_8(\cO)_\qt,\Mat_8(\cO)_\qt),\alpha_\cE = (\Psi,\Psi,\Psi))
\]
we know that
\[
\pi_*(\tau^*(\cE_\qt)) \cong (\Mat_8(\cO)_\qt\times \Mat_8(\cO)_\qt,\Mat_8(\cO)_\qt\times \Mat_8(\cO)_\qt,\Mat_8(\cO)_\qt\times \Mat_8(\cO)_\qt)
\]
from \Cref{split_E_twist}. Also, since $\Cl(\cE_\qt) = (\Cl(\Mat_8(\cO)_\qt),\Cl(\Mat_8(\cO)_\qt),\Cl(\Mat_8(\cO)_\qt))$ it is clear that $\Spin_{\cE_\qt} = (\Spin_8,\Spin_8,\Spin_8)$. From the calculations in \Cref{Spin_Chevalley_computations}, we know that for $(a,b,c)\in \Spin_{\cE_\qt}$ we have
\begin{align*}
&(\Psi,\Psi,\Psi)(a,b,c)\\
= &\big(\big(\chi(\theta^+(a)),\chi(\theta^-(a))\big),\big(\chi(\theta^+(b)),\chi(\theta^-(b))\big),\big(\chi(\theta^+(c)),\chi(\theta^-(c))\big)\big).
\end{align*}
On the other hand, since $S^{\sqcup 3}$ is the split degree $3$ \'etale cover of $S$ we know the explicit description of $\gamma_\cE$, and so
\[
(\gamma_\cE \circ \chi_\cE)(a,b,c) = \gamma_\cE(\chi(a),\chi(b),\chi(c)) = \big((\chi(b),\chi(c)),(\chi(c),\chi(a)),(\chi(a),\chi(b))\big).
\]
Therefore, $(a,b,c) \in \Spin_{\cT\spl}$ if and only if the relations
\begin{align*}
\chi(\theta^+(a))&= \chi(b) & \chi(\theta^-(a)) &= \chi(c) \\
\chi(\theta^+(b))&= \chi(c) & \chi(\theta^-(b)) &= \chi(a) \\
\chi(\theta^+(c))&= \chi(a) & \chi(\theta^-(c)) &= \chi(b)
\end{align*}
hold. If these relations hold, then $\theta^+(a) = xb$ for some $x\in \Ker(\chi)$. In turn, this means that $\theta^-(a)=\theta^+(x)\theta^+(b)$ and so $\chi(\theta^-(a)) = \chi(\theta^+(x))\chi(\theta^+(b))$. However, we have both $\chi(\theta^-(a))=\chi(c)$ and $\chi(\theta^+(b))=\chi(c)$ from the relations and so we must also have $\chi(\theta^+(x))=1$, i.e., $\theta^+(x) \in \Ker(\chi)$ also.

From the discussion in the proof of \Cref{Chevalley_sheafification}, we recall that for $T' \in \Sch_T$ we have $\Ker(\chi)(T') = \{h_{e_3-e_4}(t)h_{e_3+e_4}(t) \mid t\in \runity_2(T')\}$ where
\[
h_\alpha(t) = w_\alpha(t)w_\alpha(-1) = x_\alpha(t)x_{-\alpha}(-t^{-1})x_\alpha(t)x_\alpha(-1)x_{-\alpha}(1)x_\alpha(-1)
\]
for a root $\alpha$ and $t\in \cO(T')^\times$ as in \cite[Lemma 19]{Steinberg}. Using the description of the automorphism $\theta^+$ of $\Spin_8$ from \Cref{theta+_description}, we can compute that $\theta^+(h_{e_3-e_4}(t)) = h_{e_3-e_4}(t)$ and that
\begin{align*}
&\theta^+(h_{e_3+e_4}(t))\\
= &x_{-e_1-e_2}(-t)x_{e_1+e_2}(t^{-1})x_{-e_1-e_2}(-t)x_{-e_1-e_2}(1)x_{e_1+e_2}(-1)x_{-e_1-e_2}(1) \\
= &w_{-e_1-e_2}(-t)w_{-e_1-e_2}(1) \\
= &w_{-e_1-e_2}(-t)w_{-e_1-e_2}(-1)w_{-e_1-e_2}(1)^2 \\
= &h_{-e_1-e_2}(-t)h_{-e_1-e_2}(-1) \\
= &h_{-e_1-e_2}(t).
\end{align*}
Here we use the facts that $w_\alpha(t)^{-1} = w_\alpha(-t)$, that $h_\alpha(t)$ is multiplicative in $t$, and \cite[Lemma 1(3)]{Rue20} which states that in spin groups we have the relation $w_\alpha(t)^2 = h_\alpha(-1)$. Thus, $\theta^+(h_{e_3-e_4}(t)h_{e_3+e_4}(t))$ is not in $\Ker(\chi)$ unless $t=1$ and hence $h_{e_3-e_4}(t)h_{e_3+e_4}(t) = 1$ also.

Therefore, we conclude that $\theta^+(a)=b$ and, by the same argument, that $\theta^+(b) = c$. It is then clear that the map
\begin{align*}
\Spin_{\cT\spl} &\iso \Spin_8 \\
(a,b,c) &\mapsto a
\end{align*}
is an isomorphism with inverse $a \mapsto (a,\theta^+(a),\theta^-(a))$.

\noindent\ref{simpc_group_prelim_ii}: First, we show that we have the claimed isomorphism over $L'$ between $\preSpin$ groups. Let $X\in \Sch_{L'}$. By definition we have that
\[
\preSpin_{\cT'}(X) = \{a \in \Spin_{\cE'_\qt}(X) \mid \alpha_{\cE'}(a) = (\gamma_{\cE'} \circ \chi_{\cE'})(a) \}.
\]
Next, we have
\begin{align*}
&h^*(\preSpin_\cT|_{L\times_T T'})(X)= \\
& \{ a \in \Spin_{h^*(\cE_\qt|_{L\times_T T'})}(X) \mid h^*(\alpha_\cE|_{L\times_T T'})(a) = \big(h^*(\gamma_{\cE}|_{L\times_T T'})\circ h^*(\chi_\cE|_{L\times_T T'})\big)(a) \}.
\end{align*}
Since we are given the isomorphism $\psi \colon \cE'_\qt \iso h^*(\cE_\qt|_{L\times_T T'})$, we have the induced isomorphism between the Clifford algebras which in turn gives the functoriality of $\Spin$, so we have a group isomorphism
\begin{align*}
\Spin_{\cE'_\qt} &\iso \Spin_{h^*(\cE_\qt|_{L\times_T T'})} \\
a &\mapsto \Cl(\psi)(a).
\end{align*}
We must check that this isomorphism restricts to our desired isomorphism between $\preSpin$ groups. Since $\psi$ is coming from an isomorphism of trialitarian algebras, i.e., $(h,\psi)$ is a morphism in $\fTriAlg(T')$, we have that
\[
h^*(\alpha_{\cE}|_{L\times_T T'})\circ \Cl(\psi) = (\tau_L)_*(\pi_L^*(\psi))\circ \alpha_{\cE'}.
\]
By combining \Cref{autos_respect_vector_rep} and \Cref{gamma_respects_autos}, we have that
\[
h^*(\gamma_{\cE}|_{L\times_T T'})\circ h^*(\chi_\cE|_{L\times_T T'}) \circ \Cl(\psi) = (\tau_L)_*(\pi_L^*(\psi)) \circ \gamma_{\cE'} \circ \chi_{\cE'}.
\]
Therefore, if $a\in \preSpin_{\cT'}(X)$, then
\begin{align*}
\big(h^*(\gamma_{\cE}|_{L\times_T T'})\circ h^*(\chi_\cE|_{L\times_T T'})\circ \Cl(\psi)\big)(a) &= \big((\tau_L)_*(\pi_L^*(\psi)) \circ \gamma_{\cE'} \circ \chi_{\cE'}\big)(a) \\
&= \big((\tau_L)_*(\pi_L^*(\psi))\big)((\gamma_{\cE'} \circ \chi_{\cE'})(a)) \\
&= \big((\tau_L)_*(\pi_L^*(\psi))\big)(\alpha_{\cE'}(a)) \\
&= \big((\tau_L)_*(\pi_L^*(\psi))\circ \alpha_{\cE'}\big)(a) \\
&= \big(h^*(\alpha_{\cE}|_{L\times_T T'})\circ \Cl(\psi)\big)(a)
\end{align*}
which shows that $\Cl(\psi)(a) \in h^*(\preSpin_\cT|_{L\times_T T'})(X)$ as desired.

Finally, it is left to show that the pushforward of $h^*(\preSpin_\cT|_{L\times_T T'})$ along $f'\colon L' \to T'$ is equal to $g^*(\Spin_\cT)$. Given $Y \in \Sch_{T'}$, we have $g^*(\Spin_\cT)(Y) = \Spin_\cT(Y) = \preSpin_\cT(L\times_T Y)$. However, we have a diagram of schemes
\[
\begin{tikzcd}
L'\times_{T'} Y \ar[r] \ar[d,"\rotatebox{90}{$\sim$}"] & L' \ar[d,"h"] & \\
L\times_T Y \ar[r] \ar[d] & L\times_T {T'} \ar[r] \ar[d] & L \ar[d] \\
Y \ar[r] & T' \ar[r,"g"] & T
\end{tikzcd}
\]
and so
\begin{align*}
\preSpin_\cT(L\times_T Y) &= \preSpin_\cT(L'\times_{T'} Y) \\
&= h^*(\preSpin_\cT|_{L\times_T T'})(L'\times_{T'} Y) \\
&= f'_*\big(h^*(\preSpin_\cT|_{L\times_T T'})\big)(Y)
\end{align*}
showing that $g^*(\Spin_\cT) = f'_*\big(h^*(\preSpin_\cT|_{L\times_T T'})\big)$. Thus setting $\phi^\simpc = f'_*(\Cl(\psi))$ produces our desired isomorphism.

\noindent\ref{simpc_group_prelim_iii}: This follows from \ref{simpc_group_prelim_i} since any trialitarian algebra $(T,\cT)$ is locally isomorphic to $(T,\cT\spl|_T)$ because $\fTriAlg$ is a gerbe by \Cref{TriAlg_gerbe}.
\end{proof}

As a result of \Cref{simpc_group_prelim}, it is immediate that we have a morphism of stacks
\begin{align}
\cF^\simpc \colon \fTriAlg &\to \fD_4^\simpc \label{F^sc}\\
(T,\cT) &\mapsto (T,\Spin_\cT) \nonumber \\
(g,\phi) &\mapsto (g,\phi^\simpc). \nonumber
\end{align}

\begin{thm}\label{F^sc_equiv}
The morphism $\cF^\simpc$ above is an equivalence of gerbes.
\end{thm}
\begin{proof}
We argue using \Cref{equivalence_gerbes}. We focus on the object $(S,\cT\spl) \in \fTriAlg(S)$, which by \Cref{simpc_group_prelim}\ref{simpc_group_prelim_i} is mapped by $\cF^\simpc$ to $\Spin_8$. In particular, $\cF^\simpc(S,\cT\spl) = \Spin_{\cT\spl} \iso \Spin_8$ via the isomorphism
\begin{align*}
\Spin_{\cT\spl}(X) = \{(a,\theta^+(a),\theta^-(a))\in \Spin_8(X)^3 \} &\iso \Spin_8(X) \\
(a,\theta^+(a),\theta^-(a)) &\mapsto a.
\end{align*}

We know that $\bAut_{\fTriAlg}(S,\cT\spl) \cong \PGO_8^+ \rtimes \SS_3$ by \ref{TriAlg_automorphism_sheaf} and we also know that $\bAut_{\fD_4^\simpc}(\Spin_8) \cong \PGO_8^+ \rtimes \SS_3$ by \Cref{aut_spin_PGO}. Therefore, $\cF^\simpc$ gives us an induced group homomorphism
\[
\cF^\simpc_{(S,\cT\spl)} \colon \PGO_8^+ \rtimes \SS_3 \to \PGO_8^+ \rtimes \SS_3
\]
which we argue is an isomorphism. 

The automorphism group of $(S,\cT\spl|_T)$ takes the following form, as in \ref{TriAlg_automorphism_sheaf}. By definition $\cT\spl = (S^{\sqcup 3},(\Mat_8(\cO)_\qt,\Mat_8(\cO)_\qt,\Mat_8(\cO)_\qt),(\Phi,\Phi,\Phi))$. The elements $\phi \in \PGO_8^+(S) \subset (\PGO_8^+\rtimes \SS_3)(S)$ act on the algebra by
\[
(\phi,\theta^+(\phi),\theta^-(\phi))\colon (\Mat_8(\cO)_\qt,\Mat_8(\cO)_\qt,\Mat_8(\cO)_\qt)\iso (\Mat_8(\cO)_\qt,\Mat_8(\cO)_\qt,\Mat_8(\cO)_\qt).
\]
The elements of $\SS_3(S)$ permute the components of $S^{\sqcup 3}$. Let $\sigma \in \SS_3(S)$ be a diagonal section. If $\sigma \in A_3(S)\subset \SS_3(S)$, then it also cyclically permutes the three components of the algebra. However, if $\sigma \notin A_3(S)$, the it both permutes the factors of the algebra according to $\sigma$ and applies the automorphism $(\pi_\bO(\varphi),\pi_\bO(\varphi),\pi_\bO(\varphi))$ where $\pi_\bO(\varphi) \in \PGO_8(S)$ is the element of \Cref{theta_description}\ref{theta_description_iii}.

The induced action on the Clifford algebra is similar. A section $\phi \in \PGO_8^+(S)$ acts by $\big(\Cl(\phi),\Cl(\theta^+(\phi)),\Cl(\theta^-(\phi))\big)$ on $(\Cl(\Mat_8(\cO)_\qt),\Cl(\Mat_8(\cO)_\qt),\Cl(\Mat_8(\cO)_\qt))$. The diagonal sections of $A_3(S)$ simply permute the factors of the Clifford algebra while the other diagonal sections of $\SS_3(S)$ both permute the factors and apply $(\Cl(\pi_\bO(\varphi)),\Cl(\pi_\bO(\varphi)),\Cl(\pi_\bO(\varphi)))$. However, by the commutativity of \Cref{action_on_Clifford} and by \Cref{theta_description}\ref{theta_description_i}, we have that
\begin{equation}\label{F^sc_equiv_eq_i}
(\Cl(\pi_\bO(\varphi)),\Cl(\pi_\bO(\varphi)),\Cl(\pi_\bO(\varphi)))|_{(\Spin_8,\Spin_8,\Spin_8)} = (\theta,\theta,\theta).
\end{equation}

As a penultimate step, we point out that the commutativity of \Cref{action_on_Clifford} implies that for $\phi \in \PGO_8^+(S)$, if we restrict the induced automorphism $\Cl(\phi)$ of the Clifford algebra to $\Spin_8$, we obtain the usual action of $\phi$ as an inner automorphism of $\Spin_8$, i.e. viewing $\phi \in \PGO_8^+(S) \subseteq \bAut(\Spin_8)$. Indeed, locally there exist sections $x\in \Spin_8$ lying above $\phi$, and $\phi = \Inn(x)$ as an automorphism of $\Spin_8$. However, \Cref{action_on_Clifford} shows that $\Cl(\phi)=\Inn(x)$ as an inner automorphism of the Clifford algebra, and this justifies our claim.

Finally, we can describe $\cF^\simpc_{(S,\cT\spl)}$. A section $\phi \in \PGO_8^+(S)$ viewed as a subset of $\bAut_{\fTriAlg}(S,\cT\spl)(S)$ acts on the Clifford algebra by $\big(\Cl(\phi),\Cl(\theta^+(\phi)),\Cl(\theta^-(\phi))\big)$, which by the previous paragraph restricts to the automorphism $(\phi,\theta^+(\phi),\theta^-(\phi))$ on $(\Spin_8,\Spin_8,\Spin_8)$. Thus, taking the projection onto the first factor we see that $\cF^\simpc_{(S,\cT\spl)}(\phi) = \phi$. Now we consider diagonal sections of $\SS_3(S)$. For the cycle $c=(1\; 2\; 3) \in A_3(S) \subset \SS_3(S)$, it acts as a permutation on the Clifford algebra and thus also on $(\Spin_8,\Spin_8,\Spin_8)$, sending
\[
(a,\theta^+(a),\theta^-(a)) \mapsto (\theta^+(a),\theta^-(a),a).
\]
Taking the first projection, this shows that $\cF^\simpc_{(S,\cT\spl)}(c) = \theta^+$. Lastly, the transposition $\lambda = (2\; 3) \in \SS_3(S)$ acts on $(\Spin_8,\Spin_8,\Spin_8)$ by transposing the last two positions and applying the automorphism of \Cref{F^sc_equiv_eq_i}. It thus sends
\[
(a,\theta^+(a),\theta^-(a)) \mapsto \big(\theta(a),\theta(\theta^-(a)),\theta(\theta^+(a))\big).
\]
Considering the first factor, we obtain that $\cF^\simpc_{(S,\cT\spl)}(\lambda) = \theta$. In summary,
\begin{align*}
\cF^\simpc_{(S,\cT\spl)}(S) \colon (\PGO_8^+ \rtimes \SS_3)(S) &\to (\PGO_8^+ \rtimes \SS_3)(S) \\
\phi &\mapsto \phi \\
c &\mapsto \theta^+ \\
\lambda &\mapsto \theta
\end{align*}
for $\phi \in \PGO_8^+(S)$ and $c,\lambda$ diagonal as above. This discussion is independent of the base and therefore also describes how $\cF^\simpc_{(S,\cT\spl)}(T)$ acts over any $T$ when the section of $\SS_3(T)$ is diagonal. It is thus clear that $\cF^\simpc_{(S,\cT\spl)}$ restricts to an automorphism of the subpresheaf $\PGO_8^+ \rtimes \PP_3$ and therefore $\cF^\simpc_{(S,\cT\spl)}$ is an isomorphism itself. Hence, we are done.
\end{proof}

\begin{rem}\label{PGO_T_Spin_T_isogeny}
As one would expect, give a trialitarian algebra $(T,\cT)\in \fTriAlg$, there is an isogeny of groups $\Spin_\cT \surj \PGO^+_\cT$. Let $\cT = (L,\cE_\qt,\alpha_\cE)$ and let $X\in \Sch_T$. We know by definition that $\Spin_\cT(X) \subset \Spin_{\cE_\qt}(L\times_T X)$. Any element of $\PGO^+_\cT(X)$ is an automorphism of $(L\times_T X, \cE_\qt|_{L\times_T X}, \alpha_\cE|_{L\times_T X})$ which fixes $L\times_T X$ and can therefore be viewed as simply a quadratic triple isomorphism of $\cE_\qt|_{L\times_T X}$, i.e., $\PGO^+_\cT(X) \subset \PGO_{\cE_\qt}(X)$. In fact, since $\phi$ respects $\alpha_\cE$, we have $\PGO^+_\cT(X) \subset \PGO^+_{\cE_\qt}(X)$. We can then consider the standard isogeny
\begin{align*}
\Spin_{\cE_\qt} &\surj \PGO^+_{\cE_\qt} \\
a &\mapsto \Inn(\chi_{\cE}(a)).
\end{align*}
Let $a\in \Spin_\cT$ and set $\phi=\Inn(\chi_{\cE}(a))$. By the commutativity of \Cref{action_on_Clifford}, we know that $\Cl(\phi) = \Inn(a)$ as an automorphism of $\Cl(\cE_\qt)$. Locally, $\chi_{\cE}(a)$ is an element of the form $(a_1,a_2,a_3)\in (\Mat_8(\cO),\Mat_8(\cO),\Mat_8(\cO))$ and one can easily compute that $(\tau_L)_*\big(\pi_L^*\big(\Inn(\chi_{\cE}(a))\big)\big) = \Inn\big(\gamma_\cE(\chi_\cE(a))\big)$ holds locally, and thus globally as well. However, by assumption $(\gamma_\cE\circ \chi_\cE)(a) = \alpha_\cE(a)$ and therefore the diagram
\[
\begin{tikzcd}
\Cl(\cE_\qt) \ar[r,"\alpha_\cE"] \ar[d,"\Cl(\phi)=\Inn(a)"] & (\tau_L)_*(\pi_L^*(\cE_\qt)) \ar[d,"(\tau_L)_*(\pi_L^*(\phi))=\Inn(\alpha_\cE(a))"] \\
\Cl(\cE_\qt) \ar[r,"\alpha_\cE"] & (\tau_L)_*(\pi_L^*(\cE_\qt))
\end{tikzcd}
\]
commutes, meaning that $\phi \in \PGO^+_\cT$. For the split trialitarian algebra $(S,\cT\spl)$, we recover the usual isogeny $\Spin_{\cT\spl} \cong \Spin_8 \surj \PGO^+_8 \cong \PGO^+_{\cT\spl}$. Since this is the description locally, we see that in general $\Spin_\cT \surj \PGO^+_\cT$ is surjective and the kernel is the center. Because of this, an automorphism of $\Spin_\cT$ induces an automorphism of $\PGO^+_\cT$ and thus we get a morphism of stacks
\begin{align*}
\fD_4^\simpc &\to \fD_4^\adj \\
(T,\Spin_{\cT}) &\mapsto (T,\PGO^+_\cT)
\end{align*}
which makes the diagram
\[
\begin{tikzcd}[row sep=-1ex]
 & \fD_4^\simpc \ar[dd] \\
\fTriAlg \ar[ur,near end,"\cF^\simpc"] \ar[dr,near end,swap,"\cF^\adj"] & \\
 & \fD_4^\adj
\end{tikzcd}
\]
commute up to equality.
\end{rem}

\end{document}